\documentclass[reqno,11pt]{amsart}
\usepackage{amsmath, latexsym, amsfonts, amssymb, amsthm, amscd}
\usepackage{graphicx}
\usepackage[utf8]{inputenc}
\usepackage{stmaryrd}
\usepackage{mathabx}
\usepackage{hyperref}
\usepackage{xcolor}
\usepackage{etoolbox}
\usepackage{cleveref}

\usepackage{enumitem}
\setlist{itemsep=1pt,topsep=3pt,parsep=1pt,leftmargin=\parindent,itemindent=\parindent}
\setlist[itemize]{itemsep=1pt,topsep=2pt,parsep=1pt,wide}

\usepackage[left=2.2cm,right=2.2cm,top=2.2cm,bottom=2.2cm]{geometry}
\renewcommand{\bar}{\overline}

\newcommand{\lint}{\llbracket}
\newcommand{\rint}{\rrbracket}

\numberwithin{equation}{section}

\newtheorem{theorem}{Theorem}[section]

\newtheorem{lemma}[theorem]{Lemma}
\newtheorem{proposition}[theorem]{Proposition}

\newtheorem{rem}[theorem]{Remark}

\newcommand{\dd}{\mathrm{d}}
\newcommand{\ind}{\mathbf{1}}

\newcommand{\Z}{\mathbb{Z}}

\renewcommand{\P}{\mathrm{P}}
\newcommand{\E}{\mathrm{E}}
\renewcommand{\tilde}{\widetilde}
\renewcommand{\hat}{\widehat}

\newcommand{\bt}{\mathbf{t}}

\newcommand{\cG}{{\ensuremath{\mathcal G}} }

\newcommand{\cX}{{\ensuremath{\mathcal X}} }

\newcommand{\cJ}{{\ensuremath{\mathcal J}} }

\newcommand{\cN}{{\ensuremath{\mathcal N}} }
\newcommand{\cL}{{\ensuremath{\mathcal L}} }
\newcommand{\cT}{{\ensuremath{\mathcal T}} }

\newcommand{\cS}{{\ensuremath{\mathcal S}} }

\newcommand{\bP}{{\ensuremath{\mathbf P}} }
\newcommand{\bQ}{{\ensuremath{\mathbf Q}} }
\newcommand{\bE}{{\ensuremath{\mathbf E}} }

\DeclareMathSymbol{\leqslant}{\mathalpha}{AMSa}{"36} % nicer `smaller or equal'
\DeclareMathSymbol{\geqslant}{\mathalpha}{AMSa}{"3E} % nicer `larger or equal'
\DeclareMathSymbol{\eset}{\mathalpha}{AMSb}{"3F}     % nicer `emptyset'
\renewcommand{\leq}{\;\leqslant\;}                   % re-def. of < or \renewcommand{\geq}{\;\geqslant\;}                   % re-def. of > or \newcommand{\dd}{\,\text{\rm d}}             % a straight d for differentials

\newcommand{\bbE}{{\ensuremath{\mathbb E}} }

\newcommand{\bbN}{{\ensuremath{\mathbb N}} }

\newcommand{\bbP}{{\ensuremath{\mathbb P}} }

\newcommand{\bbR}{{\ensuremath{\mathbb R}} }

\newcommand{\bbZ}{{\ensuremath{\mathbb Z}} }

\newcommand{\gb}{\beta}
\newcommand{\gep}{\varepsilon}       % \ge already exists...

\newcommand{\go}{\omega}

\newcommand{\gl}{\lambda}

\makeatletter
\def\captionfont@{\footnotesize}
\def\captionheadfont@{\scshape}

\long\def\@makecaption#1#2{%
  \vspace{2mm}
  \setbox\@tempboxa\vbox{\color@setgroup
    \advance\hsize-6pc\noindent
    \captionfont@\captionheadfont@#1\@xp\@ifnotempty\@xp
        {\@cdr#2\@nil}{.\captionfont@\upshape\enspace#2}%
    \unskip\kern-6pc\par
    \global\setbox\@ne\lastbox\color@endgroup}%
  \ifhbox\@ne % the normal case
    \setbox\@ne\hbox{\unhbox\@ne\unskip\unskip\unpenalty\unkern}%
  \fi
  \ifdim\wd\@tempboxa=\z@ % this means caption will fit on one line
    \setbox\@ne\hbox to\columnwidth{\hss\kern-6pc\box\@ne\hss}%
  \else % tempboxa contained more than one line
    \setbox\@ne\vbox{\unvbox\@tempboxa\parskip\z@skip
        \noindent\unhbox\@ne\advance\hsize-6pc\par}%
\fi
  \ifnum\@tempcnta<64 % if the float IS a figure...
    \addvspace\abovecaptionskip
    \moveright 3pc\box\@ne
  \else % if the float IS NOT a figure...
    \moveright 3pc\box\@ne
    \nobreak
    \vskip\belowcaptionskip
  \fi
\relax
}
\makeatother
\def\writefig#1 #2 #3 {\rlap{\kern #1 truecm
\raise #2 truecm \hbox{#3}}}

\newcommand{\tf}{\mathtt{F}}
\newcommand{\tg}{\mathfrak{f}}

\newcommand{\cons}{\texttt{c}}

\title[Random Walk Pinning Model I: lower bounds and disorder irrelevance]{The Random Walk Pinning Model  I:\\ 
Lower bounds on the free energy and disorder irrelevance}
\author{Quentin Berger}
\address{Université Sorbonne Paris Nord, Laboratoire d'Analyse, Géométrie et Applications, CNRS UMR 7539, 99 Av. J-B Clément, 93430 Villetaneuse, France and Institut Universitaire de France}
\email{quentin.berger@math.univ-paris13.fr}
\author{Hubert Lacoin}
\address{IMPA, Estrada Dona Castorina, 110,
Rio de Janeiro, Brazil}
\email{lacoin@impa.br}

\date{\today}

\begin{document}

\begin{abstract}
The Random Walk Pinning Model (RWPM) is a statistical mechanics model in which the trajectory of a continuous time random walk $X=(X_t)_{t\geq 0}$ is rewarded according to the time it spends together with a moving catalyst whose position is given by $Y=(Y_t)_{t\geq 0}$.
More specifically for a system of size $T$, the law of~$X$ is tilted by the Gibbs factor $\exp(\beta \int_0^T \mathbf{1}_{\{X_t=Y_t\}} \mathrm{d}t)$, where $\beta \geq 0$ is a parameter modeling the strength of the interaction.
The position $(Y_t)_{t\ge 0}$ is given by the \textit{quenched} trajectory of a second continuous-time random walk, with the same distribution as $X$ but a different jump rate~$\rho \geq 0$. 
The parameter $\rho$ is interpreted as the \textit{disorder intensity}.
For fixed $\rho\ge 0$, the RWPM undergoes a localization phase transition when $\beta$ passes a critical value $\beta_c(\rho)$.
We thoroughly investigate whether a disorder of arbitrarily small intensity affects the features of the phase transition of the RWPM, known as the question of disorder relevance.
We interpret our results in the light of a prediction made in the~70's in theoretical physics: Harris' criterion.
We focus our analysis on the case of transient \textit{$\gamma$-stable walks} on~$\mathbb{Z}$, \textit{i.e.}\ random walks in the domain of attraction of a $\gamma$-stable law, with $\gamma\in (0,1)$. 
In the present paper, we derive lower bounds for the free energy, which results in either a proof of disorder irrelevance when $\gamma \in(\frac23,1)$ or upper bounds on the critical point shift when \(\gamma \in (0,\frac23]\).
More precisely, when $\gamma \in(\frac23,1)$, our estimates imply that that $\beta_c(\rho)=\beta_c(0)$ and $\rho$ is small, showing disorder irrelevance.
When $\gamma\in (0,\frac23]$ our companion paper~\cite{BLrel} shows that $\beta_c(\rho)>\beta_c(0)$ for every $\rho>0$, showing disorder relevance: we derive here upper bounds on the critical point shift $\beta_c(\rho)-\beta_c(0)$, which are matching the lower bounds obtained in~\cite{BLrel}.
For good measure, our analysis also includes the case of the simple random walk of $\bbZ^d$ (for $d\ge 3$) for which no upper bound on the critical point shift was previously known.
\end{abstract}

\maketitle

% \setcounter{tocdepth}{1}
% \tableofcontents

\section{Introduction}

The continuous time Random Walk Pinning Model (RWPM) has been introduced in \cite{BS10}. It is  a random walk \((X_t)_{t\geq 0}\) rewarded according to the time it spends together with another \textit{quenched} random walk trajectory \((Y_t)_{t\geq 0}\). 
The partition function of the model -- which is the main object of study in this paper -- corresponds to the Parabolic Anderson Model with a moving catalyst, studied earlier in~\cite{GH06}.
Let us also mention that the discrete time version of the model had already appeared in~\cite{Bir04}, as a partial annealing of the partition function of the directed polymer model under the size biased law, see also \cite{BGdH11,GdH07} for the continuous time setup.
The RWPM model has since then been studied for its own sake in \cite{BL11,BT10,BS10,BS11}, with the scope of understanding how disorder affects its phase transition. 
We refer to \cite[\S3.2]{Zyg24} for an overview of the model as well as of its connection with the directed polymer in the random environment.

\subsection{The model}

Given $J: \bbZ^d \to \bbR_+$ a \textit{symmetric} function on $\bbZ^d$ such that $\sum_{x\in \bbZ^d} J(x) =1$ we let $W=(W_t)_{t\geq 0}$ be a continuous-time random walk on $\mathbb Z^d$ with transition kernel~$J$.
In other words, $W$ is a continuous time Markov chain with generator $\cL$ given by
\begin{equation}
\mathcal L f(x)= \sum_{y\in\bbZ^d} J(y)\left(f(x+y)-f(y)\right) \,.
\end{equation}
We let $\P$ denote the distribution of $W$.
In this article, we consider more specifically two families of kernel $J$: \textit{(i)} the simple random walk on $\bbZ^d$, for $d\geq 1$; \textit{(ii)} one-dimensional random walks in the domain of attraction of a $\gamma$-stable process with $\gamma \in (0,2)$.
(We will soon restrict our attention to the transient cases \(d\geq 3\) or \(\gamma\in (0,1)\).)
These are given by the respective expressions
\begin{align}
  \tag{SRW}
  \label{SRW}
  J(x)&= \frac{1}{2d} \; \ind_{\{|x|=1\}} \,, \qquad x\in \bbZ^d \,, \\
  J(x)&=\varphi(|x|)(1+|x|)^{-(1+\gamma)} \,,\qquad x\in \bbZ \,,
  \tag{$\gamma$-stable}
  \label{JPP}
\end{align}
where $\varphi$ is a positive slowly varying function, \textit{i.e.}\ such that \(\lim_{x\to\infty}\varphi(cx)/\varphi(x) =1\) for any \(c>0\), see \cite{BGT89}.
Let us note that in both cases, by the local central limit theorem, the return probability to zero satisfies
\begin{equation}
\label{eq:limitpn1}
\lim_{t\to\infty} \frac{\log \P(W_{t}=0)}{\log t}=  -a \,,
\end{equation}
with $a=\frac{d}{2}$ for~\eqref{SRW} and $a=\frac{1}{\gamma}$ for~\eqref{JPP}.
For technical convenience, when considering $\gamma$-stable walks, we assume that the function $J$ is unimodal, that is $J$ is a non-increasing function of~$|x|$.

\begin{rem}
The value of $J(0)$ is not relevant for the definition of $\cL$.
Our assumption that $J$ is unimodal however implies that $J(0)>0$. The ``true'' jump rate of~$W$ is thus $\sum_{x\ne 0} J(x) <1$.
\end{rem}

Given $\rho \in [0,1)$, we consider  $X,Y$  two \emph{independent} continuous-time random walks which are slowed-down versions of $W$, with respective jump rates $(1-\rho)$ and $\rho$. In other words 
\[
X_t = W^{(1)}_{(1-\rho)t} \quad \text{ and }\quad Y_t = W^{(2)}_{\rho t} , 
\]
where $W^{(1)},W^{(2)}$ are two independent copies of $W$.
Since~$X$ and~$Y$ play different roles, we use different letters to denote their distribution: we let $\bP_{1-\rho}$ (or simply $\bP$) denote the law of~$X$ and $\bbP_{\rho}$ (or simply~$\bbP$) the law of~$Y$. 
% In most instances the superscripts $\rho$ and $(1-\rho)$ will not appear to lighten the notation.
Given $T>0$ (the polymer length) and a fixed realization of $Y$ (\emph{quenched disorder}) we define an energy functional on the set of trajectories by setting
\[
H^Y_T(X):=\int^T_0 \ind_{\{X_t=Y_t\}} \dd t \,.
\]
Then, given $\gb>0$ (the inverse temperature), the Random Walk Pinning Model (RWPM) is defined as the probability distribution $\bP_{\gb,T}^Y$ which is absolutely continuous with respect to $\bP$, with Radon-Nikodym density given by
\begin{equation}
\label{def:gibbs}
\frac{\dd \bP^{Y}_{\gb,T}}{ \dd \bP} (X)  := \frac{1}{Z_{\gb,T}^{Y}}\, e^{\beta H^Y_T(X)} \,, \quad  \text{ where } \quad 
Z_{\gb,T}^{Y} := \bE\Big[ e^{\beta H^Y_T(X)} \Big] \,.
\end{equation}
The renormalization factor $Z_{\gb,T}^{Y}$ makes $\bP_{\gb,T}^Y$ a probability measure and is referred to as the (\textit{free boundary}) partition function of the model.
When compared with $\bP$, the measure $\bP^{Y}_{\gb,T}$ favors trajectories $(X_t)_{t\ge 0}$ which overlap with $Y$ within the time interval $[0,T]$.
For technical reasons, we also define a constrained version of the model, setting $\bP^{Y,\cons}_{\beta,T}:= \bP^{Y}_{\gb,T}(\, \cdot  \mid X_T=Y_T)$. 
The corresponding \textit{constrained boundary} partition function is
\begin{equation}\label{compart}
  Z^{Y,\cons}_{\beta,T}:= \beta \bE\left[ e^{\beta  H^Y_T(X)}  \ind_{\{X_T=Y_T\}}\right] \,,
\end{equation}
where the factor $\beta$ in front of $\bE[\cdot]$ is introduced for practical reasons.

\subsection{Free energy and phase transition}

We are interested in the behavior of $(X_t)_{t\in [0,T]}$ under $\bP^{Y}_{\gb,T}$ and in particular in the typical value of the overlap $H^Y_T(X)$.
This behavior can be extracted from the study of the (quenched) free energy of the model, the definition of which we now introduce (together with a summarized proof of its existence).

\begin{proposition}
\label{freeenergy}
If either \eqref{SRW} or \eqref{JPP} holds,
 the \emph{quenched free energy}, defined by
\[
\tf(\rho,\beta):= \lim_{T\to \infty} \frac{1}{T} \log  Z^{Y}_{\beta,T} = \lim_{T\to \infty} \frac{1}{T}\bbE\left[ \log  Z^{Y}_{\beta,T}\right],\]
exists for every $\rho\in [0,1)$ and $\beta>0$ and the convergence holds $\bbP$-almost surely and in $L^1(\bbP)$.
The same limit is obtained when one replaces $Z^{Y}_{\beta,T} $ with $Z^{Y,\cons}_{\beta,T}$. It satisfies the following properties:
\begin{enumerate}
\item \label{free:i}
For every $\beta\geq 0$ and $\rho\in [0,1)$, $\tf(\rho,\beta)\ge 0$;
\item \label{free:ii}
The function $\beta\mapsto \tf(\rho,\beta)$ is non-decreasing and convex;
\item \label{free:iii}
The function $\rho\mapsto \tf(\rho,\beta)$ is non-increasing.
\end{enumerate}

\noindent
We can then define the critical point 
\[
\beta_c(\rho):= \inf\big\{ \gb > 0 \,:\,  \tf(\rho,\gb) >0 \big\} \,,
\]
and we have the following property:
\begin{enumerate}[resume]
\item \label{free:iv}
The function $\rho\mapsto \beta_c(\rho)$ is non-decreasing.
\end{enumerate}
\end{proposition}

\begin{proof}
The existence of the free energy in the case of the simple random walk has been established in \cite[Theorem~1.1]{BS10}, together with property~\ref{free:i}. 
For the sake of completeness we provided a proof for the $\gamma$-stable  walks in Appendix~\ref{app:freeen}.
It can be checked by hand that the property~\ref{free:ii} is satisfied by $\log  Z^{Y}_{\beta,T}$ hence by the limit. Property~\ref{free:iii} has been proved with a very general argument in \cite[Appendix~C]{BS11}, and property~\ref{free:iv} is a consequence of~\ref{free:iii}.
\end{proof}

Since we have $\bE^{Y}_{\gb,T} \big[ e^{u H^{Y}_{T}(X)} \big]= Z^{Y}_{\beta+u,T}\slash Z^{Y}_{\beta,T}$
the function $\tf(\rho,\cdot)$ encodes the asymptotic behavior of the Laplace transform of $H^{Y}_{T}(X)$ under $\bP^{Y}_{\gb,T}$.
Using standard large deviation techniques we obtain that for every $\beta\in \bbR$ and $\gep>0$ there exists $c(\gep,\beta)>0$ and $T_0(Y)<\infty$ such that for all $T>T_0$
\begin{equation}
  \label{eq:qontact}
\bP^{Y}_{\gb,T} \left(   \frac{1}{T} H^Y_{\beta,T}(X) \notin \left[\partial^-_{\beta}\tf(\rho,\beta)-\gep ,\partial^+_{\beta} \tf(\rho,\beta)+\gep \right]  \right) \le e^{-c(\gep,\beta)\, T} \,,
\end{equation}
where $\partial^{\pm}_\beta$ denotes the left or right derivative. The estimate is also valid under $\bP^{Y,\cons}_{\gb,T}$.
These derivative are always defined by convexity and $\partial^-_{\beta}\tf(\rho,\beta)>0$ for $\beta>\beta_c(\rho)$ and $\partial^+_{\beta}\tf(\rho,\beta)=0$ for $\beta<\beta_c(\rho)$.
Note also that \eqref{smooth} below implies that when \(\rho>0\), at \(\beta= \beta_c(\rho)\) we have $\partial^+_{\beta}\tf(\rho,\beta_c)=0$.

The critical point $\gb_c(\rho)$ therefore marks the transition between a \emph{delocalized phase} for $\gb<\gb_c(\rho)$, where $H^{Y}_{T}(X)$ is typically much smaller than~$T$, and a \emph{localized phase} for $\gb>\gb_c(\rho)$, where $H^{Y}_{T}(X)$ is typically of order $T$.

\subsection{The homogeneous (and annealed) model}

The case $\rho=0$ corresponds to the continuous time homogeneous pinning model on a defect line, which has been well studied and understood (we refer to \cite[Ch.~2]{Gia07} for an overview of the model in the discrete time setting, the continuous time setup being very similar).
In that case the free and constrained partition functions are equal to
\begin{equation}
  \label{annealedpartz}
z_{\beta,T}:=\E\left[ e^{\beta \int^T_0  \ind_{\{W_s=0\}} \dd s}\right] \quad \text{ and } \quad z^{\cons}_{\beta,T}:=\beta \, \E\left[ e^{\beta \int^T_0  \ind_{\{W_s=0\}} \dd s}\ind_{\{W_s=0\}}\right].
\end{equation}
Note that $z_{\beta,T}$  also corresponds, for any $\rho\in[0,1)$, to the \textsl{annealed} partition function $\bbE\big[ Z_{T,\gb}^{Y} \big]$, since
\begin{equation}
\bbE\left[Z^{Y}_{\beta,T} \right] = \bbE_{\rho} \otimes \bE_{1-\rho}\left[ e^{\beta \int^T_0  \ind_{\{X_s=Y_s\}} \dd s} \right]
\end{equation}
and $X-Y \stackrel{(d)}{=} W$ (recall we assumed that $J$ is symmetric and the jump rates of~$X$ and~$Y$ add up to~$1$).

\begin{rem}
In \cite{BS10,BS11}, the authors consider the case where the random walk~$X$ has jump rate~$1$ and~$Y$ has jump rate~$\varrho$.
Our setting is in fact recovered by letting $\rho = \frac{\varrho}{1+\varrho}$ and changing the time horizon by a factor~$1+\varrho$. 
Our choice was dictated by the commodity of having no dependence in $\rho$ in the annealed partition function.
\end{rem}

\noindent The value of the \textit{homogeneous} (or \textit{annealed}) free energy, defined as
\begin{equation*}
  % \label{annealed}
\tf(\beta) = \tf(0,\beta) := \lim_{T\to \infty} \frac{1}{T} \log z_{T,\beta}=  \lim_{T\to \infty} \frac{1}{T} \log z^{\cons}_{T,\beta}\,,
\end{equation*}
can be computed with the help of basic renewal theory, see \cite[Ch.~XI]{FellerII}.
Setting 
\begin{equation}
\label{def:beta0}
  \beta_0:= \left(\int^{\infty}_0 \P(W_t=0) \dd s\right)^{-1}\,,
\end{equation}
one can show that $\tf(\gb)$ is either the unique solution of the following implicit equation in $b$
\begin{equation}
\label{eq:freeenergy}
\beta \int_0^{+\infty} e^{- b t} \P(W_t=0) \dd t =  1 
\end{equation}
if such a solution exists, or $\tf(\beta)=0$ if \eqref{eq:freeenergy} does not admit any solution.
Note that the solution of \eqref{eq:freeenergy} exists if and only if $\beta\in [\beta_0,\infty)\setminus \{0\}$ and it is positive if $\beta>\beta_0$.
Recalling the definition of the critical point in Proposition \ref{freeenergy} we have  $\gb_c(0)=\gb_0$. Furthermore $\tf(\beta)$ is analytic on $\bbR\setminus\{0\}$. 

Combining \eqref{eq:freeenergy} with \eqref{eq:limitpn1} (or more precisely, with \eqref{K-reg} below) and some Tauberian theory, it possible to identify the critical behavior of the free energy. We refer to \cite[Thm.~2.1]{Gia07} for the analogous result (and its proof) in the discrete setup.

\begin{proposition}\label{homener}
The homogeneous free energy has the following critical behavior:
\begin{equation}
\label{eq:energyasymp}
\lim_{\beta \downarrow \beta_0}\frac{\log\tf(\gb)}{\log (\beta-\beta_0)} = \nu  \qquad \text{ with } \quad  \nu := 1\vee \frac{1}{|a-1|}\in [1,\infty] \,,
\end{equation}
where \(a\) is from~\eqref{eq:limitpn1}.
When $\gamma\ne 1$ or $d\ne 2$, $\nu<\infty$ and  we have more precisely
\begin{equation}
\tf(\gb) \stackrel{\beta\downarrow \beta_0}{\sim}  (\gb-\gb_0)^{\nu} \hat L\left(\frac{1}{\gb-\gb_0}\right) ,
\end{equation}
 where $\hat L$ is some (explicit) slowly varying function.  The function $\hat L$ can be replaced by a constant when either $\varphi$ is asymptotically constant and $\gamma\ne 1/2$ ($\gamma$-stable case) and when $d\ne 4$ (SRW case).
\end{proposition}

\noindent
The above covers both the case where $W$ is recurrent (in which case $a\in(0,1]$ and $\beta_0=0$)
and where~$W$ is transient ($a\in[1,\infty)$ and $\beta_0>0$).
In particular:
\begin{itemize}
  \item In the case~\eqref{SRW}, we have $\nu = 2$ in dimension $d=1$ ($\gb_0=0$) and $d=3$ ($\gb_0>0$), $\nu=+\infty$ in dimension $d=2$ ($\gb_0=0$), and $\nu =1$ in dimension $d\geq 4$ ($\gb_0 >0$);
  \item In the case~\eqref{JPP}, we have $\nu=\frac{\gamma}{\gamma-1}$ if $\gamma \in (1,2)$ ($\gb_0=0$), $\nu=+\infty$ if $\gamma=1$ ($\gb_0=0$ or $\gb_0>0$), and $\nu = 1\vee \frac{\gamma}{1-\gamma}$ if $\gamma \in (0,1)$ ($\gb_0>0$).
\end{itemize} 

\subsection{The question of disorder relevance}

The quenched and annealed free energies can be compared via Jensen's inequality:
\[
\tf(\rho,\beta) = \lim_{T\to \infty} \frac{1}{T}\bbE\left[ \log  Z^{Y}_{\beta,T}\right] \leq \lim_{T\to \infty} \frac{1}{T} \log  \bbE[Z^{Y}_{\beta,T}]  =\tf(\gb) \,.
\]
In particular, the quenched and annealed critical points verify $\gb_c(\rho) \geq \gb_0 = \gb_c(0)$ (this also follows from property~\ref{free:iv} of \Cref{freeenergy}).
There are therefore two natural questions when considering the influence of disorder on the model:
\begin{enumerate}
\item [(1)] Is the inequality $\gb_c(\rho) \ge \gb_0$ strict? In other words, does disorder \emph{shift} the critical point?
\item [(2)] Is the free energy critical exponent preserved when disorder is introduced. More precisely, do we have
$\tf(\rho,\gb) =  (\gb-\gb_c(\rho))^{\nu+o(1)}$ as $\gb\downarrow \gb_c(\rho)$ as in~\eqref{eq:energyasymp}?
\end{enumerate}
In this context, the jump rate parameter~$\rho$ can be interpreted as the ``intensity'' of the disorder and we can refine the above questions to whether the critical point and critical exponents are affected by a disorder of \textit{arbitrarily small} intensity.

\smallskip

The case of a recurrent random walk $W$ has been analyzed in the SRW case in dimension $d=1,2$ and \cite[Theorem 1.2]{BS10} shows that both the critical point and the critical exponents are conserved. 
The approach used in \cite{BS10} is specific to recurrent walks and seems sufficiently flexible to extend to the case of any recurrent walk that satisfies~\eqref{eq:limitpn1}.
In fact, the proof of~\cite[Thm.~2.1-(2.1)]{AB18a} should give that we also have \(\beta_c(\rho) = \beta_0\) even for a transient walk with \(a=1\) in~\eqref{eq:limitpn1}.
The present work is therefore concerned with the analysis of disorder relevance for the RWPM in the case where~$W$ is transient (with \(a>1\) in~\eqref{eq:limitpn1}), in particular:
\begin{itemize}
  \item In the case~\eqref{SRW},  we focus on dimensions $d\geq 3$;
  \item In the case~\eqref{JPP},  we focus on parameters $\gamma\in (0,1)$.
\end{itemize}
% We let aside the case $\gamma=1$ which is a bit singular, but we may expect that in this case $\beta_c(\rho)=\beta(\rho)$ \textit{for every $\rho\ge 0$} -- we refer the reader to  \cite{AB18a,AZ10}  for results with this flavor proved for related models.

% W let us shortly discuss the case recurrent random walks
% through the perspective of the following result
%
% \begin{theorem}\cite[]{BS10}
%  In dimension $d=1,2$, then $\gb_c(\rho) = \gb_0$ for any $\rho \in (0,1)$. Furthermore for every $\rho\in (0,1)$
%  \begin{equation}
%  \lim_{\beta\to 0+}\log \frac{\tf(\rho,\beta)}{\tf(\beta)}=\begin{cases} 2   \quad & \text{ when } d=1,\\
%  \infty  \quad & \text{ when } =2.
%                                                            \end{cases}
%  \end{equation}
% \end{theorem}

\subsubsection*{Harris criterion for disordered systems.}

In his influential paper on disordered system~\cite{Har74}, A.B.~Harris made a general prediction concerning the influence of disorder for statistical mechanics model which relies on the sign of the specific-heat exponent of the corresponding homogeneous model.
Applied to the RWPM pinning model under the assumption \eqref{eq:limitpn1}, the prediction reads as follows:
\begin{enumerate}
\item [(a)] If $\nu<2$, that is if $a >\frac 3 2$ (\textit{i.e.}\ $\gamma <\frac23$ or $d\geq 4$), then disorder is \textit{relevant}: the features of the phase transition of $\tf(\rho,\beta)$ with \textit{any} $\rho\in (0,1)$ should differ from that of the pure model.
\item [(b)] If $\nu >2$ that is if $a \in (1, \frac 3 2)$ (\textit{i.e.}\ $\gamma \in (\frac23,1)$), then disorder is \textit{irrelevant}: the effect of disorder on the phase transition should not be noticeable for small values of $\rho$.
\end{enumerate}
The case $\nu=2$ (\textit{i.e.}\ \(\gamma=\frac23\) or \(d=3\)) is called \emph{marginal} and Harris' criterion does make any prediction for the effect of disorder in that case, which should be model-dependent.
Let us mention that Harris' predictions have been conceived to study disorder relevance in contexts such as the random field Ising model~\cite{IM75} or the disordered pinning on a defect line~\cite{DHV92}, where the disorder is given by an i.i.d.\ field of random variables.
Our disorder, given by a realization of $Y$, displays some long-range correlations and it is thus not clear whether Harris' criterion should apply or not, see for instance \cite{WH83} for a generic heuristic study of disorder relevance in correlated environment.

\section{Main results}

For the sake of presenting a global picture, we introduce the results proved in this work together with those proved in the companion paper \cite{BLrel}.
\begin{enumerate}
  \item We prove disorder irrelevance when $\gamma \in (\frac23,1)$, showing that for small values of $\rho$, $\tf(\rho,\beta)$ and $\tf(\beta)$ have the same order of magnitude.
  \item We prove disorder relevance when $\gamma \in (0,\frac23]$ and we provide upper and lower bounds on the critical point shift $\beta_c(\rho)-\beta_c(0)$ in the small $\rho$ asymptotic. 
  When $\gamma\in (0, \frac 23)$ we obtain matching polynomial bounds, while for $\gamma=2/3$ we obtain (non-matching) subpolynomial bounds.
\end{enumerate}
% A companion paper \cite{BLrel} completes the picture by showing disorder relevance when \(\gamma \in (0,\frac23]\) and corresponding lower bounds on the critical point shift. In particular,~\cite{BLrel} shows that, in addition to being relevant for $\gamma\in(0,\frac23)$, disorder is also \textit{always} relevant in the marginal case $\gamma =\frac23$, for all choices of slowly varying function~$\varphi$ in~\eqref{JPP} -- this is in contrast with what happens in the disordered pinning model, see \Cref{sec:pinning} below.
These result are complemented with corresponding results for the SRW case, completing those that previously appeared in the literature in \cite{BL11,BT10,BS10,BS11}. 
We also compare in Section~\ref{sec:litterature} below our results with those obtained for the disordered pinning model  \cite{Ale08,AZ09,BCPSZ14,CTT17,CdH13b,DGLT09,GT06,GLT10,GLT11,Lac10ecp,Ton08a} to highlight a key difference between the two models in the marginal case (when $\nu=2$).
In both this paper and its companion one, in order to simplify the exposition, we focus on the case where  $\lim_{|x|\to \infty}\varphi(x)=1$ in~\eqref{JPP}, except when  $\gamma=\frac23$ where we consider arbitrary $\varphi$, as it is necessary to underline the key difference between the RWPM and the disordered pinning model.

\subsection{Irrelevance of disorder when \texorpdfstring{$\gamma\in (\frac23,1)$}{}}
\label{dfs}

Our first set of results deal with the case $\gamma\in(\frac{2}{3},1)$, in which case $\nu>2$.
Harris criterion thus predicts disorder irrelevance and we show that this prediction is indeed correct: for small values of $\rho$,  $\tf(\rho,\beta)$ and $\tf(\gb)$ are comparable around $\beta_0$. 

\begin{theorem}
\label{th:irrel}
If $J(x)\stackrel{|x|\to \infty}{\sim}  |x|^{-(1+\gamma)}$ with $\gamma\in (\frac{2}{3},1)$, we have $\beta_c(\rho)= \beta_0$ for $\rho$ sufficiently small.
 Furthermore, we have 

\begin{equation}\label{lauwer}
\lim_{\rho\downarrow 0} \liminf_{\gb \downarrow \gb_0} \frac{\tf(\rho,\beta)}{\tf(\gb)} =1.
\end{equation}
\end{theorem}
\noindent In \cite{BLrel}, we prove that, in spite of~\eqref{lauwer}, for fixed \(\rho>0\) we do not not have the asymptotic equivalence $\tf(\rho,\beta)\sim\tf(\gb)$ as \(\beta\downarrow 0\).
We also prove that the coincidence of critical points does not hold for every $\rho$.

\begin{proposition}[\cite{BLrel}, Propositions~1.5 and~1.6]
   If $J(x)\stackrel{|x|\to \infty}{\sim}  |x|^{-(1+\gamma)}$  with $\gamma\in (\frac{2}{3},1)$, then 
   \begin{enumerate}
    \item For every $\rho \in (0,1)$, we have $\limsup_{\gb \downarrow \gb_0} \frac{\tf(\rho,\beta)}{\tf(\gb)} <1$.
    \item There exists $\delta>0$ such that $\beta_c(\rho)>\beta_0$ for any $\rho >1-\delta$.
   \end{enumerate}

\end{proposition}

\subsection{Estimates for the critical point shift in the relevant regime \texorpdfstring{$\gamma \in (0,\frac23]$}{}}
\label{fds}

Our result for the cases  $\gamma \in (0,\frac23)$ and $\gamma =\frac23$ differ in nature so we present them separately (we also include the case of the simple random walk for good measure).

\begin{theorem}
  \label{th:rel}
  Assume that either $J(x)\stackrel{|x|\to \infty}{\sim}  |x|^{-(1+\gamma)}$ with $\gamma\in (0,\frac{2}{3})$ or that \(d\geq 4\) in the case~\eqref{SRW}.
  Then there exists $C>0$ (depending on $\gamma$ or $d$) such that for every $\rho\in (0,\frac12)$ we have
  \begin{equation}\label{lkl}
  \beta_c(\rho)-\beta_0\le C \rho^{\frac{1}{2-\nu}} \,,
  \end{equation}
  with \(\nu\) as in \eqref{eq:energyasymp}, \textit{i.e.}\ \(\nu = 1\vee \frac{\gamma}{1-\gamma}\) or \(\nu =1\) for the SRW in dimension \(d\geq 4\).
\end{theorem}

For the SRW in dimension $d\ge 5$ the bound \eqref{lkl} matches the lower bound proved in \cite{BS10} (see Section~\ref{sec:litterature} below).
As far as asymptotically $\gamma$-stable walks are concerned, we also show in~\cite{BLrel} that the upper bound~\eqref{lkl} is sharp when $\gamma\in(0,\frac{2}{3})\setminus\{\frac12\}$.

\begin{theorem}[\cite{BLrel}, Theorem~1.3]
 If $J(x)\stackrel{|x|\to \infty}{\sim}  |x|^{-(1+\gamma)}$  with $\gamma\in (0,\frac{2}{3})\setminus\{\frac12\}$ then there is a constant $c>0$ such that for any \(\rho\in (0,\frac12)\)
    \begin{equation}\label{lkl2}
  \beta_c(\rho)-\beta_0\ge c\, \rho^{\frac{1}{2-\nu}} \quad  \text{with}   \quad \nu= 1\vee \frac{\gamma}{1-\gamma}.
  \end{equation}
  % In particular, note that \( \frac{1}{2-\nu} = \frac{1-\gamma}{2-3\gamma}\vee 1\).
\end{theorem}

\begin{rem}
In the case $\gamma=\frac12$ (or $d=4$ for the SRW) we do not expect~\eqref{lkl} to be sharp.
One rather expects $\beta_c(\rho)-\beta_0$ to be of order $\frac{\rho}{\log (1/\rho)}$, just as for the disordered pinning model, see \cite{AZ09}. 
Proving this would require addressing some additional technical complications which we have chosen to leave aside.
\end{rem}

We now move to the marginal case $\gamma=\frac23$, allowing for a slowly varying function \(\varphi\) in~\eqref{JPP}.
In order to ease the exposition, we state here the result in a special case where $\varphi$ behaves like a power of~$\log$.
The result in its most general form is displayed in \Cref{prop:multiprop} below.
We also give a bound for the SRW in dimension $d=3$.

\begin{theorem}
\label{th:3DRW}
Assume that~\eqref{JPP} holds with \(\gamma =\frac23\) and with $\varphi(t) \stackrel{t\to \infty}{\sim}  (\log t)^{\kappa}$  for some $\kappa\in \bbR$.
Then there exists $C>0$ such that, for any $\rho\in (0,\frac12)$ 
\begin{equation}
  \label{soluce}
  \log (\beta_c(\rho)-\beta_0 ) \le -C \,
  \begin{cases}  
     \rho^{-\frac{1}{1+3\kappa}}  & \text{ if } \kappa>1/3 \,,\\ 
     \rho^{-\frac{1}{2}} \, \log^{\frac12}\big(\frac1\rho\big)   & \text{ if } \kappa=1/3 \,,\\
     \rho^{-\frac12} & \text{ if } \kappa<1/3 \,.
  \end{cases}
 \end{equation}
For the SRW in dimension $d=3$, there exists $C>0$ such that for any $\rho\in (0,\frac12)$ we have
 \begin{equation}
  \log (\beta_c(\rho)-\beta_0 ) \le - C \rho^{-\frac12} \,.
 \end{equation}
\end{theorem}
\noindent It was shown previously that $\log (\beta_c(\rho)-\beta_0 ) \ge -c \rho^{-(2+\gep)}$ for the $3$-dimensional simple random walk, see~\cite{BS11}. 
For the $\gamma$-stable case, the following result illustrates that the bound~\eqref{soluce} is close to being sharp.
It also highlights that $\beta_c(\rho)>\beta_0$ for all $\rho>0$, no matter what the slowly varying function~$\varphi$ is.

\begin{theorem}[\cite{BLrel}, Theorem~1.4]
Assume that~\eqref{JPP} holds with \(\gamma =\frac23\).
Then we have $\beta_c(\rho)>\beta_0$ for all $\rho\in (0,1)$.
Furthermore if $\varphi(t) \stackrel{t\to \infty}{\sim}  (\log t)^{\kappa}$ for some $\kappa\in \bbR$,
then there exists $c>0$ such that, for any $\rho\in (0,\frac12)$ 
\begin{equation}
  \label{soluce2}
  \log (\beta_c(\rho)-\beta_0 ) \ge -c \,
  \begin{cases}  
     \rho^{-\frac{1}{3\kappa}}  & \text{ if } \kappa>1/3 \,,\\ 
     \rho^{-1} \log\big(\frac1\rho\big)   & \text{ if } \kappa=1/3 \,,\\
     \rho^{-1} & \text{ if } \kappa<1/3 \,.
  \end{cases}
 \end{equation}
\end{theorem}

Overall, the above results confirm the validity of Harris' criterion for the RWPM. Furthermore it establishes that in the marginal case (which corresponds to $\gamma=2/3$), disorder is marginally relevant for every choice of slowly varying function $\varphi$.
To conclude this presentation of our main results, let us mention that the upper bound on $\beta_c(\rho)$ are obtained as consequences of lower bound obtained for the free energy $\tf(\rho, \beta)$. More precisely we prove that above the threshold  defined by~\eqref{lkl} and \eqref{soluce}, $\tf(\rho, \beta)$ and $\tf(\beta)$ are of the same order,
see Propositions \ref{jenseninq} and \ref{prop:multiprop} below.

\subsection{Comparing our results with the literature}
\label{sec:litterature}

We recall here, for context and discussion, the results that where previously obtained concerning disorder relevance for the RWPM and a closely related model, the disordered pinning model. 
We start with the latter.

\subsubsection*{The pinning model on a defect line.}
Let us now briefly review a model for which Harris criterion has fully shown its validity: \textit{the disordered pinning model on a defect line}. 
For the sake of making the connection with the RWPM more straightforward, we consider a continuous time version  of the model. 
(While all the result we cite have been proved for a discrete analogue, their proof can be adapted to the continuous setup.) 
We let $(W_t)_{t\ge 0}$ be a continuous time random walk satisfying \eqref{JPP} (the SRW case can also be considered but it would not add much to the discussion) and we let $(\go_t)_{t\ge 0}$ be a standard one-dimensional Brownian motion, playing the role of the disorder (for this reason we denote its law by $\bbP$).
Given $\rho\ge 0$ and $\beta\in \bbR$ we set
\[
H^{\rho,\go}_{\beta,T}(W):= \int^T_{0} \ind_{\{W_t=0\}}\left( \beta \dd t + \sqrt{\rho} \, \dd \go_t -\frac\rho2 \dd t\right) \,.
\]
The term $-\frac\rho2$ is present to cancel $\sqrt{\rho} \, \dd \go_t$ when taking the expectation w.r.t.\ $\go$, see~\eqref{eq:pinningannealed} below.
We consider the following Gibbs modification of the law \(\P\) of $W$, in analogy with~\eqref{def:gibbs}:
\begin{equation*}
\frac{\dd \P^{\rho,\go}_{\beta,T}}{\dd \P}(W):=  \frac{1}{ Z^{\rho,\go}_{\beta,T}} \, e^{ H^{\rho,\go}_{\beta,T}(W) }\quad \text{ with }  \quad    Z^{\rho,\go}_{\beta,T}:=   \E\left[ e^{ H^{\rho,\go}_{\beta,T}(W) }\right].
\end{equation*}
Note that this model has the same annealed partition function as the RWPM, namely
\begin{equation}
  \label{eq:pinningannealed}
\bbE[Z^{\rho,\go}_{\beta,T}]= z_{T,\beta} = \E\left[ e^{\beta \int^T_0  \ind_{\{W_s=0\}} \dd s}\right] \,.
\end{equation}
As for the RWPM, we define the free energy of this model by setting
\begin{equation*}
 \tg(\rho,\beta):=\lim_{T\to \infty} \frac{1}{T} \log Z^{\rho,\go}_{T,\beta}.
\end{equation*}
The function $\tg(\rho,\beta)$ verifies the same properties as $\tf(\rho,\beta)$ (see \Cref{freeenergy}) and we can define the critical point of the localization transition by setting
\begin{equation*}
\beta_c^{\rm pin}(\rho):=\inf \big\{ \beta \in \bbR  : \ \tg(\rho,\beta)>0 \big\}.
\end{equation*}
We also have that $\tg(\rho,\beta)\le \tf(\beta)$ and in particular $\beta_c^{\rm pin}(\rho) \geq \beta_0$.
The following results  have been proved concerning disorder relevance/irrelevance for the pinning model (in the analogue discrete time setting), regarding the phase diagram associated with $\tg(\rho,\beta)$:
\begin{enumerate}
 \item  For any $\rho>0$, there is a constant $c_{\rho}$ such that $\tg(\rho,\beta^{\rm pin}_c(\rho)+u) \leq c_\rho u^2$ for every $u\ge 0$. In particular the critical exponent associated to the free energy, if it exists, is always larger than~$2$ (and thus differs from $\nu$ when $\nu<2$), see \cite{CdH13b,GT06}.
 
 \item \label{ii-irrel}
  When $\nu>2$ (\textit{i.e.}\ $\gamma\in (\frac{2}{3},1)$) then we have disorder irrelevance. For small enough $\rho$ we have $\beta_c^{\rm pin}(\rho)=\beta_0$ and also $\tg(\rho,\gb) \stackrel {\beta \downarrow \beta_0} {\sim} \tf(\beta)$, see  \cite{Ale08,GT09,Lac10ecp,Ton08a}.

 \item Whenever $\nu<2$ (\textit{i.e.}\ $\gamma\in (0,\frac{2}{3})$) then we have disorder relevance: there is a shift of the critical point for every $\rho>0$, and $\beta_c^{\rm pin}(\rho)-\beta_0 = \rho^{\frac{1}{2-\nu}+o(1)}$ as $\rho\downarrow 0$; we refer to \cite{AZ09, DGLT09} (see also~\cite{BCPSZ14,CTT17} for a sharp estimate).
 
\item When $\nu=2$ (\textit{i.e.}\ $\gamma= \frac23$), whether or not we have a critical point shift for small $\rho$ depends on the detailed asymptotic behavior of $\P(W_t=0)$. 
More precisely, we have $\beta_c^{\rm pin}(\rho)=\beta_0$ for small values of~$\rho$ (disorder irrelevance) when
\begin{equation}
  \label{eq:criterion}
 \int^{\infty}_{1} \big(t^{2}\P(W_t=0)\big)^{-2}\dd t < \infty,
\end{equation}
and  $\beta_c^{\rm pin}(\rho)>\beta_0$ for every $\rho>0$ (disorder relevance) when the integral diverges. 
We refer to~\cite{Ale08,Lac10ecp,Ton08a} for the first part of the statement and to~\cite{BL18, GLT10,GLT11} for the second part. 
\end{enumerate}
The above results display striking similarities with the one displayed in Sections \ref{dfs} and \ref{fds}, but also a couple of important differences. 
In particular, the criterion \eqref{eq:criterion} establishes that for some choices of $\varphi$ disorder is marginally relevant while for some others disorder is marginally irrelevant, something that does not occur for the RWPM. 
Additionally, in the irrelevant regime, the equivalence between quenched and free energy displayed in \ref{ii-irrel} never holds for the RWPM.

\subsubsection*{Existing results for the Random Walk Pinning Model.}
As far as the RWPM is concerned, the question of disorder relevance has been considered only in the case~\eqref{SRW} where $W$ is the simple random walk on~$\bbZ^d$, either in continuous or discrete time, see~\cite{BL11,BT10,BS10,BS11}.
Let us collect existing results in the following theorem.

\begin{theorem}\label{oldth}
Let $W=(W_t)_{t\geq 0}$ be a simple random random walk in dimension $d\ge 3$.
Then, the following hold:
\begin{enumerate}
\item For any $\gep>0$, there exist constants $c,c_{\gep}$ such that for all $\rho\in(0,1)$
\begin{equation}
  \label{shiftos}
\gb_c(\rho) -\gb_0 \geq 
\begin{cases}
c \rho \quad & \text{ if } d\geq 5 \,,\\
c_{\gep} \rho^{1+\gep} \quad & \text{ if } d =4 \,, \\
\exp( - c_{\gep} / \rho^{2+\gep}) \quad & \text{ if } d =3\,.
\end{cases}
\end{equation}

\item For any $\rho \in (0,1)$, there is a constant $c_{\rho}>0$ such that, for any $u\ge 0$,
\begin{equation}
  \label{smooth}
  \tf(\rho,\beta_c(\rho)+u) \leq c_{\rho}\, u^2 \,.
\end{equation}
\end{enumerate}
\end{theorem}

\noindent The bound $(i)$ was established in \cite{BS10} (for $d\ge 4$) and \cite{BS11} for $d=3$, while $(ii)$ was proved in \cite{BL11}.
These results establish that disorder is relevant for every $d\ge 3$.

The proofs in \cite{BL11,BT10,BS10,BS11} all adapt ideas developed in \cite{DGLT09,GLT10,GT06} in the context of the disordered pinning model to show that, for the RWPM, disorder is \textit{at least as relevant} as for the disordered pinning on a defect line.
However this comparison only goes one way.
Indeed the argument used to show disorder irrelevance for the disorder pinning model in \cite{Ale08,Lac10ecp,Ton08a} are based on the fact that the second moment of the partition function is bounded for $\beta=\beta_0$ when
$\rho$ is sufficiently small. 
This property \textit{fails to hold} for the RWPM, for any choice of $\gamma$.

\smallskip

This last observation is our main motivation to study the RWPM in a setup where the annealed exponent $\nu$ in \eqref{eq:energyasymp} is (strictly) larger than $2$, for which Harris criterion predicts irrelevance. The $\gamma$-stable random walk represents the simplest such setup. 
Theorems~\ref{th:rel} and~\ref{th:3DRW}, while very similar in spirit to analogous results proved for the disordered pinning model, require original arguments. 
Some of the ideas developed in the study of this specific model, may find application in the study of other disordered model for which the disorder has a non i.i.d.\ structure.

Let us finally mention that our results for the RWPM in the simple random walk case answers, at least partially, a question appearing in \cite{Zyg24} (as Question 3): `` Can one find matching lower bounds for~\eqref{shiftos}''.

\subsection{Organization of the paper and brief outline of ideas}
Let us now give a brief overview of the remainder of the paper.

\begin{itemize}
  \item Section~\ref{sec:prelim} is thought as a toolbox, where we gather several technical results which will be of use in the rest of the paper. 
  An important tool is the rewriting of the partition function  of the RWPM  which allows to interpret it as  the partition function of a continuous time renewal process $\tau$ interacting with a random potential $w$.
  While none of these results are original we provide a proof for some of them for the commodity of the reader (some of the proofs are postponed to the appendix).

  \item In \Cref{sec:Jensen} we prove Theorem \ref{th:rel}  in the case \(\gamma \in (0,\frac12]\) or \(d\geq 4\), using a relatively soft argument based on Jensen's inequality.
  
  \item In \Cref{sec:lower} we state the lower bounds on the free energy needed to prove \Cref{th:irrel,th:rel,th:3DRW}, see \Cref{prop:multiprop}. 
  The first steps of the proof are also performed in this section, since they are common to all needed estimates.
  The main idea is to reduce to a finite-volume estimate thanks to an approximate factorization of the model:
  the needed finite-volume estimate is stated in \Cref{trickyy}, the goal then being reduced to showing that a ratio of second-to-first moment of the partition function is close to~\(1\), in the spirit of the Paley-Zygmund inequality. 

  \item \Cref{trixx} gives the second step of the proof, which is a general method to control the moment ratio of \Cref{trickyy}.
  Here, the idea is to use the weighted renewal representation of Section~\ref{sec:prelim} in order to obtain an upper bound on the moment ratio in terms of ``interval overlaps'' of two independent renewals; this is stated in~\Cref{topitop2}.

  \item The remaining sections complete the different proofs by controlling the ``interval overlaps'', using ad-hoc techniques in each case.
  \Cref{trixx1} concludes the  proof of \Cref{prop:multiprop} in the case \(\gamma \in (\frac12,\frac23)\) and \(\gamma=\frac23\) (or \(d=3\)).
  \Cref{trixx2} concludes the proof of \Cref{prop:multiprop} in the case \(\gamma\in(\frac23,1)\). 

  \item An appendix collects technical results: \Cref{app:freeen} proves the existence of the free energy; \Cref{app:homogeneous} proves an important estimates on the homogeneous pinning model; \Cref{app:LLT} proves a refined version of the local limit theorem.
\end{itemize}

\section{Technical preliminaries}
\label{sec:prelim}

\subsection{Reinterpretation of the partition function}
\label{sec:rewrite}

Following \cite{BS10,BS11}, we interpret the (free and constrained) partition functions $Z^{Y}_{\gb,T}$ and $Z^{Y,\cons}_{\beta,T}$ as partition functions of a different model, in which a continuous time renewal process~$\tau$ interacts with a random potential~$w$. 

\subsubsection*{Rewriting the partition function}
This is done by expanding the exponential $e^{\beta H^Y_T(X)}$ appearing in~\eqref{compart}. 
Using the Markov property for $X$ (with a fixed realization of $Y$), we indeed get that
\begin{equation}
  \label{goodinho}
  \begin{split}
Z_{\gb,T}^Y&  = 1+ \sum_{k=1}^{\infty} \gb^k \int_{\cX_k(T)} \bP\big( X_{t_i}  = Y_{t_i}, \forall i \in \{1,\ldots, k\}\big) \dd t_1 \cdots \dd t_k \\
&  = 1+ \sum_{k=1}^{\infty} \gb^k \int_{\cX_k(T)} \prod_{i=1}^k\bP\big( X_{t_i-t_{i-1}}  = Y_{t_i} -Y_{t_{i-1}} \big) \dd t_1 \cdots \dd t_k \,,
 \end{split}
\end{equation}
where $\cX_k(T):=\{ {\bf t}\in \bbR^k  : 0< t_1< t_2<\dots<t_k<T\}$ is the $k$-dimensional simplex. 
An analogous formula holds for $Z_{\gb,T}^{Y,\cons}$, adding the constraint that $X_T=Y_T$: recalling the extra factor \(\beta\) in its definition and setting by convention $t_0=0$ and $t_{k+1} =T$, we have
\begin{equation}
  \label{goodinho2}
  Z_{\gb,T}^{Y,\cons} = \beta \bP(X_T =Y_T) + \sum^{\infty}_{k=1} \beta^{k+1} \int_{\cX_k(T)} \prod_{i=1}^{k+1} \bP\big( X_{t_i-t_{i-1}}  = Y_{t_i} -Y_{t_{i-1}} \big) \prod_{i=1}^{k} \dd t_i \,.
\end{equation}
Note that 
\(
\bbE[ \bP \big( X_{t_i-t_{i-1}}  = Y_{t_i} -Y_{t_{i-1}} \big)] = \P(W_{t_i-t_{i-1}}=0) \,.
\)
Next, we introduce a kernel $K(t)$ along with its \textit{disordered} counterpart $K_w$ (recall \eqref{def:beta0}). For \(t>0\), set
\begin{equation}
  \label{def:K}
  \begin{split}
    K(t)&:=\beta_0 \,\P(W_t=0) \\
    K_w(s,t,Y)&:=\beta_0\,\bP\left(X_{t-s}=Y_t-Y_s \right) .
  \end{split} 
\end{equation}
We have in particular $\int^{\infty}_0 K(t)\dd t=1$ and $\bbE\left[K_w(s,t,Y)\right]= K(t-s)$. 
The expansions \eqref{goodinho}-\eqref{goodinho2} can then be rewritten as follows: 
\begin{equation}
\label{goodexpress}
\begin{split}
Z^{Y}_{\beta,T}&= 1+\sum^{\infty}_{k=1} \left(\beta/\beta_0 \right)^k \int_{\cX_k(T)} \prod_{i=1}^k K_w(t_{i-1},t_i,Y) \dd t_i \,,\\
Z^{Y,\cons}_{\beta,T}&=  (\gb/\gb_0) K_w(0,T,Y)+\sum^{\infty}_{k=1} \left(\beta/\beta_0 \right)^{k+1} \int_{\cX_k(T)} \prod_{i=1}^{k+1}  K_w(t_{i-1},t_i,Y)  \prod_{i=1}^{k} \dd t_i \,.
\end{split}
\end{equation}
Applied to the case $\rho=0$, the annealed (free and constrained) partition functions can be rewritten as
\begin{equation}
  \label{hom-goodexpress}
  \begin{split}
   z_{\gb,T} & = 1+ \sum_{k=1}^{\infty} (\gb/\gb_0)^k \int_{\cX_k(T)} \prod_{i=1}^k K(t_i-t_{i-1}) \dd t_i \,,\\
   z_{\gb,T}^{\cons} &= (\gb/\gb_0)K(T) + \sum_{k=1}^{\infty}  (\gb/\gb_0)^{k+1} \int_{\cX_k(T)} \prod_{i=1}^{k+1} K(t_i-t_{i-1}) \prod_{i=1}^k \dd t_i \,.
\end{split} 
 \end{equation}

\subsubsection*{Interpretation as a weighted renewal process}

Let us define $\tau$ a renewal process on \(\bbR_+\) with inter-arrival density $K(\cdot)$. 
In other words, the sequence $(\tau_i)_{i\ge 0}$ satisfies $\tau_0=0$ and its increments are i.i.d.\ positive random variables with density $K(\cdot)$. The process $\tau$ can alternatively be considered as a locally finite random subset of~$\bbR_+$. 
We let $\bQ$ denote the distribution of $\tau$ and we also define the conditional probability by setting $\bQ( \,\cdot  \mid T\in \tau):=\lim_{\gep\to 0}\bQ(\, \cdot \mid \tau \cap [T,T+\gep]\ne \emptyset)$.
Let also \(u(T)\) be the renewal density $u(\cdot)$ defined for $A\subset (0,\infty)$ by $ \bQ[ |\tau \cap A|]=: \int_A u(t)\dd t$, and note that \(u(T)= z^{\cons}_{\gb_0,T}\).

We further introduce 
\begin{equation}
  \label{def:w}
w(s,t,Y):= \frac{K_w(s,t,Y)}{K(t-s)}  = \frac{\bP\left(X_{t-s}=Y_t-Y_s \right)}{\P\left(W_{t-s}=0 \right)} \,,
\end{equation}
and observe that $\bbE[w(s,t,Y)] =1$.
Setting $\cN_T:= \max\{k\ge 0 :\, \tau_k\le T\}$ we have the following interpretations for the expressions~\eqref{goodexpress}-\eqref{hom-goodexpress} in the constrained case
\begin{equation}\label{asanexpect}
 Z^{Y,\cons}_{\beta,T}=  u(T)\bQ\Big[ (\beta/\beta_0)^{\cN_T} \prod_{i=1}^{\cN_T} w(\tau_{i-1},\tau_i,Y) \ \Big| \ T\in \tau \Big] \quad  \text{ and }  \quad z^{\cons}_{\gb,T}=  u(T) \bQ\Big[ (\beta/\beta_0)^{\cN_T} \ \Big| \ T\in \tau \Big]\,.
\end{equation}
Hence $Z^{Y,\cons}_{\beta,T}$ (resp.\ $z^{\cons}_{\gb,T}$) corresponds to the partition function a constrained renewal process which receives a multiplicative bonus equal to $(\gb/\gb_0) \, w(\tau_{i-1},\tau_i,Y)$ (resp.\ $\gb/\gb_0$) for the $i$-th renewal point. 
While the expression \eqref{asanexpect} will mostly not be used in our proof, the idea of representing the partition function as an expectation with respect to a renewal process is at the core of most of our reasoning. 
% Let us now give a few properties of the distribution of the renewal process \(\tau\).

% Therefore, the partition functions in~\eqref{goodexpress} can be seen as a continuous renewal process with inter-arrival distribution~\eqref{def:K}, weighted by the potential $w$. 

\subsection{Some properties of the distribution of the renewal process}

By the local central limit theorem, we have that $K(t)\sim c_d t^{-\frac{d}{2}}$ as $t\to\infty$ when $(W_t)_{t\geq 0}$ is a simple random walk~\eqref{SRW} on \(\bbZ^d\).
In the $\gamma$-stable case~\eqref{JPP}, the \(\gamma\)-stable version of the local limit theorem (see e.g.\ \cite[Chapter~9]{GK68}) also provides an asymptotic expression for $\P(W_t=0)$ and hence for $K(t)$.
More precisely $K(t)$ tends to zero and satisfies the relation
\begin{equation}
  \label{implicit}
\varphi(1/K(t)) K(t)^{\gamma} \stackrel{t\to \infty}{\sim} \frac{c_{\gamma}}{t} ,
\end{equation}
where $c_{\gamma}$ is an explicit constant (in fact, we provide a proof in \Cref{app:LLT}).
This implies that, in all cases, $K(\cdot)$ is of the form
\begin{equation}
  \label{K-reg}
 K(t)=  L(t)t^{-(1+\alpha)} \,,
\end{equation}
where either $\alpha=\frac{d}{2}-1$ in case \eqref{SRW} or $\alpha =\frac{1-\gamma}{\gamma}$ in case \eqref{JPP} and $L(\cdot)$ is a slowly varying function (related to $\gamma$ and $\varphi$). 
For instance, in the $\gamma$-stable case, we deduce from~\eqref{implicit} that for any $\kappa \in \bbR$ and $\gamma \in (0,1)$ there exists $c_{\kappa,\gamma}>0$ such that 
\begin{equation}
  \label{philog}
 \varphi(t) \stackrel{t\to \infty}{\sim}  (\log t)^{\kappa}   \quad  \Longleftrightarrow  \quad L(t) \stackrel{t\to \infty}{\sim}   c_{\kappa,\gamma} (\log t)^{-\kappa/\gamma} \,.
\end{equation}
Let us finally mention a last property of regular varying function, referred to as Potter's bound, see~\cite[Thm.~1.5.6]{BGT89}. Given $t\ge 1$ and $\delta>0$ we have, for all $s\ge t$
\begin{equation}\label{potter}
    c_{\delta}\left(\frac{s}{t}\right)^{-(1+\alpha)-\delta}\le   \frac{K(s)}{K(t)}\le C_{\delta}\left(\frac{s}{t}\right)^{-(1+\alpha)+\delta} \,.
\end{equation}
We will also use this bound for $J(\cdot)$, replacing the exponent $\alpha$ by $\gamma$.

% \subsubsection* {About the renewal measure density.}

\smallskip

Let us observe that $u(t):= z_{\gb_0,t}$ defined above corresponds to the \textit{average density} of the renewal process in the sense that  for $A\subset (0,\infty)$ measurable we have $ \bQ[ |\tau \cap A|]=: \int_A u(t)\dd t$.
% A direct computation shows that this density $u(\cdot)$ exists and that 
% \begin{equation}
% \label{renewdensity}
% u(t) = K(t) + \sum_{k=1}^{\infty} \int_{\cX_k(t)} \prod_{i=1}^{k+1} K(t_i-t_{i-1}) \prod_{i=1}^k \dd t_i =:  z_{\gb_0,t}^{\cons} \,.
% \end{equation}
A continuous time version of Doney's~\cite{Don97} strong renewal theorem in the case of infinite mean (proven in \cite[Lem.~A.1]{BS11} and \cite[Thm.~8.3]{Top10}) yields the following asymptotic: if \(\alpha \in (0,1)\), we have
\begin{equation}
\label{eq:Doney}
u(t)  \stackrel{t\to \infty}{\sim} \frac{\alpha \sin(\pi\alpha)}{\pi} \frac{t^{\alpha-1}}{L(t)}=  \frac{\alpha \sin(\pi\alpha)}{\pi \, t^2 K(t)}.
\end{equation}
If on the other hand \(\alpha>1\), the renewal theorem shows that \(\lim_{t\to\infty} u(t) = \big(\int_0^{\infty} sK(s) \dd s\big)^{-1}\).

\subsection{Near critical estimates of the pure model}\label{phpure}

Recalling the implicit expression~\eqref{eq:freeenergy} for the free energy, we introduce for $\beta\geq \gb_0$ the probability density:
\begin{equation}
  \label{def:Kbeta}
  K_{\beta}(t) = \frac{\beta}{\beta_0} e^{-t  \tf(\gb)} K(t) \, .
\end{equation}
Indeed, thanks to~\eqref{eq:freeenergy}, we have $\int^{\infty}_0 K_{\beta}(t)\dd t=1$. 
 We let $\bQ_{\beta}$ denote the distribution of the renewal process whose i.i.d.\ increments have density $K_{\beta}(\cdot)$, and the consider $u_{\beta}(\cdot)$ the associate renewal density defined by $\bQ_{\beta}[ |\tau \cap A|]=:\int_A u_{\beta}(t)\dd t$. Like in the case $\beta=\beta_0$ an explicit computation yields,
\begin{equation}
  \label{zubet}
z_{\gb,T}^{\cons}= u_{\beta}(t) e^{\tf(\beta)t} \,.
\end{equation}
The renewal theorem implies that $u_{\beta}(t)$ converges to $(\int^t_0 tK_{\beta}(t) \dd t)^{-1}=\tf'(\beta)$. The following lemma, that we prove in \Cref{app:homogeneous}, shows that $u_{\beta}(t)$ is comparable to $u(t)$ for $t\leq \tf(\gb)^{-1}$ and comparable to $\tf'(\beta)$ for $t\ge \tf(\gb)^{-1}$.

\begin{lemma}
\label{lem:Zc}
Assume that $K(t) \sim L(t)t^{-(1+\alpha)}$. Then there exist constants $c_1,c_2$ such that for any $\gb\in [\gb_0,2\gb_0]$:
\begin{equation}\label{cawre}
c_1 u\Big( t \wedge \frac{1}{\tf(\gb)} \Big) \leq u_{\gb}(t) = z_{\gb,T}^{\cons} \, e^{-\tf(\beta)t} \leq c_2 u\Big( t \wedge \frac{1}{\tf(\gb)} \Big)  \,.
\end{equation}
Additionally, we have $c_1\tf'(\gb) \leq u(\frac{1}{\tf(\gb)}) \leq c_2 \tf'(\gb)$ where $\tf'$ denotes the derivative of $\tf$.
Furthermore for any $\gep>0$, there is some $\eta>0$ such that uniformly for $t \leq \eta/\tf(\gb)$ we have 
\begin{equation}\label{cawre2}
  (1-\gep) u(t)\leq  u_{\gb}(t) \leq (1+\gep) \frac{\gb}{\gb_0} u(t) \, .
\end{equation}
\end{lemma}

\subsection{Local limit theorem and a few random walk estimates}

Let us now collect a few useful random walk estimates, that are proven either in Appendix~\ref{app:LLT} or \cite[Section~2.3]{BLrel}.
We start with the local limit theorem (see e.g.\ \cite[Chapter 9]{GK68}), that we complement with a refinement.
Recall that \(K(t):= \beta_0\P(W_t=0)\).

\begin{proposition}[Local limit theorem]
  \label{prop:LLT}
  Let $g(x)=e^{-|x|^2/4\pi}$ in the case~\eqref{SRW} and $g=g_{\gamma}$ the density of the symmetric $\gamma$-stable distribution normalized so that $g_{\gamma}(0)=1$ in the case~\eqref{JPP}.
  Then, we have 
\begin{equation}\label{LLT}
  \lim_{t\to \infty}\sup_{y\in \Z^d} \left| \frac{\beta_0 \P(W_t=y)}{K(t)} - g \left(yK(t) \right) \right| =0\,.
\end{equation}
Moreover \(t\mapsto K(t) = \beta_0 \P(W_t=0)\) is differentiable and 
  \begin{equation}\label{krixx}
  K'(t) \stackrel{t \to \infty}{\sim} - (1+\alpha) t^{-1} K(t)\,.
  \end{equation}
\end{proposition}

We also give a local large deviation estimate for the \(\gamma\)-stable case, which improves \Cref{prop:LLT} in the case where \(|y|\) is much larger than \(K(t)^{-1}\).
The bounds are due to the so-called one big-jump behavior. We refer to \cite[Thm.~2.4]{Ber19a}  for the analogous result in the discrete time setting (the proof can easily be adapted to the continuous-time setting).

\begin{proposition}
  \label{prop:localLD}
Assume that~\eqref{JPP} holds with \(\gamma \in (0,2)\).
Then, given $c>0$ there exist constants \(c_1,c_2\) such that, for any $|y| \geq c  K(t)^{-1}$, we have
\begin{equation*}
c_1  \, t J(y) \le  \P(W_t=y) \leq c_2\, t J(y) \,.
\end{equation*}
\end{proposition}

We also complete these results by displaying useful monotonicity properties.

\begin{lemma}[\cite{BLrel}, Lemma~2.3]
\label{lem:unimod}
If $x\mapsto J(x)$ is unimodal, then so is \(x\mapsto \P(W_t=x)\) for all $t > 0$. 
More precisely, for all \(t>0\),
\begin{equation*}
|x| \leq |y| \quad \Rightarrow \quad \P(W_t=x) \geq \P(W_t =x) .
\end{equation*}
Furthermore $t\mapsto\P(W_t=0)$ and hence $t\mapsto K(t)$ are non-increasing.
\end{lemma}

\section{A first convexity lower bound on the free energy}
\label{sec:Jensen}

We prove in this section a general lower bound on the free energy, that can be obtained with a soft argument based on Jensen's inequality. 
This result implies in particular that \Cref{th:rel} holds when $\nu=1$, \textit{i.e.}\ in the SRW case for $d\ge 4$ and in the $\gamma$-stable case for $\gamma\in (0,1/2]$.

\begin{proposition}\label{jenseninq}
  If $W=(W_t)_{t\geq 0}$ satisfies either \eqref{SRW} or \eqref{JPP}, then there exists a constant $C_0$ such that for all $\rho\in(0, \frac12)$ we have for any $\beta\in [\beta_0,2\beta_0]$
  \begin{equation}\label{trkt}
  \tf(\rho,\beta+C_0\rho)\ge \tf(\beta).
  \end{equation}
\end{proposition}
\noindent  The proof of
 \Cref{jenseninq} relies on the following estimate which we prove just below.

\begin{lemma}\label{envibou}
If $W=(W_t)_{t\geq 0}$ satisfies either \eqref{SRW} or \eqref{JPP},
there exists a constant $C$ (depending on $J$) such that for all $t\ge 0$ and $\rho\in (0,\frac12)$ we have
\begin{equation}
 \bbE\left[\log w(0,t,Y)\right]\ge -C \rho \,.
\end{equation}
\end{lemma}

\begin{proof}[Proof of Proposition \ref{jenseninq}]
Recalling the definition of $\bQ_{\beta}$ from \eqref{def:Kbeta}, we define the conditional distribution $\bQ_{\beta,T}:=\lim_{\gep \to 0}\bQ_{\beta}(  \cdot \ | \ \tau\cap [T,T+\gep)\ne \emptyset )$, \textit{i.e.}\ the distribution of the constrained renewal.
From~\eqref{asanexpect}, we have 
 \begin{equation*}
  \bbE\left[  \log Z^{Y,\mathrm{c}}_{\beta+C_0\rho,T}\right]= \log z^{\mathrm{c}}_{\beta,T} +\bbE\bigg[ \log \bQ_{\beta,T}\bigg[  \Big(\frac{\beta+C_0\rho}{\beta}\Big)^{\cN_T}\prod_{i=1}^{\cN_T} w(\tau_{i-1},\tau_i,Y) \bigg] \bigg] \,.
 \end{equation*}
By Jensen's inequality, passing the logarithm inside the expectation $\bQ_{\beta,T}$, we have that the second term is bounded from below by
\begin{equation*}
%  \bbE\bigg[ \log \bQ_{\beta,T}\bigg[  \Big(\frac{\beta+C\rho}{\beta}\Big)^{\cN_T}\prod_{i=1}^{\cN_T} w(\tau_{i-1},\tau_i,Y) \bigg] \bigg]\\
%  \ge 
\bQ_{\beta,T}\bigg[  \sum_{i=1}^{\cN_T}  \Big( \log \Big(\frac{\beta+C_0\rho}{\beta}\Big)- \bbE\left[\log w(\tau_{i-1},\tau_i,Y)  \right]\Big) \bigg]\ge 0 \,,  
 \end{equation*}
 where in the last inequality we have used \Cref{envibou}, which implies that all terms in the sum are non-negative (provided that $C_0$ has been fixed large enough, recalling also that $\beta \leq 2 \beta_0$).
 We conclude that $\bbE\big[  \log Z^{Y,\mathrm{c}}_{\beta+C\rho,T}\big]\ge \log z^{\mathrm{c}}_{\beta,T}$, so we get the inequality after dividing by~$T$ and letting $T\to\infty$.
\end{proof}

\begin{proof}[Proof of Lemma \ref{envibou}]
For this proof only we set $w(t):=w(0,t,Y)$, for better readability.
We use the following bound, valid for all \(u\geq -1\):
\(
\log (1+u)\ge u-u^2+\ind_{\{u\le -1/2\}} \log(1+u) \,.
\)
Applied to $u=w(t)-1$ and using the fact that $\bbE[w(t)]=1$, we obtain
\begin{equation}
\bbE[ \log w(t)]\ge -\bbE[ w(t)^2-1] + \bbE\left[\ind_{\{w(t)\le 1/2\}} \log w(t)\right] \,.
\end{equation}
To conclude we are going to show that there exists  $C>0$ such that for all $\rho\in (0,\frac12)$
\begin{equation}
  \label{tthing}
  \bbE[ w(t)^2-1] \leq C \rho 
  \quad \text{ and } \quad
  \bbE\left[\ind_{\{w(t)\le 1/2\}} \log w(t)\right] \ge -C\rho\,.
\end{equation}

\smallskip\noindent
\textit{First inequality in~\eqref{tthing}. }
By the definition~\eqref{def:w} of \(w(t) = w(0,t,Y)\), we have 
\begin{equation*}
 \bbE[ w(t)^2]= \frac{\bbE\left[ \bP(X_t=Y_t)^2\right]}{\P(W_t=0)^2}\le \frac{\bP(X_t=0)}{\P(W_t=0)}
 =\frac{K( (1-\rho) t)}{K(t)} \,,
\end{equation*}
where for the middle inequality we have just used that $\bP(X_t=Y_t)^2\le \bP(X_t=Y_t) \bP(X_t=0)$ (thanks to \Cref{lem:unimod}) and \(\bbE[\bP(X_t=Y_t)] = \P(W_t=0)\).
The proof of the first line of \eqref{tthing} is thus reduced to showing that
\begin{equation}\label{CCCC}
 \sup_{t\ge 0} \frac{K((1-\rho)t)}{K(t)}\le 1+C\rho \,.
\end{equation}
For $t\in [0,1]$, using the fact that $K(\cdot)$ is Lipschitz on $[0,1]$, we have
\begin{equation*}
 \frac{K((1-\rho)t)}{K(t)}-1\le \frac{ K((1-\rho)t)-K(t)}{K(t)}\le  \rho \max_{t\in [0,1]} |K'(t)|.
\end{equation*}
For the case of $t\ge 1$ we observe that 
\begin{equation*}
  \frac{K((1-\rho)t)}{K(t)}-1 =\frac{t\int^{1}_{1-\rho} (-K'( u t)) \dd u   }{ K(t)}\le   \rho \sup_{t\ge 1,u\ge \frac12} \frac{|t K'(ut)|}{K(t)} \,,
\end{equation*}
and the supremum is finite by \Cref{prop:LLT}-\eqref{krixx}. This concludes the proof of \eqref{CCCC}. 

\smallskip\noindent
\textit{Second inequality in~\eqref{tthing}. }
Because of the local limit theorem \Cref{prop:LLT}-\eqref{LLT}, we observe that having $w(t)\le 1/2$ means that $|Y_t|$ is unusually large. 
This corresponds to $|Y_t|> ct^{1/2}$ for the simple random walk and $|Y_t|>c K(t)^{-1}$ in the case of $\gamma$-stable walk.
%  (notice that we have put a strict inequality so that the observation is also valid for small values of $t$).
In the case of the simple random walk, we therefore have,
\begin{equation}\label{rrho1}
 \bbE\left[\ind_{\{w(t)\le 1/2\}} \log w(t)\right]
 \ge \sum_{|y|> c t^{1/2}} \log\left(\frac{\bP(X_t=y)}{\bP(X_t=0)}\right) \bbP[Y_t=y].
\end{equation}
In the case of the $\gamma$-stable walk we have in the same manner
\begin{equation}\label{rrho2}
  \bbE\left[\ind_{\{w(t)\le 1/2\}}\log w(t)\right]
 \ge \sum_{|y|> c K(t)^{-1}} \log\left(\frac{\bP(X_t=y)}{\bP(X_t=0)}\right) \bbP[Y_t=y].
\end{equation}
% (in both \eqref{rrho1} and \eqref{rrho2} we use that  using also that $\P(W_t=0) \leq  \bP(X_t=0)$ which is a consequence of Lemma \ref{lem:unimod}).
To conclude the proof, we need to bound the r.h.s.\ in \eqref{rrho1}-\eqref{rrho2} by $-C\rho$.
Recall also for the following that \(\bP(X_t=y) = \P(W_{(1-\rho)t}=y)\) and \(\bbP[Y_t=y] = \P[W_{\rho t}=y]\).

\smallskip\noindent
\textit{The random walk case.}
We are going to use the fact that for every $t> 0$ and $|y|>c t^{1/2}$ we have   
\begin{equation}\label{LLT2}
 \frac{\P(W_{t}=y)}{\P(W_t=0)}\ge e^{-\frac{c'|y|^2}{t}}.
\end{equation}
The above is a consequence of the local limit theorem~\eqref{LLT} when $t\ge 1$ and $|y|=O(\sqrt{t})$ and can be obtained by rougher methods when either $t$ is small or $y$ is large.
We get directly from \eqref{LLT2}  that
\begin{equation*}
  \bbE\left[\ind_{\{w(t)\le 1/2\}} \log w(t)\right]\ge 
   -\frac{c'}{(1-\rho)t} \sum_{|y|> c K(t)^{-1}} |y|^2 \P[W_{\rho t}=y]  \geq -\frac{c'}{(1-\rho)t} \bbE\left[ |W_{\rho t}|^2 \right] \geq - c'' \rho .
\end{equation*}

\smallskip\noindent
\textit{The $\gamma$-stable case.}
We use \Cref{prop:localLD} to obtain
\begin{equation}\label{neuf}
  \bbE\left[\ind_{\{w(t)\le 1/2\}} \log w(t)\right] 
  \geq - \sum_{|y|> c K(t)^{-1}} \log \left( \frac{K(t(1-\rho))}{c_1 t(1-\rho) J(y)}\right) \cdot c_2\, \rho t \, J(y)\\
\end{equation} 
Since $K(t(1-\rho))/(c_1 t(1-\rho))\ge  \frac{c'_1 K(t)}{t}$  (recall that $\rho>1/2$)
to conclude we just need to show that there exists a constant $C$ such that 
\begin{equation}
  \sup_{t\ge 0)} \sum_{|y|> c K(t)^{-1}} \log \left( \frac{ c'_1 K(t)}{ t J(y)}\right)  t J(y) \le C.
\end{equation}
When $t\in(0,1)$ we have $K(t)/t\le \beta_0$ and thus  
\begin{equation*}
\sum_{|y|> c K(t)^{-1}} \log \left( \frac{ c'_1 K(t)}{ t J(y)}\right)  t J(y)
  \le t\bigg( \log\left(  \beta_0 c'_1\right) + \sum_{|y|>0}  (\log J(y))J(y) \bigg)\ge - C t.
\end{equation*}
When $t\ge 1$ using Potter's bound for $J(\cdot)$ (see~\eqref{potter}, with \(\alpha\) replaced by \(\gamma\) and with $\delta=\gamma/2$) we have for $|y|> c K(t)^{-1}$, 
  \begin{equation*}
    % \label{Potts1}
 \frac{1}{C}(|y| K(t))^{-1-3\gamma/2}    \le \frac{J(y)}{J(cK(t)^{-1})}\le  C (|y| K(t))^{-1-\gamma/2}.
\end{equation*}
We also have, as a consequence of \eqref{implicit}, that $\frac{1}{c'} K(t)/t\leq J(cK(t)^{-1}) \leq c' K(t)/t$, so that for $|y|> c K(t)^{-1}$ we have
\begin{equation}\label{framingJ}
\frac{1}{C'} (|y| K(t))^{-1-3\gamma/2} \leq \frac{t J(y)}{K(t)} \leq C' (|y| K(t))^{-1-\gamma/2} \,.
\end{equation}
Using these estimates in \eqref{neuf} -- we use the lower bound in \eqref{framingJ} to control the $\log$ term  and the upper bound for $J(y)$ outside of $\log$ --
this yields
\begin{equation*}
  % \label{neufp}
\sum_{|y|> c K(t)^{-1}} \log \left( \frac{ c'_1 K(t)}{ t J(y)}\right)  t J(y)
  \le  C''  K(t)\sum_{|y|> c K(t)^{-1}}  \log ( C'' |y| K(t)) \cdot (|y| K(t))^{-1-\gamma/2},
\end{equation*}
and we conclude by observing that by Riemann sum approximation
\begin{equation}
 \lim_{t\to \infty} K(t)\sum_{|y|> c K(t)^{-1}}  \log ( C'' |y| K(t)) \cdot (|y| K(t))^{-1-\gamma/2} = \int \ind_{\{|u|\ge c\}}
 \log (C'' u) u^{-1-\gamma/2} \dd u,
\end{equation}
so that in particular the quantity is bounded uniformly in $t$.
\end{proof}

% 
% 
% \section{Properties of the free energy}
% 
% 
% 
% 
% \subsection{Existence of the free energy}
% 
% 
% 
% % 
% 
% To conclude, using the relation~\eqref{freecontraint} between the free and the constrained partition function, we also get that, $\bbP$-a.s. and in $L^1(\bbP)$,
% \begin{equation}
%  \lim_{T\to \infty} \frac{1}{T}\log  Z^{Y}_{\beta,T} =\lim_{T\to \infty} \frac{1}{T} \bbE \log  Z^{Y}_{\beta,T} = \tf(\rho,\gb).
% \end{equation}

% Let us start with a lower bound on the free energy that comes as the result of the  combination of Proposition \ref{envibou} with Jensen's inequality. Recalling the definition of $L$ \eqref{L-reg}, we set 

% {\blue Note that the restriction to small $\beta$ is necessary. For the random walk on $\bbZ^d$ we have for $\beta$ large  
% $$\tf(\rho,\beta)\sim \beta -\rho \log \beta+O(1)$$}

% With \Cref{envibou} at hand, we  conclude the proof of \Cref{jenseninq}.

\section{The main lower bound on the free energy}
\label{sec:lower}

In this section, we gather a single statement all the different lower bounds that we need to establish in order to prove \Cref{th:irrel}, \Cref{th:rel} (for $\gamma \in (\frac12,\frac23)$) and \Cref{th:3DRW}, and we further reduce the proof of this statement to an estimate of the second moment of a modified partition function, see Proposition \ref{trickyy}.
This estimate, which is the technical core of the proof, is proved in \Cref{trixx,trixx1,trixx2}.

\subsection{A statement that gathers them all}

We introduce a notation to treat the case $\gamma=\frac{2}{3}$ with arbitrary slowly varying $\varphi$.
Recalling the definition of $L$ in \eqref{K-reg}, we set
\begin{equation}\label{defRR}
 \mathrm R(t):= \int_{1<u<s<t}\frac{L(u)^2}{L(s)^2}   \frac{\dd u}{u} \frac{\dd s}{s}. 
\end{equation}
The function $\mathrm{R}(\cdot)$ is increasing and slowly varying. Note that we necessarily have $\mathrm R(t)/\log t \to\infty$ as \(t\to\infty\), as can be seen by simply restricting the integral to $s\in (u,Cu)$, where the constant $C$ can be taken arbitrarily large. If $\varphi(u)\stackrel{u\to \infty}{\sim}  (\log u)^{\kappa}$ we have $L(t) \stackrel{t\to \infty}{\sim} c_{\kappa} (\log t)^{-3\kappa/2}$, see~\eqref{philog}, and this implies that 
\begin{equation*}
\mathrm R(t) \sim c'_{\kappa} 
\begin{cases} 
  (\log t)^2  \quad &\text{ if  }\kappa<1/3,\\
  (\log \log t)(\log t)^{2} \quad &\text{ if } \kappa=1/3,\\
  (\log t)^{1+3\kappa} &\text{ if  } \kappa>1/3. 
\end{cases}
\end{equation*}
We also set $\mathbf r(u):= (\mathrm R^{-1}(1/u))^{-1/3}$. The expression appearing in the r.h.s.\ of \eqref{soluce} corresponds to $\mathbf r(C\rho)$ for some constant $C>0$.

\begin{proposition}\label{prop:multiprop}
  Given $\gep>0$, there exists $\rho_1(\gep,J)>0$,  and $C_1(\gep,J)>0$ such that:
  
  \begin{enumerate}
   \item \label{multi-i}
    If $J(x)\stackrel{|x|\to \infty}{\sim} |x|^{-(1+\gamma)}$ with $\gamma\in \left(\frac{2}{3},1 \right)$ then  for every $\rho\in [0,\rho_1]$ and every $\beta\in (\beta_0,2\beta_0]$ we have
   \begin{equation}\label{eq:lllb}
    \tf(\rho,\beta)\ge (1-\gep)\tf(\beta).
    \end{equation}
     \item \label{multi-ii}
  If $J(x)=\varphi(|x|)(1+|x|)^{-\frac{5}{3}}$, then  for every $\rho\in [0,\rho_1]$ and every $\beta\in [\beta_0+  \mathbf{r}(C_1\rho),2\beta_0]$ we have
   \begin{equation}
    \tf(\rho,\beta)\ge (1-\gep)\tf(\beta).
   \end{equation}
     \item \label{multi-iii}
      If $(W_t)_{t\geq 0}$ is the SRW with $d=3$, then  for every $\rho\in [0,\rho_1]$ and
     $\beta\in [\beta_0+ e^{- \rho^{-1/2}/C_1},2\beta_0]$ we have
   \begin{equation}
    \tf(\rho,\beta)\ge (1-\gep)\tf(\beta).
   \end{equation}
     \item \label{multi-iv}
      If  $J(x)\stackrel{|x|\to \infty}{\sim} |x|^{-(1+\gamma)}$ with $\gamma\in \left(\frac12,\frac{2}{3}\right)$  then for every $\beta\in [\beta_0+ C_1 \rho^{\frac{1}{2-\nu}},  2\beta_0]$ we have
     \begin{equation}
    \tf(\rho,\beta)\ge (1-\gep)\tf(\beta).
   \end{equation}
  \end{enumerate}
  \end{proposition}

\subsection{Reducing the proof to a finite volume estimate}

We explain here how the proof of \Cref{prop:multiprop} can be reduced to a finite volume estimate.
We assume that \(\gamma\in (\frac12,1)\) or that $(W_t)_{t\geq 0}$ is the SRW in dimension $d=3$.
We consider a system of length $nT$ where  $T=A \tf(\beta)^{-1}$, $A$ is a large constant (which does not depend on $\beta$) and $n$ is an arbitrary integer. 
The idea is to split our system into $n$ blocks of length $T$.
We first define for $r\le s$ the constrained partition function on a segment by setting 
\begin{equation}\label{partsegos}
  Z^{Y}_{\beta,[r,s]}:= \beta \bE\Big[ e^{\beta \int^s_r  \ind_{\{X_{t}=Y_t\}} \dd t}  \ind_{\{X_{s}=Y_s\}} \ \Big| \ X_r=Y_r\Big] =
  Z^{ \theta_r(Y),\mathrm c}_{\beta,s-r} \,,
\end{equation}
where $\theta_r(Y)=(Y_{t+r}-Y_r)_{t\ge 0}$, and by convention $Z^{Y}_{\beta,[s,s]}=1$. 
We then introduce a modified partition function for the \(i\)-th block, by setting
\begin{equation}
  \label{def:barZ}
 \hat Z_i(Y):=  K(T)\int_{(i-1+\eta)T <r<s<i T } Z^{Y}_{\beta,[r,s]} \dd r \dd s.
\end{equation}
We will show that an approximate factorization holds and that a lower bound for the partition function~$Z^{Y}_{\beta,nT}$ can be obtained by multiplying $\hat Z_i(Y)$ over an arbitrary collection of blocks (see \Cref{blocksection} below).
We then obtain a lower bound for $\tf(\rho,\beta)$ by showing that most blocks can bring a significant contribution: this is the following statement.

\begin{proposition}
\label{bloxyblox}
Given $\eta>0$, and $A\ge A_0(\eta)$ sufficiently large, there exists $\rho_1(\eta,A,J)>0$ and $C_1(\eta,A,J)>0$ such that, for $\beta$ in the range specified by \Cref{prop:multiprop}, we have
 \begin{equation*}
  \bbP\left[ \hat Z_1(Y)\ge  A^{-4} e^{(1-\eta)A} \right]\ge (1-\eta) \,.
 \end{equation*}
\end{proposition}

\noindent The details of how \Cref{prop:multiprop} is deduced from \Cref{bloxyblox} are given in the next subsections.
Proposition \ref{bloxyblox} itself is proved by using a standard second moment method, getting bounds for the first two moments and using Paley-Zygmund inequality.
The first moment is relatively easy to estimate.
\begin{lemma}
  \label{lem:firstmom}
  Let $T:= A \tf(\beta)^{-1}$. Then there is some $A_0$ such that, for $A\ge A_0$, for every $\beta$ sufficiently close to $\gb_0$, we have
\begin{equation*}
 \bbE\big[  \hat Z_1(Y)\big]\ge   A^{-3} e^{(1-\eta)A}.
\end{equation*}
\end{lemma}

\begin{proof}[Proof of \Cref{lem:firstmom}]
First of all, recalling the definition~\eqref{def:barZ} of \(\hat Z_1(Y)\), we have
\begin{equation*}
  % \label{purespe}
 \bbE\big[ \hat Z_1(Y)\big]=K(T)\int_{\eta T< r <s <T} z^{\cons}_{\beta, s-r} \dd r \dd s .
\end{equation*}
Now, consider the restriction of the integral to $r\in [\eta T , \eta T+ \tf(\beta)^{-1}]$ and $s\in [ T -\tf(\beta)^{-1}, T]$, so that in particular \(s-r \geq \tf(\beta)^{-1}\).
Then, using \Cref{lem:Zc} to get that \(z^{\cons}_{\beta,s-r} \ge c \, u( \tf(\beta)^{-1} ) e^{\tf(\beta)(s-r)}\) for \(s-r\geq \tf(\beta)^{-1}\), we obtain that
\begin{equation}
  \label{eq:borneinfEZbar}
  \bbE\big[  \hat Z_1(Y)\big]\ge c' K(T)  \frac{u( 1/\tf(\beta) )}{\tf(\beta)^2}  e^{\tf(\beta) \, (1-\eta) T}\ge c''   \frac{K(1/\tf(\beta)) u(1/\tf(\gb)) }{\tf(\beta)^2}  A^{-3} e^{(1-\eta)A}.
\end{equation}
In the last inequality we used Potter's bound~\eqref{potter} to get $K(T)= K(A \tf(\beta)^{-1})\ge c A^{-2}K(\tf(\beta)^{-1})$; recall that $\gamma<1/2$ so we have  $\alpha\in (0,1)$.
Then \eqref{eq:Doney} implies that $K(t)u(t) t^{2}$ converges to a constant and we obtain that
$\bbE[\hat Z_1(Y)]\ge  c' A^{-2} e^{(1-\eta)A}$.
This gives the desired conclusion by taking~\(A\) large enough.
\end{proof}

\noindent The estimate for the second moment is given in \Cref{trickyy} below and is much harder to obtain. 
Its proof requires to distinguish between various cases and spans Sections~\ref{trixx}-\ref{trixx1}-\ref{trixx2}.

\begin{proposition}
  \label{trickyy}
 Given $\eta>0$, and $A\ge A_0(\eta)$ sufficiently large, there exists $\rho_1(\eta,A,J)>0$,  and $C_1(\eta,A,J)>0$ such that for any $\rho,\beta$ satisfying the assumption of Proposition \ref{prop:multiprop}
 \begin{equation*}
 \frac{\bbE\big[  \hat Z_1(Y)^2\big]}{\bbE\big[ \hat Z_1(Y)\big]^2}\le \frac{1}{1-\eta}\left(\frac{A}{A-1} \right)^2.
 \end{equation*}
\end{proposition}

We conclude this subsection by deducing Proposition \ref{bloxyblox} from the moment estimates.

\begin{proof}[Proof of \Cref{bloxyblox}]
Using \Cref{lem:firstmom}, Paley-Zygmund inequality and Proposition \ref{trickyy}, we get
\begin{equation*}
  \bbP\left(  \hat Z_1(Y) \geq A^{-4} e^{(1-\eta) A} \right) 
    \geq \bbP\left( \hat Z_1(Y) \geq A^{-1} \bbE\big[ \hat Z_1(Y) \big]  \right)  \geq \left(\frac{A}{A-1}\right)^2 \frac{ \bbE\big[ \hat Z_1(Y) \big]^2}{ \bbE\big[\big( \hat Z_1(Y) \big)^2\big]}\ge 1-\eta.
\qedhere
\end{equation*}
% the second inequality simply being Paley-Zygmund inequality.
% This concludes the proof of \Cref{bloxyblox} thanks to \Cref{trickyy}.  
\end{proof}

\subsection{Block decomposition and factorization of the partition function}
\label{blocksection}

Let us now explain how our factorization of the partition functions works and how one can deduce \Cref{prop:multiprop} from \Cref{bloxyblox}.

\smallskip\noindent
\textit{The block decomposition. }
For \(n\in \bbN\), we decompose $Z^{Y}_{\beta,nT}$ according to the set of segments of the form $((i-1)T,iT]$, which we refer to as \textit{blocks}, that are visited by the set $\{t_1,\dots,t_k\}$ in the integral~\eqref{goodexpress}.
For an arbitrary $k\geq 0$ and $\bt \in \cX_k(nT)$, we define $I(\bt)$ the set of blocks visited by $\bt$, that is
  \begin{equation*}
   I(\bt) = \Big\{ i  : \  \{t_1,\dots, t_k\} \cap ((i-1)T,iT] \ne \emptyset\Big\}
  \end{equation*}
More specifically, letting $I$ encode the set of visited blocks, we can write
\begin{equation}\label{decompz}
  Z^{Y}_{\beta,nT}
 =: \sum_{I\subset \lint 1, n \rint } Z^{Y}_{\beta,T,I} \,,
  \end{equation}
where $Z^{Y}_{\beta,T,\emptyset}:= 1$ and for $|I|\ge 1$,  $Z^{Y}_{\beta,T,I}$ is obtained by restricting the integrals \eqref{goodexpress} to the sets
\[
\cX_k(T,I):= \Big\{ {\bf t} \in\bigcup_{k\ge 0} \cX_k(nT)  : \ I({\bf t})=I \Big\} \,,
\]
that is
\begin{equation*}
  Z^{Y}_{\beta,T,I}:=\sum^{\infty}_{k=1} \left(\beta/\beta_0 \right)^k \int_{\cX_k(T,I)} \prod_{i=1}^{k}  K_w(t_{i-1},t_i,Y) \prod_{i=1}^k \dd t_i.
\end{equation*}
Recalling \eqref{partsegos}, by integrating over all $t_i$ within a block except for the first one we obtain that for $I=\{i_1,\dots, i_{\ell}\}$ with $1\leq i_1 < \cdots <i_{\ell} $, we have
\begin{equation}\label{blockdecompo}
  Z^{Y}_{\beta,T,I}=\!\!\!
    \int\limits_{ \substack{ (r_{j},s_{j})_{j=1}^{\ell} \\ (i_j-1)T < r_j\le s_j\le i_{j}T } } \!\!\! (\beta/\beta_0)^{\ell} \prod_{j=1}^{\ell} K_w(s_{{j-1}},r_{j},Y) Z^{Y}_{\beta,[r_{j},s_{j}]}\dd r_j(\delta_{r_j}(\dd s_j)+ \dd s_j) \,,
\end{equation}
where by convention $s_0=0$.
The Dirac mass term $\delta_{r_i}(\dd s_i)$ appearing in the integral over~$t_i$ is present to take into account the possibility that a given block is visited only once.

\smallskip\noindent
\textit{The factorization step and conclusion. }
Considering $n\ge 1$, we need to prove that if $\beta$ and $\rho$ satisfy the assumption of  Proposition \ref{prop:multiprop}
then for all $n$ sufficiently large we have
\begin{equation}\label{strox}
  Z^{Y}_{\beta,nT}\ge e^{n(1-\gep)A},
\end{equation}
which readily implies that $\frac{1}{nT}\log Z^{Y}_{\beta,nT} \geq (1-\gep) \tf(\gb)$, \textit{i.e.}\ that $\tf(\rho,\gb) \geq (1-\gep)\tf(\gb)$.
We use the bound $Z^{Y}_{\beta,nT} \ge  Z^{Y}_{\beta,T,I}$ (recall \eqref{decompz}) for one specific set $I$ of \textit{good blocks} which avoids stretches of unfavorable environment.
We can then conclude using the following factorization lower bound (recall the definition~\eqref{def:barZ} of \(\hat Z_i(Y)\)).

\begin{lemma}\label{intend}
Given $\eta\in(0,1/2)$, there exists a constant $c$ (which depends on $\eta$) and $n_0(Y,T)$ such that for any $n\ge n_0$, for any $I=\{i_1,\dots,i_{\ell}\} \subset \lint 1, n\rint$ such that $\ell=|I|\ge n/2$, for any choice of $r_j,s_j$ which satisfy $(i_j-1+\eta)T<r_j<s_j<i_jT$, we have \(\bbP\)-a.s.
\begin{equation}\label{kipoint}
  Z^{Y}_{\beta,nT} \geq  Z^{Y}_{\beta,T,I} \geq  c^n \prod_{i\in I} \hat Z_i(Y).
\end{equation}
\end{lemma}

\begin{proof}[Proof of Proposition \ref{prop:multiprop}]
Recalling Proposition \ref{bloxyblox} a block is said to be \textit{good} if $\hat Z_i(Y) \geq A^{-4} e^{(1-\eta) A}$.
The law of large number implies that \(\bbP\)-a.s., for $n$ sufficiently large, the number of good blocks in $\lint 1, n\rint$ is larger than $(1-2\eta)n$. 
Therefore, taking $I$ to be the set of such good blocks, we have thanks to \Cref{intend}
\begin{equation*}
  Z^{Y}_{\beta,nT}\ge c^n \left(A^{-4} e^{(1-\eta)A}\right)^{(1-2\eta)n}.
\end{equation*}
We conclude that \eqref{strox} holds by taking first $\eta$ sufficiently small and then $A$ sufficiently large.
\end{proof}

\subsection{Proof of \texorpdfstring{\Cref{intend}}{Lemma}}
\label{technis}

Recall the definition~\eqref{blockdecompo} of  \(Z^{Y}_{\beta,T,I}\).
The terms $K_w(s_{j-1},r_{j},Y)$ appearing in the product within the integrand are problematic because they prevent the factorization of the integral into a product.
To solve this, we introduce a minimal distance between $s_{j-1}$ and $r_j$.
We obtain a simple lower bound for $Z^{Y}_{\beta,T,I}$ by forbidding the visit of a proportion $\eta$ of each block where~$\eta$ is a fixed small parameter:
\begin{equation}\label{blockdecompobis}
Z^{Y}_{\beta,T,I}\ge
  \!\!\! \int\limits_{ \substack{ (r_{j},s_{j})_{j=1}^{\ell} \\ (i_j-1+\eta)T < r_j\le s_j\le i_{j}T } }  \!\!\! \prod_{j=1}^{\ell} K_w(s_{{j-1}},r_{j},Y) Z^{Y}_{\beta,[r_{j},s_{j}]}\dd r_j \dd s_j =: \hat Z^{Y}_{\beta,T,I}.
\end{equation}
Notice that we suppressed the term $(\beta/\beta_0)^{\ell}\ge 1$ and the $\delta_{r_j}(\dd s_j)$ to make the expression more compact (their contribution is negligible).
Then, recalling the definition~\eqref{def:barZ} of \(\hat Z_i(Y)\), \eqref{kipoint} is obtained by using the following inequality in \eqref{blockdecompobis}: there exists a constant $c$ (which depends on $\eta$) and $n_0(Y,T)$ such that for any $n\ge n_0$, 
\begin{equation}
  \label{dodge2}
  \prod_{j=1}^{\ell} K_w(s_{{j-1}},r_{j},Y) = \prod_{i\in I} K_w(s_{i_{-}},r_i,Y)  \ge c^n K(T)^{|I|} \,,
\end{equation}
where $i_{-}$ denotes the index preceding $i$ in $I$, that is $i_{-}=\max\{j\in I, j<i\}$.
The inequality holds provided \(r_i-s_{i_-}\geq \eta T\) for all \(i\); it is very rough, so we use \textit{ad-hoc} (brute force) method.

\smallskip\noindent
\textit{Under Assumption~\eqref{JPP}. }
Using the monotonicity of $\bP(X_{s}=x)$ in $|x|$ (see \Cref{lem:unimod}), we obtain
\begin{equation}\label{jkljkl}
\gb_0^{-1} K_w(s_{i_{-}},r_i,Y)\ge \bP\big( X_{r_i-s_{i-}}= \hat Y_{iT}- \hat Y_{(i_{-}-1)T} \big) \,,
\end{equation}
where $(\hat Y_{t})$ is the walk obtained from $(Y_t)$ by  reversing the values of negative jumps, making all jumps positive.
Now, there exists $c>0$ such that for $t$ sufficiently large
\begin{equation}\label{doublebound}
 \bP(X_t=y)\ge c \min(K(t), tJ(|y|) )\ge c' K(t)(1+ |y| K(t))^{-2}.
\end{equation}
The first inequality follows from \Cref{prop:LLT}-\eqref{LLT} (for $|y|\le 1/K(t)$) and \Cref{prop:localLD} (for $|y|\ge 1/K(t)$). 
The second inequality follows from Potter's bound~\eqref{potter} applied to $J(y)$ with $\delta=1-\gamma$ when $|y|\ge 1/K(t)$. More precisely
we have
\[
tJ(|y|) \ge c t J(1/K(t))  (|y|K(t))^{-2} \ge c' K(t) (|y|K(t))^{-2} \,,
\]
where the last inequality comes from \eqref{implicit}. 
We deduce from~\eqref{doublebound} that 
\[
 \prod_{i\in I} K_w(s_{i_{-}},r_i,Y) \geq (c')^n  \prod_{i\in I} K(r_i-s_{i-}) \Big( 1+ K(r_i-s_{i-}) (\hat Y_{iT}- \hat Y_{(i_{-}-1)T} ) \Big)^{-2}\,.
\]
Since by construction $ r_i-s_{i-}\in [\eta T, (i- i_- + 1)T]$, using again Potter's bound~\eqref{potter}, there exist a constant $C=C(\eta)>0$  such that
\begin{equation*}
\frac{1}{C} K(T)\; (i-i_-)^{-2}\leq K(r_i-s_{i-}) \leq C K(T) \,.
\end{equation*}
Hence we obtain (for a different $c>0$) that
\[  \prod_{i\in I} K_w(s_{i_{-}},r_i,Y)\ge 
c^{n} K(T)^{|I|} \prod_{i\in I} (i-i_-)^{-2}\left(1+ C K(T)(\hat Y_{iT}- \hat Y_{(i_{-}-1)T} ) \right)^{-2}\,.
\]    
Using the assumption $|I|\ge n/2$, Jensen's inequality implies that $\prod_{i\in I} (i-i_-)\le (n/|I|)^{|I|} \leq 2^{|I|}$. 
Hence, to conclude, we only need to show that
\begin{equation*}
 \prod_{i\in I} \big( 1+ K(T)(\hat Y_{iT}- \hat Y_{(i_{-}-1)T} )\big)\le C^n
\end{equation*}
with a constant $C$ independent of $T$.
By applying the simple inequality $(1+\sum a_j)\le \prod(1+a_j)$ (valid for any collection of positive numbers) to $a_j = \hat  Y_{jT}-\hat Y_{(j-1)T}$ for $j \in \{i_-, \ldots, i\}$, $i\in I$, we obtain
\begin{equation*}
  % \label{laborrn}
 \prod_{i\in I}   \big( 1+ K(T)(\hat Y_{iT}- \hat Y_{(i_{-}-1)T})\big)
 \le \prod_{j=1}^{T} \big( 1+ K(T)(\hat Y_{jT}- \hat Y_{(j-1)T})\big)^{2} \,,
\end{equation*}
where the power $2$ accounts for the possibility that each $ \hat  Y_{j}-\hat Y_{j-1}$  might appear twice in the product.
Then, using the law of large numbers we have for all $n$ sufficiently large
\begin{equation}
  \log \prod_{i\in I}   \big( 1+ K(T)(\hat Y_{iT}- \hat Y_{(i_{-}-1)T})\big)
  \le 3 n\log \bbE\big[ \log \big( 1+ K(T)\hat Y_{T}\big) \big] \,.
\end{equation}
To conclude, we observe that $\bbE[ \log ( 1+ K(T)\hat Y_{T}) ] $ is uniformly bounded for $T\ge 1$. Indeed, $K(T)\hat Y_T$ converges in law, and it is a standard exercise to check that its tail distribution can be uniformly bounded by a negative power (see e.g.\ \cite[Thm.~1.1]{Nag79} or \cite[Thm.~2.1]{Ber19a} for a simpler statement, both in the discrete setting).

\smallskip\noindent
\textit{Under Assumption~\eqref{SRW}. }
Since the coordinates are independent, we may focus on proving \eqref{dodge2} for $d=1$, where \(K(T)\sim c /\sqrt{T}\).
% (again, we use \textit{ad-hoc} rough methods):
% \begin{equation}
%   \label{dodge2SRW}
%   \prod_{i\in I} \bP( X_{r_i- s_{i_-}} = Y_{r_i}-Y_{s_{i_-}})\ge c^n (1/\sqrt{T})^{|I|} \,.
% \end{equation}
By combining the local limit theorem and easy tail estimates, we have
$\bP(X_t =y) \geq \frac{c}{\sqrt{t}}\, e^{- c' y^2/t}$ for any $t\geq 1$ and $y\in \bbZ$.
Hence, we obtain 
\[
  \prod_{i\in I} \bP( X_{r_i- s_{i_-}} = Y_{r_i}-Y_{s_{i_-}}) \geq  c^n \prod_{i\in I} \frac{1}{ \sqrt{r_i-s_{i_{-}}}} e^{- c'(Y_{r_i}-Y_{s_{i_{-}}})^2/(r_i-s_{i_{-}})} \,.
\]
Since by construction we have $r_i-s_{i-}  \in [\eta T, (i-i_{-}+1)T]$, we get (for different constants $c,c'$)
\begin{equation}
  \label{interm-dodge}
   \prod_{i\in I} \bP( X_{r_i- s_{i_-}} = Y_{r_i}-Y_{s_{i_-}})  \ge c^n  K(T)^{|I|} \exp\bigg(  - c' \sum_{i\in I} \frac{(Y_{r_i}-Y_{s_{i_{-}}})^2}{(i-i_{-}) T}\bigg) \,.
\end{equation}
% and it remains to show that the sum in the exponential is bounded from above by $Cn$, independently of $T,I$ (and $r_i,s_{i_-}$).
By Cauchy-Schwarz inequality, we have
\[
  (Y_{r_i}-Y_{s_{i_{-}}})^2 \leq (i - i_{-}+1)\Big((Y_{r_i}-Y_{(i-1) T})^2 +  \sum_{j=i_-+2}^{i-1}  (Y_{jT} -Y_{(j-1)T})^2+ (Y_{(i_-+1)T}-Y_{s_{i_-}})^2  \Big) \,.
\]
Bounding $(i - i_{-}+1)/(i-i_{-})$ by $2$ and adding another factor $2$ to account for the fact that some terms may appear twice, we get that for $n$ sufficiently large we have
\[
  \sum_{i\in I} \frac{(Y_{r_i}-Y_{s_{i_{-}}})^2}{(i-i_{-}) T} \leq 4 \sum_{j=1}^n  \sup_{r,s\, \in[(j-1)T,jT]} \frac{1}{T} ( Y_r-Y_s )^2 \le 5n \bbE\Big[ \sup_{r,s\,  \in [0,T]} \frac{1}{T} ( Y_r-Y_s )^2\Big] \,,
\]
the last inequality following from the law of large numbers.
Since the last expectation is bounded by a constant, plugging this in~\eqref{interm-dodge} proves \eqref{dodge2}.
\qed

\section{Bounding the moment ratio using renewal processes}
\label{trixx}

In this section, we give the first step of the proof of \Cref{trickyy}.
It is common to all cases and consists in bounding the moment ratio using the interpretation in terms of weighted renewal processes.
More precisely, we give a general bound on \(\bbE[\hat Z_1(Y)^2]/\bbE[\hat Z_1]^2\) in terms of visits of renewal intervals by another independent renewal.

Given two renewal process $\tau$ and $\tau'$ and $0<a<b$, we define the following set of indices
\begin{equation}
  \label{def:J}
  \begin{split}
    \cJ^{(1)}_{[a,b]} &:= \big\{ i\ge 1  : \  \tau'\cap [\tau_{i},\tau_{i+1}] \ne \emptyset  \text{ and } [\tau_i,\tau_{i+1}]\subset[a,b]  \big\}, \\
  \cJ^{(2)}_{[a,b]} & := \big\{ j\ge 1  : \ \tau \cap [\tau'_{j},\tau'_{j+1}] \ne \emptyset \text{ and } [\tau'_j,\tau'_{j+1}]\subset[a,b]  \big\},
  \end{split}
\end{equation}
and
\begin{equation}
  \label{def:J2}
  \begin{split}
  \mathfrak J_{[a,b]}&:= \Big\{ i\ge 1 : \Big[\tau_i,\frac{\tau_i+\tau_{i+1}}{2}\Big]\cap \tau' \ne \emptyset \text{ and }  [\tau_i,\tau_{i+1}]\subset[a,b]  \Big\} \,, \\
  \mathfrak J'_{[a,b]}&:= \Big\{ i\ge 1 : \Big[\frac{\tau_i+\tau_{i+1}}{2},\tau_{i+1}\Big]\cap \tau' \ne \emptyset \text{ and }  [\tau_i,\tau_{i+1}]\subset[a,b]  \Big\} \,.
\end{split}
\end{equation}
For \(\beta>\beta_0\), we also introduce $\tilde Q = \tilde Q_{\gb}$ the distribution of the stationary renewal with inter-arrival law $K_{\beta}(\cdot)$ on~$\bbR_+$.
Under $\tilde Q_{\beta}$, the increments $(\tau_{i}-\tau_{i-1})_{i\ge 2}$ are i.i.d.\ and independent of~$\tau_1$, with \(\tau_1\) having density $\tf'(\beta) \bar K_{\beta}(t)$ where  $\bar K_{\beta}(t)=\int^{\infty}_t K_{\beta}(s)\dd s$. Note that \(K_{\beta}(\cdot)\) has a finite mean \(1/\tf'(\beta)\).
We stress that this choice makes the renewal process invariant by translation and time reversal. 

We are going to prove the following proposition.
We will use it afterwards, in \Cref{trixx1,trixx2}, to prove Proposition~\ref{trickyy} in all different cases.

\begin{proposition}
  \label{topitop2}
 There exist positive constant $c_0=c_0(J)$ and $C_A$ such that for all $\beta$ and $\rho$
 \begin{equation}
  \label{topitop}
\frac{\bbE\big[  \hat Z_1(Y)^2\big]}{ \bbE\big[  \hat Z_1(Y)\big]^2} -1 
 \leq C_A  \tilde Q^{\otimes 2} \left[ e^{c_0 \rho \, (|\cJ^{(1)}_{[0,T]}|+|\cJ^{(2)}_{[0,T]}|)}-1 \right]\le   C_A  \tilde Q^{\otimes 2} \left[ e^{4c_0 \rho  |\mathfrak J_{[0,T]}|}-1 \right].
\end{equation}
\end{proposition}

\begin{proof}
Let us first start with the easier second inequality  in \eqref{topitop}. Using Cauchy-Schwartz inequality twice, and noting that $|\cJ^{(1)}_{[0,T]}|$ and $|\cJ^{(2)}_{[0,T]}|$ have the same distribution, as well as $|\mathfrak J_{[0,T]}|$ and $|\mathfrak J'_{[0,T]}|$ by time reversal, we get
\begin{equation*}
  \tilde Q^{\otimes 2} \Big[ e^{\gl (|\cJ^{(1)}_{[0,T]}|+|\cJ^{(2)}_{[0,T]}|)}\Big]\le   \tilde Q^{\otimes 2} \Big[ e^{2\gl |\cJ^{(1)}_{[0,T]}|}\Big]=  \tilde Q^{\otimes 2} \left[ e^{2\gl (|\mathfrak J_{[0,T]}|+|\mathfrak J'_{[0,T]}| )}\right]\le   \tilde Q^{\otimes 2} \left[ e^{4\gl |\mathfrak J_{[0,T]}|}\right]\,.
\end{equation*}

To prove the first inequality in \eqref{topitop} we define a point process on  $[\eta T, T]$ whose distribution  $\hat Q_T$ is given by (for $k\ge 2$)
\begin{equation}
 \hat Q_T \big( \# \tau =k \ ; (\tau_1,\tau_2,\dots, \tau_k)\in \dd t_1\dots \dd t_k \big)
 =  \frac{\dd t_1}{\hat z_T}\left(\frac{\beta}{\beta_0}\right)^{k-1} \prod_{i=2}^k K(t_i-t_{i-1})\dd t_i \,.
\end{equation}
where $\tau \subset [\eta T, T]$ is a finite set which is identified with a finite increasing sequence and $\hat z_T$ is the associated partition function and can be written as \(\hat z_T=\int_{\eta T< s< t< T} z^{\cons}_{\beta,t-s} \dd s \dd t \).
We first prove an intermediate estimate by showing that
\begin{equation}\label{lefirsta}
  \frac{\bbE\big[  \hat Z_1(Y)^2\big]}{ \bbE\big[  \hat Z_1(Y)\big]^2}-1  \le  \hat Q^{\otimes 2}_T \bigg[ e^{c \rho \, (| \cJ^{(1)}_{[\eta T,T]}|+| \cJ^{(2)}_{[\eta T,T]}|)}-1 \bigg].
\end{equation}
Recalling the definition~\eqref{def:w} of \(w(s,t,Y)\), note that we have
\begin{equation*}
 \frac{\hat Z_1(Y)}{\bbE\big[ \hat Z_1(Y) \big]}=  \hat Q_T\bigg[ \prod_{i=1}^{\#\tau}   w(\tau_{i-1},\tau_i,Y)   \bigg] \,.
\end{equation*}
Thus the second moment ratio \(\bbE[\hat Z_1(Y)^2]/\bbE[\hat Z_1]^2\) is equal to
\begin{equation*}
 \bbE\bigg[\hat Q^{\otimes 2}_T \bigg( \prod_{i=2}^{\#\tau}   w(\tau_{i-1},\tau_i,Y) \prod_{j=2}^{\#\tau'}   w(\tau'_{j-1},\tau'_j,Y) \bigg)  \bigg]
 = \hat Q^{\otimes 2}_T \bbE\bigg[\prod_{i=2}^{\#\tau}   w(\tau_{i-1},\tau_i,Y) \prod_{j=2}^{\#\tau'}   w(\tau'_{j-1},\tau'_j,Y) \bigg].
\end{equation*}
Note that $w(a,b,Y)$ is independent of $w(c,d,Y)$ if and only if the intervals $(a,b)$ and $(c,d)$ are disjoint.
To bound the expectation above for a fixed realization of $\tau$ and $\tau'$, we simply ``remove''  from the product 
the indices $i\in \cJ^{(1)}_{[\eta T,T]}$ and $j\in \cJ^{(2)}_{[\eta T,T]}$ so that the remaining intervals 
$(\tau_{i-1},\tau_i)$ and $(\tau'_{j-1},\tau'_j)$ are disjoint.
By ``removing'', we simply mean that we bound these terms by a constant. 
More precisely, we have $\bP(X_{b-a} = Y_b-Y_a) \leq \bP(X_{b-a} =0)$ thanks to \Cref{lem:unimod}, so that recalling the definition~\eqref{def:w} of \(w(a,b,Y)\) we obtain that
\begin{equation*}
  % \label{eq:constantbound} 
 w(a,b,Y)\leq  \sup_{t\ge 0}\frac{ K((1-\rho)t)}{K(t)} \le e^{c_0 \rho} \, \,,
\end{equation*}
using also \eqref{CCCC} for the last inequality.
Using this bound for all $i\in \cJ^{(1)}_{[\eta T,T]}$ and $j\in \cJ^{(2)}_{[\eta T,T]}$ we get
\begin{multline*}
  \prod_{i=2}^{\#\tau}   w(\tau_{i-1},\tau_i,Y) \prod_{j=2}^{\#\tau'}   w(\tau'_{j-1},\tau'_j,Y) \\ \le
  e^{c_0 \rho \, (|\mathcal  \cJ^{(1)}_{[\eta T,T]}|+|  \cJ^{(2)}_{[\eta T,T]}|)} \prod_{i =2 ,\,  i\notin  \cJ^{(1)}_{[\eta T,T]}}^{\# \tau}   w(\tau_{i-1},\tau_i,Y) \prod_{j =2 ,\, j\notin  \cJ^{(2)}_{[\eta T,T]}}^{\# \tau'}  w(\tau'_{j-1},\tau'_j,Y) \,.
\end{multline*}
Since having $i\notin \mathcal  \cJ^{(1)}_{[\eta T,T]}$ and $j\notin  \cJ^{(2)}_{[\eta T,T]}$ implies that $[\tau_{i-1},\tau_i] \cap [\tau'_{j-1},\tau'_j]=\emptyset$, we can use the independence of the weights $w(a,b,Y)$ on disjoint intervals and the fact that $\bbE[ w(a,b,Y) ]=1$ to obtain
\begin{equation*}
  \bbE\bigg[\prod_{i=2}^{\#\tau}   w(\tau_{i-1},\tau_i,Y) \prod_{j=2}^{\#\tau'}   w(\tau'_{j-1},\tau'_j,Y)\bigg] \le
  e^{c_0 \rho \, (| \cJ^{(1)}_{[\eta T,T]}|+| \cJ^{(2)}_{[\eta T,T]}|) } \,.
\end{equation*}
This concludes the proof of~\eqref{lefirsta} and the last step is then to replace $\hat Q_T$ by the stationary distribution~$\tilde Q$.

We let $\tilde Q_T$ denote the distribution of $\tau\cap[\eta T,T]$ under $\tilde Q$. 
We have, for any $k\ge 1$,
\begin{equation*}
 \tilde Q_T\big( \# \tau=k \ ; (\tau_1,\tau_2,\dots, \tau_k \big)\in \dd t_1\dots \dd t_k )
 = \tf'(\beta) \bar K_{\beta}(t_1-\eta T)\dd t_1\prod_{i=2}^k K_{\beta}(t_i-t_{i-1}) \dd t_i \bar K_{\beta}(T-t_k) \,.
\end{equation*}
The distribution $\hat Q_T$ is thus absolutely continuous with respect to $\tilde Q_T$. We are going to prove that there exists $C_A>0$ such that for every $k\ge 2$ and $t_1,\dots, t_k$
\begin{equation}\label{byca}
 \frac{\dd \hat Q_T}{\dd \tilde Q_T}(t_1,\dots,t_k) := \frac{e^{\tf(\beta)(t_k-t_1)}}{\hat z_T \tf'(\beta) \bar K_{\beta}(t_1-\eta T)  \bar K_{\beta}(T-t_k)} \leq \sqrt{C_A}.
\end{equation}
Combining \eqref{lefirsta} and \eqref{byca}, we therefore obtain that 
\begin{equation*}
\frac{\bbE\big[  \hat Z_1(Y)^2\big]}{ \bbE\big[  \hat Z_1(Y)\big]^2} -1 
 \le C_A  \tilde Q_T^{\otimes 2} \left[ e^{c \rho \, (| \cJ^{(1)}_{[\eta T,T]}|+| \cJ^{(2)}_{[\eta T,T]}|)}-1 \right]  \le C_A  \tilde Q^{\otimes 2} \left[ e^{c \rho \, (| \cJ^{(1)}_{[0,T]}|+| \cJ^{(2)}_{[0,T]}|)}-1 \right]\,,
\end{equation*}
which concludes the proof.
Let us now prove \eqref{byca}.
By~\eqref{eq:borneinfEZbar} and \Cref{lem:Zc}, we get that
\[
\hat z_T = K(T)^{-1} \bbE[\hat Z_1(Y)] \geq c_A \frac{\tf'(\gb)}{\tf(\gb)^2} e^{\tf(\beta) \, (1-\eta)T}
\]
and it thus sufficient to prove that for some \(c_A'>0\)
\begin{equation}
  \label{infov} 
 \forall v \in [0,T]\,, \quad  \frac{\tf'(\gb)}{\tf(\gb)} e^{v\tf(\gb)} \bar K_{\gb}(v)\ge c_A'\,.
 \end{equation}
Using this inequality for $v= t_1-\eta T$ and $v=T-t_k$, we obtain then \eqref{byca} for $C_A= (c_A^2 c_A')^{-2}$.
To prove~\eqref{infov}, recalling that \(T=A\tf(\gb)^{-1}\) and \(\beta\geq \beta_0\), let us observe that
\[
  e^{v \tf(\gb)}\bar K_{\gb} (v) = \frac{\gb}{\gb_0} \int_0^{+\infty} e^{-s \tf(\gb)} K(v+s) \dd s \geq e^{-A} \int_0^{T} K(2T) \dd s \geq c e^{-A} T K(2T) \,,
\] 
using also the monotonicity of $K(\cdot)$ (see \Cref{lem:unimod}).
Now, by Potter's bound, we have \(T K(2T) \ge c''_A \tf(\gb)^{-1} K(\tf(\gb)^{-1})\).
We then simply observe that $\tf(\beta)^{-2}\tf'(\beta)K(1/\tf(\beta))$ is of order~$1$, as can be deduced from the fact that \(\tf'(\gb)\) is comparable to \(u(\tf(\gb)^{-1})\) (see \Cref{lem:Zc}) together with~\eqref{eq:Doney}.
This concludes the proof of~\eqref{infov} and thus of \Cref{topitop2}.
\end{proof}

\section{Proving \texorpdfstring{\Cref{trickyy}}{the proposition} when disorder is relevant: \texorpdfstring{\(\gamma\in (\frac12,\frac23]\) or \(d=3\)}{}}
\label{trixx1}

We now expose how \Cref{trickyy} can be obtained from \Cref{topitop2}.
We distinguish between the non-marginal case $\gamma\in \left(\frac{1}{2},\frac{2}{3}\right)$, \textit{i.e.}\ item~\ref{multi-iv}, and the marginal case $\gamma=\frac23$ (or $d=3$ for the SRW), \textit{i.e.}\ items~\ref{multi-ii}-\ref{multi-iii}.

\subsection{The non-marginal case \texorpdfstring{$\gamma\in \left(\frac{1}{2},\frac{2}{3}\right)$}{}: main statement}

Recall the definition~\eqref{def:J2} of $\mathfrak J_{[a,b]}$ and let us define, for $H>1$ and $[a,b] \subset [0,T]$, the quantities
\begin{equation}
  \label{def:NH}
N_H([a,b]):=\#\left\{ i\in \mathfrak{J}_{[a,b]} :  \tau_{i+1}-\tau_{i} \in (H,2H] \right\}
\end{equation}
and $N_1([a,b]):=\#\left\{ i\in \mathfrak{J}_{[a,b]}:  \tau_{i+1}-\tau_{i} \in (0,2] \right\}$.
We also set $N_H(t):=N_H([0,t])$.
From H\"older's multi-index inequality, we have for any $r>1$ 
\begin{equation}
  \label{multiHolder}
  \tilde Q^{\otimes 2} \left[ e^{ u |\mathfrak J_{[0,T]}| } \right]\le \prod_{j=1}^{\infty}  \tilde Q^{\otimes 2} \left[
  e^{ u\, \frac{r ^j}{r-1} \, N_{2^j}(T)}\right]^{\frac{r-1}{r^j}}.
\end{equation}
Note that $N_{2^j}(T)=0$ if $2^j>T$ so that the above product is in fact a finite one.
We need to show the following result.
\begin{lemma}
  \label{squiz}
  Assume that $\gamma\in (\frac12,\frac23)$ and let $\alpha =\frac{1-\gamma}{\gamma} \in (\frac12,1)$.
There are constant $c,c'$ such that whenever $H\le T$ and $v \le c (H/T)^{2\alpha-1}$, we have
\begin{equation*}
  \tilde Q^{\otimes 2} \left[
  e^{ v  N_H(T) }\right]\le 1 + c' v\, (T/H)^{2\alpha-1}.
\end{equation*}
\end{lemma}
\noindent Let us quickly deduce the desired result  before proving Lemma \ref{squiz}.

\begin{proof}[Proof of \Cref{trickyy}]
We set $u:=4c_0 \rho$. According to \Cref{topitop2}, we have
\begin{equation}
\label{wobz}
     \frac{\bbE\big[  \hat Z_1(Y)^2\big]}{\bbE\big[  \hat Z_1(Y)\big]^2}\le    1+ C_A\left(\tilde Q^{\otimes 2} \left[ e^{ u |\mathfrak J_{[0,T]}| } \right]-1\right).
\end{equation}
Recalling the assumption of \Cref{prop:multiprop}-\ref{multi-iv}, given $\delta>0$, if $C_1$ is sufficiently large we may assume that $u\le \delta T^{1-2\alpha} = \delta T^{(\nu-2)/\nu}$: indeed, by \Cref{homener}, $T$ is of order $(\beta-\beta_0)^{-\nu}$ and by assumption $(\beta-\beta_0)\ge  C_1 \rho^\frac{1}{2-\nu}$.
We then apply~\eqref{multiHolder} and we notice that $u \, \frac{r^{j}}{r-1} \leq \frac{\delta}{r-1} (2^j/T)^{2\alpha-1}$ for every $j$.
We can thus use \Cref{squiz} with $v=\frac{r^{j}u}{r-1}$ for each term, and obtain
\begin{align*}
  \tilde Q^{\otimes 2} \left[ e^{ u |\mathfrak J_{[0,T]}| } \right]
 &  \le \prod_{j=1}^{\infty}   \left( 1 + c' \frac{u\, r ^j}{r-1}  (T/2^j)^{2\alpha-1} \right)^{\frac{r-1}{r^j}}\\
  & \leq  \exp \left( c'' u\, T^{2\alpha-1} \sum_{j\geq 0} 2^{-(2\alpha-1)j} \right)\le e^{C \, \delta}.
\end{align*}
We can then conclude using \eqref{wobz}, provided $\delta$ is sufficiently small.
\end{proof}

\subsection{Proof of Lemma \ref{squiz}}

To simplify notation we assume that $H>1$ (the argument easily adapts to the case $H=1$).
Let us fix $\gep>0$ such that $1/\gep \in \mathbb{N}$.
Let us stress right away that if \(\gep T > 2H\) then necessarily \(N_H(T) \leq 2/\gep\) and we can expand \(e^{v N_H(T)} \leq 1+ c'_{\gep} v\) which concludes the proof (recall \(H\leq T\)).
We therefore assume in the rest of the proof that \(\gep T \geq  2H\).
Note that we have
\[ 
N_H(T) \leq \sum_{i=1}^{1/\gep} N_H\big([(i-1)\gep T, i\gep T ]\big) + \sum_{i=1}^{1/\gep} N_H\big([(i-\tfrac12)\gep T, (i+\tfrac12)\gep T ]\big) \,,
\] 
since any interval \([\tau_i,\tau_{i-1}] \subset [0,T]\) of length inferior to $2H$ must be counted in one of the two sums.
Combining the above decomposition of \(N_H(T)\) with  H\"older's inequality (with $2/\gep$ terms), we obtain that \(\tilde Q^{\otimes 2} [ e^{ v  N_H(T) }]\) is bounded by
\begin{equation*}
   \bigg(\prod_{i=1}^{1/\gep}\tilde Q^{\otimes 2} \left[ e^{ \frac{2v}{\gep}  N_H([(i-1)\gep T, i\gep T ]) }\right]\tilde Q^{\otimes 2} \left[ e^{ \frac{2v}{\gep}  N_H([(i-\frac 1 2 )\gep T, (i+\frac 1 2)\gep T ]) }\right]\bigg)^{\gep/2} =\tilde Q^{\otimes 2} \left[e^{ \frac{2 v}{\gep}  N_H(\gep T) }\right] \,,
\end{equation*}
where we have used translation invariance for the last identity.
To control the r.h.s.\ we are going to prove the following bound. For some universal constant \(C>0\) and any \(k\geq 1\),
\begin{equation}
\label{boundNH}
\tilde Q^{\otimes 2}\left[ N_H(\gep T)\ge k \right] \leq \Big(1-  \frac{\gep}{C} (H/T)^{2\alpha-1} \Big)^{k-1} \leq e^{ - \frac{\gep}{C} (H/T)^{2\alpha-1} (k-1)} \,.
\end{equation}
Using the identity
\begin{equation*}
  % \label{poipoi}
  \tilde Q^{\otimes 2} \left[ e^{ v  N_H(T) } \right] \leq \tilde Q^{\otimes 2} \left[ e^{ \frac{2 v}{\gep}  N_H(\gep T) }\right]
  =  1 + \sum_{k\ge 1} e^{\frac{2v}{\gep}(k-1)}\left(e^{\frac{2v}{\gep}}-1\right)  \tilde Q^{\otimes 2}\left[ N_H(\gep T)\ge k \right] \,,
\end{equation*}
we obtain thanks to~\eqref{boundNH} that
\begin{align*}
  \tilde Q^{\otimes 2} \left[ e^{ v  N_H(T) } \right] -1 
& \leq  \left(e^{\frac{2v}{\gep}}-1\right) \sum_{j\geq 0} e^{ (\frac{2v}{\gep} -  \frac{\gep}{C} (H/T)^{2\alpha-1}) \, j} \,.
\end{align*} 
Choosing \(c:= \gep^2/2C\) and provided that \(v\leq c (H/T)^{2\alpha-1}\), we therefore get that 
\begin{align*}
\tilde Q^{\otimes 2} \left[ e^{ v  N_H(T) } \right] -1
& \leq  C' v \sum_{j\geq 0} e^{  -  \frac{\gep}{2 C} (H/T)^{2\alpha-1} \, j}  = C' v \Big(1-  e^{-\frac{\gep}{2C} (H/T)^{2\alpha-1}} \Big)^{-1} \leq  C'' v (H/T)^{2\alpha-1} \,,
\end{align*}
as desired.
It therefore only remains to prove~\eqref{boundNH}.
For this we introduce $(\cT_k)_{k\ge 0}= (\tau_{i_k})_{k\ge 0}$ the ordered sequence of the elements  of
\[
\left\{ \tau_i : \  i \in\mathfrak J_{[0,\infty)} ,\ \tau_{i}-\tau_{i-1} \in (H,2H] \right\} \,.
\]
In particular, we have 
\begin{equation}\label{lghui}
\tilde Q^{\otimes 2}( N_H(\gep T)\ge k)
 =  \tilde Q^{\otimes 2}\big(\cT_{k} \leq \gep T \big)
\le \bbP\big(\forall i\in \lint 1,k-1 \rint,\ \cT_{i+1}-\cT_i \le \gep T \big)\,.
\end{equation}
Note that the increment $\cT_{k+1}-\cT_{k}$ is \textit{not} independent of $\tau'\cap [0,\cT_k]$ (in fact, we compute its conditional distribution below), so in particular the right-hand side in \eqref{lghui} is \textit{not} equal to $\bbP( \cT_2-\cT_1\le \gep T)^{k-1}$.
We are however going to show that the dependence is sufficiently weak so that an approximate factorization holds.
We then estimate $\bbP( \cT_2-\cT_1\ge  \gep T)$ to conclude the proof.
These two steps are given in \Cref{preskiid} and \Cref{tobeshown} below.
%  and provide and adequate bound on $\bbP( \cT_2-\cT_1\le \gep T)^{k-1}$ to conclude the proof of \Cref{squiz}.

Let us start with the factorization lemma. This part of the reasoning only requires that $\gamma\in (0,1)$, and will also be used in the next subsection to cover the marginal case $\gamma=\frac23$.
We need to introduce some notation (we refer to \Cref{fig:thetas} for an illustration).
Let us define $\theta_k:=\inf \{ \tau'\cap[\tau_{i_k-1},\tau_{i_k}] \}$ the entrance point of \(\tau'\) in the interval. 
It follows from our definitions (recall~\eqref{def:NH} and \eqref{def:J}-\eqref{def:J2}) that the set over which the infimum is taken is nonempty and we have  $\tau_{i_{k-1}} \leq \theta_k \leq \frac{\tau_{i_k-1}+\tau_{i_k}}{2}$.
We also define the $\sigma$-algebra 
\[
\cG_{k}:= \sigma\big( \tau \cap (-\infty, \cT_k],  \tau'\cap(-\infty,\theta_k] \big)\,,
\]
in such a way that $(\cG_k)$ is an adapted filtration for $\cT_k$.

\begin{lemma}
  \label{preskiid} 
Assume that~\eqref{JPP} holds with $\gamma\in (0,1)$.
There exists a constant $C=C_A$ (which does not depend on \(H,T\) with $H\leq T = A\tf(\gb)^{-1}$) such that for any measurable $E\subset \bbR_+$ and any \(\gb\in (\gb_0,2\gb_0)\)  we have
 \begin{equation*}
     \frac{1}{C} \tilde Q^{\otimes 2}\left[ (\cT_{2}-\cT_1)\in E  \right] \le  \tilde Q^{\otimes 2}\left[ (\cT_{k+1}-\cT_k)\in E \mid \cG_k \right]\le  C   \tilde Q^{\otimes 2}\left[ (\cT_{2}-\cT_1)\in E  \right]\,.
 \end{equation*}
In particular we have (for $C$ as above) for any \(t >0\)
\begin{equation}\label{kobeshown}
 \bbP\big(\forall i\in \lint 1,k-1 \rint, \cT_{i+1}-\cT_i \le t \big)\le \Big(1-\frac{1}{C} \tilde Q^{\otimes 2}\big( (\cT_{2}-\cT_1)> t \big)  \Big)^{k-1}.
\end{equation}
\end{lemma}
\noindent To conclude the proof of~\eqref{boundNH}, we combine \eqref{kobeshown} with the following result.

\begin{lemma}
\label{tobeshown}
Assume that $J(x)\stackrel{|x|\to \infty}{\sim} |x|^{-(1+\gamma)}$ with $\gamma\in (\frac12,\frac23)$ and let $\alpha =\frac{1-\gamma}{\gamma} \in (\frac12,1)$.
Then, for $\gep >0$ small enough, we have
\begin{equation*}
  % \label{eq:tobeshown}
  \tilde Q^{\otimes 2}\left( \cT_{2}-\cT_1\ge \gep T \right) \ge   \gep\left(\frac{T}{H}\right)^{1-2\alpha}.
 \end{equation*}
\end{lemma}
\noindent This ends the proof of~\eqref{boundNH} and thus concludes the proof of \Cref{squiz} -- it only remains to prove Lemmas~\ref{preskiid} and~\ref{tobeshown}.
\qed

\begin{proof}[Proof of \Cref{preskiid}]
Let us introduce
\begin{equation*}
  \theta'_k:= \min\tau'\cap [\cT_k,\infty), \quad \bar \theta_k= \cT_k-\theta_k \ \text{ and }\ \bar \theta'_k:=\theta'_k-\cT_k \,,
\end{equation*}
and let $\cG'_k:= \sigma( \tau \cap (-\infty, \cT_k],  \tau'\cap(-\infty,\theta'_k])$.
We refer to Figure~\ref{fig:thetas} for an illustration.
\begin{figure}
\centering 
  \includegraphics[scale=1.2]{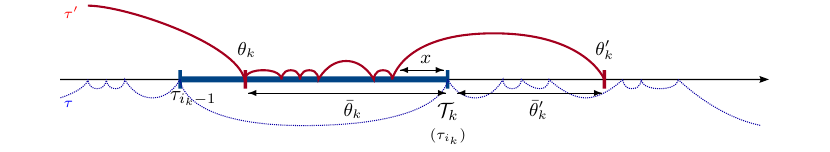}
  \caption{Illustration of the definitions of \(\theta_k,\theta_k'\) and \(\bar \theta_k,\bar \theta_k'\).
  The renewal \(\tau\) is represented below (in blue), the renewal \(\tau'\) above (in red); \(\theta_k\) is the first point of \(\tau'\) in the interval \([\tau_{i_k-1}, \tau_{i_k}]\) and \(\theta_k'\) the first point of \(\tau'\) after it, \(\bar \theta_k, \bar\theta_k'\) are the respective distances to \(\cT_k = \tau_{i_k}\).
  Recall that, by definition of $\mathfrak{J}_{[a,b]}$, the point \(\theta_k\) is in the first half of \([\tau_{i_k-1}, \tau_{i_k}]\). Recall also that \(\tau_{i_k}-\tau_{i_k-1} \in (H,2H]\).
  }
  \label{fig:thetas}
\end{figure}
Using the renewal structure of~$\tau$ we notice that the conditional probability $\bbP[ (\cT_{k+1}-\cT_k)\in E  \ | \ \cG'_k]$ is a function of the overshoot~$\bar \theta_k'$. 
Hence, there exists a non-negative function $\varphi_E$ such that 
\begin{equation*}
  \tilde Q^{\otimes 2}\left[ (\cT_{k+1}-\cT_k)\in E  \ | \ \cG'_k \right]= \varphi_E(\bar \theta'_k).
\end{equation*}
Moreover the conditional distribution of $\bar \theta'_k$ given $\cG_k$ only depends on $\bar \theta_k$, hence
\begin{equation}\label{okoki}
  \tilde Q^{\otimes 2}\big[ (\cT_{k+1}-\cT_k)\in E  \ \big| \ \cG_k \big]= \bbE\big[\varphi_E(\bar \theta'_k) \ \big| \  \cG_k\big]
   = \tilde Q^{\otimes 2}\big[\varphi_E(\bar \theta'_k) \ \big| \  \bar \theta_k\big]\,.
\end{equation}
Now, using a simple renewal computation, we can compute explicitly the conditional density of $\bar \theta'_k$ given $\{\bar \theta_k=s\}$, that we denote $f_s$.
Thanks to the stationarity of \(\tilde Q = \tilde Q_{\beta}\), we have
\begin{equation*}
 f_s(t):=K_{\beta}(s+t)+ \int^s_0 u_{\beta}(s-x)K_{\beta}(x+t)\dd x \,,
\end{equation*}
where the variable $x$ corresponds to the distance between \(\cT_k\) and the last contact of $\tau'$ before $\cT_k$; the first term \(K_{\beta}(t+s)\) corresponds to the case where \(\theta_k\) is that last contact (see \Cref{fig:thetas}).
Now, note that by definition we have $\bar \theta_k\in  [\frac12 H,2H]$.
Hence, thanks to \Cref{lem:Zc} and since \(H\leq T =A\tf(\gb)^{-1}\), we can replace $u_{\beta}$ with $u$, at the cost of a multiplicative  constant (using also the regular variation of \(u(\cdot)\)).
Using also that \(\beta \in (\beta_0,2 \beta_0)\), we obtain that \( e^{-A} K(x+t) e^{- t \tf(\gb)} \leq K_{\beta}(x+t) \leq 2 K(x+t) e^{- t \tf(\gb)}\) uniformly for \(x\in [0,s]\) with \(s\leq 2H \leq 2 A \tf(\gb)^{-1}\).
All together, we end up with 
\begin{equation*}
  f_s(t) \asymp  e^{- t \tf(\gb)} \Big( K(t+s)  + \int_0^s u(s-x) K( x+t) \dd x \Big) \,,
\end{equation*}
where $\asymp$ means that the ratio between the two sides is bounded above and below by positive constants (that may depend on \(A\)). 
Then, recalling \(K(t) \sim c t^{-(1+\alpha)}\) and the asymptotic behavior \(u(s)\sim c' s^{\alpha-1}\) from~\eqref{eq:Doney}, one can easily check that for \(s \in [\frac12 H, 2H]\) we have
\[
f_s(t) \asymp  e^{- t \tf(\gb)} \times
 \begin{cases}
    H^{\alpha} \, t^{- (1+\alpha)}   &\ \text{ if } t\ge H ,\\
    H^{-1}  \, (H/t)^{\alpha}  &\ \text{ if } t\le H .
  \end{cases}
\]
Since the r.h.s.\ does not depend on~$s$, we obtain that for any nonnegative~$\varphi_E$
\begin{equation}
  \label{eq:comparemean}
\frac{1}{C}  \tilde Q^{\otimes 2}\big[\varphi_E(\bar \theta'_k) \big]\le  \tilde Q^{\otimes 2}\big[\varphi(\bar \theta'_k) \ | \  \bar \theta_k\big]\le C\tilde Q^{\otimes 2} \big[\varphi_E(\bar \theta'_k)\big] \quad \text{a.s.},
\end{equation}
and the conclusion follows from \eqref{okoki}.
\end{proof}

\begin{proof}[Proof of \Cref{tobeshown}]
Let us first present three key statements which, combined, yield \Cref{tobeshown}. 
First of all, we have 
\begin{equation}\label{wops22}
 \bbP\left(\cT_2-\cT_1\ge \gep T \right) \ge \frac{1}{4 \tilde Q^{\otimes 2}\big[ N_H([\cT_1,\cT_1+T])\big] }- T^{-1}\tilde Q^{\otimes 2}[(\cT_{2}-\cT_{1}) \ind_{\{\cT_{2}-\cT_{1} \leq \gep T\}}].
\end{equation}
Then to estimate each of the term in the r.h.s.\ of \eqref{wops22} we use:
on the one hand that for \(t \in (0,T]\)
\begin{equation}
  \label{asymp}
 c (t/H)^{2\alpha-1} - c' \leq \tilde Q^{\otimes 2}\big[ N_H([\cT_1,\cT_1+t]) \big] \leq c'  (t/H)^{2\alpha-1} 
\end{equation}
(note that the lower bound is trivial when \(t\leq (c'/c)^{\frac{1}{2\alpha-1}} H\)); 
on the other hand that, for $t\in (H,T]$,
\begin{equation}\label{petitup}
 \tilde Q^{\otimes 2}(\cT_{2}-\cT_1 > t)  \leq \frac{C}{ \tilde Q^{\otimes 2}\left[ N_H([\cT_1,\cT_1+t])\right]}.
\end{equation}
Indeed, using \eqref{petitup} and the upper bound in~\eqref{asymp}, we obtain that 
\begin{equation}\label{rufff}
 T^{-1}\tilde Q^{\otimes 2}[(\cT_{2}-\cT_{1}) \ind_{\{\cT_{2}-\cT_{1} \leq \gep T\}}]\le T^{-1} \int^{\gep T}_0  \tilde Q^{\otimes 2}(\cT_{2}-\cT_1 > t) \dd t \le C \gep^{2(1-\alpha)} (T/H)^{1-2\alpha} \,.
\end{equation}
We then conclude the proof of \Cref{tobeshown} by using the lower bound in~\eqref{asymp} (note that the term \(-c'\) can be neglected when \(t=T \geq 4 \gep^{-1} H\), if \(\gep\) is small enough) and \eqref{rufff} inside \eqref{wops22} and taking $\gep$ sufficiently small. 
We next prove \eqref{wops22}, \eqref{asymp} and \eqref{petitup} (in that order).

\smallskip\noindent 
\textit{Proof of \eqref{wops22}.}
Let us set $k= k_T:= 2 \tilde Q^{\otimes 2}[ N_H([\cT_1,\cT_1+T])]$. 
Using Markov's inequality we have
\begin{multline*}
  \frac{1}{2}\le \tilde Q^{\otimes 2}( N_H([\cT_1,\cT_1+T])\leq k)
  = \tilde Q^{\otimes 2}( \cT_k-\cT_1 \ge  T) \\
  \leq   \tilde Q^{\otimes 2} \Big( \exists i\in \lint 1, k-1\rint,\  \cT_{i+1}-\cT_{i} \ge \gep T \Big)
  +  \tilde Q^{\otimes 2} \bigg( \sum_{i=1}^{k} (\cT_{i+1}-\cT_{i}) \ind_{\{\cT_{i+1}-\cT_{i} \leq \gep T\}} \geq  T \bigg) \,.
\end{multline*}
By sub-additivity and translation invariance, the first term is bounded by $k\bbP(\cT_2-\cT_1\geq \gep T)$.
For the second term, using Markov's inequality and translation invariance, we have
\begin{equation}
 \tilde Q^{\otimes 2} \Big( \sum_{i=1}^{k} (\cT_{i+1}-\cT_{i}) \ind_{\{\cT_{i+1}-\cT_{i} \leq \gep T\}} \geq  T \Big) \le k T^{-1} \tilde Q^{\otimes 2}[(\cT_{2}-\cT_{1}) \ind_{\{\cT_{2}-\cT_{1} \leq \gep T\}}].
\end{equation}
Reorganizing the terms and recalling the definition of $k$ we obtain \eqref{wops22}.

\smallskip\noindent 
\textit{Proof of \eqref{asymp}.}
We only deal with the case \(t \in (2H,T]\), since otherwise the bounds are trivial.
In view of the proof of \Cref{preskiid}, we may condition to the value of $\bar \theta_1$ (say $\bar \theta_1=w\in [H/2,2H]$), since it only affects the expectation by a constant factor (see in particular~\eqref{eq:comparemean}).
We now consider the expectation \(\tilde Q^{\otimes 2}[ N_H([\cT_1+H,\cT_1+t]) \mid \bar \theta_1=w ]\), which is slightly different than what appears in~\eqref{asymp}, and write it as an integral.
To help the reader, we provide the role of the variables appearing in the integral.
\begin{align*}
  s&= \tau_{i_k-1}-\tau_{i_1}\    \text{for some $k\ge 2 $} \,, &  s\in (H, t-H]\,,\\
  r&=  \tau_{i_k}-\tau_{i_k-1}\,, \ &r\in (H, 2H\wedge(t-s)]\,, \\
  v&= \Big(\frac{\tau_{i_k}-\tau_{i_k-1}}{2}-  \max \tau'\cap\Big[ \tau_{i_k-1} , \frac{\tau_{i_k}-\tau_{i_k-1}}{2} \Big] \  \Big) &  v\in (0, r/2] \,.
\end{align*}
Using that $u_{\beta}(s) K_{\beta}(r)$ corresponds to the probability density that \(\tau\) has a jump of size $r$ starting at $s+\cT_1$, and that $u_\beta(w+s+r-v) \bar K_{\beta}(v)$ is the probability density that $\tau'$ visits (the first half of) the corresponding  interval with a last contact at distance $v$ from the end of the first-half, we obtain
\begin{multline*}
\tilde Q^{\otimes 2}\big[ N_H([\cT_1+H,\cT_1+t]) \ \big| \ \bar \theta_1=w \big]
  \\ 
  =\int^{t-H}_H \int^{2H}_{H}  \int^{r/2}_0 u_{\beta}(s)K_{\beta}(r)u_{\beta}(w+s+r-v) \bar K_{\beta}(v)\ind_{\{r<(t-s)\}} \dd s \dd r \dd v \,.
\end{multline*}
Then, we notice that ignoring $\ind_{\{r<(t-s)\}}$ in the integrand may result, at worst, in overcounting one unit in $N_H([\cT_1+H,\cT_1+t])$.
Similarly replacing $\cT_1+H$ by $\cT_1$ may only increase the value of $N_H$ by one unit.
This therefore gives that 
\begin{equation}
  \label{compareIw}
  I_{w} -2 \leq \tilde Q^{\otimes 2}[N_H([\cT_1,\cT_1+t]) \mid \bar \theta_1=w] \leq I_{w} \,,
\end{equation}
with 
\[
I_w :=\int^{t-H}_H \int^{2H}_{H}  \int^{r/2}_0 u_{\beta}(s)K_{\beta}(r)u_{\beta}(w+s+r-v) \bar K_{\beta}(v) \dd s \dd r \dd v \,.
\]
Notice that since $s>H$, the quantity $(w+s+r-v)$ is always of order $s$ and therefore there exist constants such that
\[
c u_{\beta}(s) \leq u_\beta(w+s+r-v) \leq c' u_{\beta}(s) \,.
\]
for all choices of $w$, $r$ and $v$ (recall \Cref{lem:Zc}).
Replacing $r/2$ by~$H$ and $t-H$ by $t$ also does not change the order of magnitude (recall we work with \(t\in (2H,T]\)) of the integral so that we obtain
\begin{equation}
\label{theabove}
   I_w \asymp \int^{t}_H \int^{2H}_{H}  \int^{H}_0 u_{\beta}(s)^2 K_{\beta}(r) \bar K_{\beta}(v) \dd v  \dd r \dd s.
\end{equation}
Using \Cref{lem:Zc} and~\eqref{eq:Doney}, and the fact that all variables are smaller than the correlation length $T=A \tf(\gb)^{-1}$, we have $u_{\beta}(s)\asymp s^{\alpha-1}$, $K_{\beta}(r) \asymp r^{-1-\alpha}$ and  
$\bar K_{\beta}(v)\asymp v^{-\alpha}$ so that 
\(
I_w \asymp t^{2\alpha-1} H^{1-2\alpha} .
\)
We then deduce~\eqref{asymp}, recalling~\eqref{compareIw} and the fact that the conditioning affects the expectation only by a constant factor, see~\eqref{eq:comparemean}.

\smallskip\noindent 
\textit{Proof of \eqref{petitup}.}
First of all, let us write
\begin{align}\label{tryp}
  \tilde Q^{\otimes 2}[N_H([\cT_1,\cT_1+t])] = \sum_{k=1}^{\infty} \tilde Q^{\otimes 2}(N_H([\cT_1,\cT_1+t]) \geq k) = \sum_{k=1}^{\infty} \tilde Q^{\otimes 2}( \cT_{k}-\cT_1 \leq t) \,.
\end{align}
Now, by \Cref{preskiid} (cf.\ \eqref{kobeshown}), we obtain that
\begin{equation}\label{tryop}
  \tilde Q^{\otimes 2}( \cT_{k}-\cT_1 \leq t) \leq \Big( 1- C^{-1}\tilde Q^{\otimes 2}( \cT_{2}-\cT_1 > t) \Big)^{k-1} \,.
\end{equation}
Combining \eqref{tryp} and \eqref{tryop} we get that for $t\in (H,T]$, \(\tilde Q^{\otimes 2}[N_H([\cT_1,\cT_1+t])] \leq C/ \tilde Q^{\otimes 2}( \cT_{2}-\cT_1 > t) \).
\end{proof}

\subsection{The marginal case \texorpdfstring{$\gamma=\frac23$ or $d=3$}{}}

Since only the expression for $K(\cdot)$ is important, the case \eqref{SRW} for $d=3$ is equivalent to the case where $\varphi$ converges to a constant and $\gamma=\frac{2}{3}$ (in which case $\mathrm{R}(t)\sim c\, (\log t)^2$). Hence we can focus on the \(\gamma\)-stable case.
The proof relies of the following analogue of Lemma \ref{squiz}.
\begin{lemma}
  \label{squiz2}
  Assume that $K(t) = L(t)t^{-3/2}$ and, for $1\leq H\leq T$, define
  \[
  \psi_H(T) := \int_{H/2}^T \frac{L(H)^2}{L(s)^2} \frac{\dd s}{s} \,.
  \]
  Then, there exist constants \(c,c'>0\) such that, if $v\, \psi_H(T) \leq c$, we have
  \begin{equation*}
    \tilde Q^{\otimes 2} \left[ e^{ v N_H(T) } \right] \leq 1+ c'  v\, \psi_H(T) \,.
  \end{equation*}
\end{lemma}

\begin{proof}[Proof of Proposition \ref{trickyy}]
We set  $r_j = \psi_{2^j}(T)$ for $1\leq j \leq \log_2 T$ and $R_T:= \sum_{i=1}^{\log_2 T} r_j$. 
We set $u=4c_0 \rho$, and we apply the multi-index H\"older inequality, similarly as in~\eqref{multiHolder}: we obtain
\begin{equation}
  \label{multiHolder2}
  \tilde Q^{\otimes 2} \left[ e^{ u |\mathfrak J_{[0,T]}| } \right]\le \prod_{j=1}^{\log_2 T}  \tilde Q^{\otimes 2} \left[
  e^{ u\, \frac{R_T}{r_j} \, N_{2^j}(T)}\right]^{r_j/R_T}\,.
\end{equation}
Now, notice that we have
\begin{equation*}
  % \label{les2rt}
  R_T= \sum_{j=1}^{\log_2 T} \psi_{2^j}(T) = \sum_{j=1}^{\log_2 T} \int_{2^{j-1}}^T \frac{L(2^{j})^2}{L(s)^2} \frac{\dd s}{s} = \int_1^T \frac{1}{L(s)^2} \bigg( \sum_{j=1}^{\lfloor \log_2 s \rfloor } L(2^j)^2 \bigg) \frac{\dd s}{s}  \asymp \mathrm{R}(T)\,,
\end{equation*}
recalling the definition~\eqref{defRR} or \(\mathrm{R}(T)\).  Since we have \(\nu=2\) in \Cref{homener}, we get that \(\tf(\beta) \geq c (\beta-\beta_0)^{3}\), so that \(R_T\le c_A \mathrm R((\beta-\beta_0)^{-3})\).
Hence, according to the assumption of \Cref{prop:multiprop}-\ref{multi-ii} we have  $R_T\le c_A (C_1\rho)^{-1}$ and we therefore may assume that $u R_T\le \delta$, for some $\delta>0$ that can be made arbitrarily small by choosing \(C_1\) large.
We can therefore apply \Cref{squiz2} in~\eqref{multiHolder2} with $v=u R_T/r_j$ (we have $v\le \delta \psi_{2^j}(T)$) and we obtain 
\[
  \tilde Q^{\otimes 2} \left[ e^{ u |\mathfrak J_{[0,T]}| } \right]\le \prod_{j=1}^{\log_2 T}  \Big(1 +c u R_T \Big)^{r_j/R_T}  \leq  e^{ c u R_T} \leq e^{c \delta} \,.
\]
Going back to \Cref{topitop2} and taking $\delta$ sufficiently small, this ends the proof of \Cref{trickyy}.
\end{proof}

% \subsection{\texorpdfstring{Proof of \Cref{squiz2}}{Proof of Lemma}}

\begin{proof}[Proof of \Cref{squiz2}]
First of all, as in the proof of \Cref{squiz}, we fix \(\gep>0\) (small) so that \(1/\gep \in \bbN\) and we may focus on the case \(\gep T \geq 2H\).
We replicate the proof of \Cref{squiz} until~\eqref{boundNH} which we replace with
\begin{equation}
  \label{boundNH2}
  \tilde Q^{\otimes 2}\left[ N_H(\gep T)\ge k \right] \leq \Big(1-  \frac{\gep}{C} \psi_H(T)^{-1} \Big)^{k-1} \,.
\end{equation}
The proof of~\eqref{boundNH2} is similar to that of~\eqref{boundNH}: it relies on the factorization in  \Cref{preskiid} combined with the following estimate.
\begin{lemma}
  \label{tobeshown2}
  Assume that $K(t) = L(t) t^{-3/2}$ and recall that $\psi_H(T) = \int_{H/2}^T \frac{L(H)^2}{L(s)^2} \frac{\dd s}{s}$.
  For $\gep>0$ small enough, we have that 
  \begin{equation*}
    % \label{eq:tobeshown2}
    \tilde Q^{\otimes 2}\left[ \cT_{2}-\cT_1\ge \gep T \right] \ge \frac{\gep}{\psi_H(T)}.
   \end{equation*}
\end{lemma}
\noindent The conclusion of the proof of \Cref{squiz2} then follows from~\eqref{boundNH2} exactly as below~\eqref{boundNH}.
\end{proof}

\begin{proof}[Proof of \Cref{tobeshown2}]
We proceed as in the proof of Lemma \ref{tobeshown}. Note that \eqref{wops22} and \eqref{petitup} are also valid in this case so we just need to an estimate analogous to~\eqref{asymp}. Indeed, we have that there exist some constants $c,c'>0$ such that, for any $t\in(0, T]$
\begin{equation}
    \label{asymp2}
   c \psi_H(t)-c'\leq \tilde Q^{\otimes 2}[N_H([\cT_1,\cT_1+t])] \leq c'  \psi_H(t).
\end{equation}
To see this, note that \eqref{compareIw} still holds, and we can do the same simplification as in~\eqref{theabove} (we can again reduce to the case \(t > 2H\)).
Hence, replacing $u_{\beta}(s)$, $K_{\beta}(r)$ and $\bar K_{\beta}(v)$ by their asymptotic equivalents (using Lemma~\ref{lem:Zc}, \eqref{eq:Doney} and the fact that \(T =A\tf(\gb)^{-1}\)), we obtain that 
\begin{equation*}
   I_w\asymp \int^t_H \frac{\dd s}{L(s)^2 s}\int^{2H}_H L(r)r^{-3/2}\dd r\int^H_0 L(v)v^{-1/2} \dd v \asymp  \int_H^t \frac{L(H)}{L(s)^2}\frac{\dd s}{ s} \,.
\end{equation*}
From~\eqref{compareIw}, this gives that \(\frac{1}{c}\psi_H(T) -2 \leq  \tilde Q^{\otimes 2}[N_H([\cT_1,\cT_1+t]) \mid \bar \theta_1=w] \leq c \psi_H(T)\), which concludes the proof of~\eqref{asymp2} since the conditioning affects the expectation only by a constant factor, see~\eqref{eq:comparemean}.
Let us now conclude the proof of \Cref{tobeshown2} from~\eqref{wops22}, \eqref{petitup} and~\eqref{asymp2}.
We need to control the second term in \eqref{wops22} and we are going to prove that 
\begin{equation}
  \label{truncatedQ}
   T^{-1}\tilde Q^{\otimes 2}[ (\cT_2 -\cT_1) \ind_{\{\cT_2 -\cT_1 \leq \gep T\}}] \le \frac{C'\gep^{1/2}}{\psi_H(T)} \,.
\end{equation}
Then, using \eqref{truncatedQ} for $\gep$ sufficiently small, together with the lower bound in \eqref{asymp2} (note that the term \(-c'\) can be neglected when \(t=T \geq 4 \gep^{-1} H\), if \(\gep\) is small enough) inside \eqref{wops22}, we obtain the conclusion of \Cref{tobeshown2}.
To prove \eqref{truncatedQ} we first combine~\eqref{petitup} together with the lower bound in \eqref{asymp2}. We get that for $t\in (C_0 H,T]$ with \(C_0\) large enough so that \(c \psi_H(t) \geq 2 c'/c\),
\begin{equation*}
  \label{wrongbound2}
  \tilde Q^{\otimes 2}(\cT_2 -\cT_1 >t) \leq \frac{2 C/c}{ \psi_H(t)} \,.
\end{equation*}
Now, writing \(\phi(u) = \int_1^u \frac{\dd s}{L(s)^2 s}\) for simplicity, we can write \(\psi_H(t)= L(H)^2(\phi(t)-\phi(H/2))\), so that (for a different constant \(C\))
\[
\psi_H(T) \tilde Q^{\otimes 2}(\cT_2 -\cT_1 >t) \leq C \frac{\phi(T) -\phi(H/2)}{\phi(t)- \phi(H/2)} \leq C \frac{\phi(T) - \phi(t/2)}{\phi(t) -\phi(t/2)}\le C' L(t)^2 \int_{t/2}^T \frac{\dd s}{L(s)^2 s}.
\]
where we have used that \(x\mapsto \frac{b-x}{a-x}\) is non-decreasing if \(b>a\) for the second inequality and that \(\phi(t) -\phi(t/2) \asymp L(t)^{-2}\) for the last one (together with the definition of \(\phi\) to write \(\phi(T) - \phi(t/2)\) as an integral). 
Hence we have that \(\psi_H(T)\tilde Q^{\otimes 2}[ (\cT_2 -\cT_1) \ind_{\{\cT_2 -\cT_1 \leq \gep T\}}]\) is bounded by
\[
\int^{\gep T}_0 \psi_H(T)\tilde Q^{\otimes 2}\left(\cT_2-\cT_1>t\right)\dd t 
\leq C_0 H\psi_H(T) + C' \int_{C_0 H}^{\gep T} L(t)^2 \int_{t/2}^T \frac{\dd s}{L(s)^2 s} \dd t \,.
\]
The last integral can be rewritten as
\[
\int_{C_0 H/2}^T \frac{\dd s}{ L(s)^2 s} \int_{C_0 H}^{s\wedge \gep T} L(t)^2 \dd t 
\leq c \int_{C_0 H/2}^{\gep T} \dd s + c \gep T  L(\gep T)^2\int_{\gep T}^T \frac{\dd s}{L(s)^2 s} \leq c \gep T  + c' \gep T  \log\big(\tfrac1\gep\big) \,.
\]
where we have used that \(\int_{C_0H}^u L(t)^2 \dd t \leq c L(u)^2 u\) for \(u=s\) or \(u=\gep T\), and then that \(L(s)^2 \asymp L(\gep T)^2\) uniformly for \(s\in [\gep T,T]\). 
Note that the constants do not depend on \(H,T\).
All together, we have 
\[
T^{-1}\psi_H(T) \tilde Q^{\otimes 2}[ (\cT_2 -\cT_1) \ind_{\{\cT_2 -\cT_1 \leq \gep T\}}] \leq \frac{c H}{T} \psi_H(T) +c' \gep \log\big(\tfrac1\gep\big) \,.
\]
To conclude the proof of \eqref{truncatedQ} it only remains to observe that, by definition of \(\psi_H(T)\), and using Potter's bound \cite[Thm.~1.5.6]{BGT89} to bound \(L(s)^2\geq c (s/H)^{-1/2} L(H)^2\) uniformly for \(s\geq H/2\), we have
\begin{equation*}
\frac{H}{T}\psi_H(T) \le  \frac{H}{T} c^{-1} \int^T_{H/2} (s/H)^{1/2} \frac{\dd s}{s} \le c' (H/T)^{1/2}\le c'' \gep^{1/2},
\end{equation*}
where the last inequality comes from the fact that we are considering the case \(\gep T \geq 2 H\).
\end{proof}

\section{Proving \texorpdfstring{\Cref{trickyy}}{the proposition} when disorder is irrelevant: \texorpdfstring{$\gamma\in (\frac23,1)$}{}}
\label{trixx2}

In this section, we prove \Cref{prop:multiprop}-\ref{multi-i}.
Recall that the assumption implies that  $K(t)\sim c t^{-(1+\alpha)}$ with $\alpha = \frac{1-\gamma}{\gamma} \in (0,\frac12)$, see~\eqref{K-reg}.
Our starting point for the proof is the first inequality in \Cref{topitop2}. 
In particular we need to bound $| \cJ^{(1)}_{[0,T]}|+| \cJ^{(2)}_{[0,T]}|$ (recall their definition~\eqref{def:J}), and  the strategy relies on the study the sequence of ``iterated overshoots'' for the renewal process with inter-arrival density $K_{\beta}(\cdot)$, that we describe below.

\subsection{Controlling the second moment via the iterated overshoots}

Let us first define what we intend by ``iterated overshoots''.
We set $\cT_{-1}=\min(\tau_1,\tau'_1)$ and  $\cT_0=\max(\tau_1,\tau'_1)$.
Now for $k\ge 1$, if $\tau_1\ge \tau'_1$ we set
\begin{equation}\begin{split}\label{taugetauprim}
  \cT_{2k-1}&:= \inf\{\tau'_i  : \ \tau'_i\ge  \cT_{2k-2}\},\\
  \cT_{2k}&:= \inf\{\tau_i  : \ \tau_i\ge  \cT_{2k-1}\},\\
  \end{split}
\end{equation}
while if $\tau'_1> \tau_1$ we set $\cT_{2k-1}:= \inf\{\tau_i  : \ \tau'_i\ge  \cT_{2k-2}\}$,  $\cT_{2k}:= \inf\{\tau_i  : \ \tau_i\ge  \cT_{2k-1}\}$.
For \(j\geq 0\), we also define \(\cS_j:= \cT_{j}-\cT_{j-1}\), so that \((\cS_j)_{j\geq 0}\) corresponds to the iterated overshoots of a renewal with inter-arrival density $K_{\beta}$.
We refer to \Cref{fig:overshoot} for an illustration.
\begin{figure}
\centering
\includegraphics[scale=1.2]{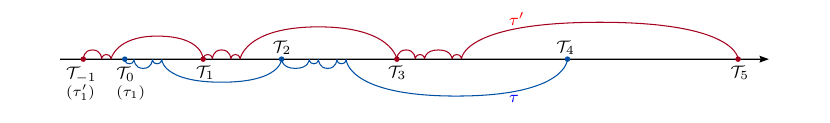}
  \caption{Illustration of the definitions of the iterated overshoots \((\cT_{i})_{i\geq -1}\); here in the case \(\tau_1'<\tau_1\).
  Each overshoot corresponds either to an interval of \(\tau\) visited by \(\tau'\) (counted in $|\cJ^{(1)}_{[0,T]}|$) or to an interval of \(\tau'\) visited by \(\tau\) (counted in $|\cJ^{(2)}_{[0,T]}|$).
  }
  \label{fig:overshoot}
\end{figure}
By construction, and recalling the definition~\eqref{def:J} of \(\cJ^{(1)}_{[0,T]},\cJ^{(2)}_{[0,T]}\), we have
$$| \cJ^{(1)}_{[0,T]}|+| \cJ^{(2)}_{[0,T]}| =\max\{ j  :\cT_j\le T    \}\le \min\{ j\ge 1:  \cS_{j}\ge T \}:= D_T \,.$$
In view of the first inequality in \Cref{topitop2}, we have that
\[
\frac{\bbE\big[  \hat Z_1(Y)^2\big]}{ \bbE\big[  \hat Z_1(Y)\big]^2} -1 
 \le C_A \tilde Q^{\otimes 2} \left[ e^{ c_0 \rho D_T}-1\right] \,.
\]
Therefore, in order to conclude the proof of \Cref{trickyy}, we need to show that for any $\delta>0$, we can find $u(\delta)$ small enough such that, for any $\gb\in (\gb_0,2\gb_0]$,
\begin{equation}
  \label{LaplaceDT}
  \tilde Q^{\otimes 2} \left[ e^{ u D_T}-1\right] \leq \delta \,
\end{equation}
(Recall that \(\tilde Q =\tilde Q_{\beta}\) is the stationary distribution with inter-arrival distribution \(K_{\gb}(\cdot)\)).
The rest of the proof consists in showing~\eqref{LaplaceDT}. It relies on the following lemma, whose proof is postponed to the next subsection, combined with the Markovian structure of \((\cS_{j})_{j\geq 0}\).

\begin{lemma}
  \label{lespetitspoissons}
There exist $\kappa= \frac14 (1-2\alpha)>0$, $\kappa'>0$, $C>0$, and $M$ (large), $\theta$ (small) such that, for all $\beta \in (\gb_0,2\gb_0]$,
 \begin{equation}
  \label{hartz0}
  \tilde Q^{\otimes 2}\left[ (\cS_2/\cS_1)^{-\kappa}\ \big| \ \cS_1=v \right] \leq 
  \begin{cases} 
    e^{-\kappa'} &\quad \text{ for all } v\in [M,\theta T]\,, \\
    C e^{-\kappa'} &\quad \text{ for all } v\in [\theta T,T] \,.
  \end{cases}
 \end{equation}
Additionally, there exists $p(M,\theta)>0$ such that:
\begin{equation}\label{hartz}
  \begin{split}
  \text{If } \cS_1\in[ \theta T,T], \quad & \    \tilde Q^{\otimes 2}\left[ \cS_2\ge T\ | \ \cS_1 \right]\ge p, \\
   \text{With probability $1$},\quad  & \      \tilde Q^{\otimes 2} \left[ \cS_2\ge M\ | \ \cS_1 \right]\ge p.
  \end{split}
\end{equation}
\end{lemma}

Let us now use \Cref{lespetitspoissons} to conclude the proof of~\eqref{LaplaceDT}.
Let $M$ be as in \Cref{lespetitspoissons} and consider $(k_i)_{i\ge 1}$ the ordered sequence of $\{j\ge 1  : \cS_j\ge M\}$, with the convention $k_0=0$.
Thanks to the second line of~\eqref{hartz}, and because of the Markovian structure of \((\cS_j)_{j\geq 0}\), we have that \(k_{i+1}-k_i\), conditionally on $\cG_i:=\sigma\left( \cS_j  : j\le k_i \right)$, is dominated by a geometric random variable with parameter \(p\).
Therefore, for any $i\ge 1$, if $a$ small enough so that $(1-p)e^a<1$, we have
\begin{equation*}
 \bbE\left[  e^{ a (k_{i+1}-k_i)} \ \big| \ \cG_i\right]\le \frac{p e^a}{1-(1-p)e^a} =: e^{\lambda(a)} \,.
\end{equation*}
As a consequence $e^{a k_n - \lambda(a) n} $ is a supermartingale, so that  setting $\bar D_T:= \min\{ i : \cS_{k_i}\ge T \}$ and recalling that \(D_T := \min\{j : \cS_j \geq T \}\), we have
\begin{equation}
  \label{keyu}
  \tilde Q^{\otimes 2}\left[ e^{a D_T -\lambda(a) \bar D_T}  \right]\le 1.
\end{equation}
Hence using Cauchy-Schwarz's inequality and applying~\eqref{keyu} with \(a=2u\), we obtain
\begin{equation*}
  \tilde Q^{\otimes 2}\left[ e^{u D_T}\right]^2 = \tilde Q^{\otimes 2}\left[e^{ \frac12 \lambda(2u) \bar D_T} e^{\frac12 (2u D_T - \lambda(2u) \bar D_T)}  \right]^2 \le  \tilde Q^{\otimes 2}\left[ e^{\lambda(2u) \bar D_T}\right] \,.
 \end{equation*}
Therefore, it remains to show that $\tilde Q^{\otimes 2}[e^{ \lambda(2u)\bar D_T}]$ can be made arbitrarily close to $1$ by taking $u$ small (note that \(\lambda(2u)\downarrow 0\) as \(u\downarrow0\)).
For this, we use \Cref{lespetitspoissons}-\eqref{hartz0}. 
Using again the Markovian structure of \((\cS_j)_{j\geq 0}\), it shows that, setting $\mathcal{X}_n:=\#\{ i\le k_n  : \cS_{k_i}\in [\theta T,T] \}$,
\[
 \prod_{i=0}^{n-1} \left(\frac{\cS_{k_i}}{\cS_{k_{i}+1}}  \right)^{\kappa} e^{\kappa' n} C^{-\mathcal{X}_n}
\]
is a $\cG_n$-supermartingale.
In particular we have
\begin{equation}
  \label{smAl}
  \tilde Q^{\otimes 2}\bigg[ \prod_{i=0}^{\bar D_T-1}\left( \frac{\cS_{k_i}}{\cS_{k_{i}+1}}  \right)^{\kappa} e^{\kappa'\bar D_T} C^{-\mathcal{X}_{\bar D_T}} \bigg]\le 1 \,.
\end{equation}
Let us set $r= \kappa'/(2\lambda(2u))$ so \(\lambda(2u) = \kappa'/2r\).
Using (multi-index) H\"older's inequality with exponents $2r$, $2r$ and $\frac{r}{r-1}$ (taking \(u\) small so that \(r>1\)), we obtain that \(\tilde Q^{\otimes 2}[  e^{\lambda(2u) \bar D_T}]\) is bounded by
\begin{equation*}
    \tilde Q^{\otimes 2}\bigg[ \prod_{i=0}^{\bar D_T-1}\left( \frac{\cS_{k_i}}{\cS_{k_{i}+1}}  \right)^{\kappa} e^{\kappa' \bar D_T} C^{-\mathcal{X}_{\bar D_T}} \bigg]^{\frac{1}{2r}} 
 \tilde Q^{\otimes 2}\bigg[ \prod_{i=0}^{\bar D_T-1}\left( \frac{\cS_{k_i}}{\cS_{k_{i}+1}}  \right)^{-\kappa} \bigg]^{\frac{1}{2r}} 
  \tilde Q^{\otimes 2}\left[ C^{\frac{1}{2(r-1)}\, \mathcal{X}_{\bar D_T}}  \right]^{\frac{r-1}{r}},
\end{equation*}
and the first term is smaller than~$1$ thanks to~\eqref{smAl}.
Now since $\mathcal{X}_{\bar D_T}$ is dominated by a geometric of parameter $p$ (thanks to \eqref{hartz}), the last term verifies
\[
  \tilde Q^{\otimes 2}\left[ C^{\frac{1}{2(r-1)}\, \mathcal{X}_{\bar D_T}}  \right]^{\frac{r-1}{r}} \leq \bigg(\frac{p C^{\frac{1}{2(r-1)}}}{1- (1-p) C^{\frac{1}{2(r-1)}}}\bigg)^{\frac{r}{r-1}} \,,
\]
which can be made arbitrarily close to~$1$ by choosing $r$ large (that is $u$ small).
For the remaining term, note that we have
\begin{equation*}
  \prod_{i=0}^{\bar D_T-1}\left( \frac{\cS_{k_{i}+1}} {\cS_{k_i}} \right)^{\kappa} \le \left(\frac{\cS_{D_T}}{\cS_0}\right)^{\kappa}\le \frac{1}{2}\Big( (\cS_{ D_T}/T)^{2\kappa} +(T/\cS_0)^{2\kappa} \Big)  \,,
\end{equation*}
using the standard inequality \(ab \leq \frac12 (a^2+b^2)\).
Thus, using also Jensen's inequality, we have (recall that \(2\kappa = \frac12 -\alpha < \frac12 (1-\alpha) <1\))
\[
\tilde Q^{\otimes 2}\bigg[ \prod_{i=0}^{\bar D_T-1}\left( \frac{\cS_{k_i}}{\cS_{k_{i}+1}}  \right)^{-\kappa} \bigg]^{\frac{1}{2r}}  \leq 2^{-\frac{1}{2r}} \Big( \tilde Q^{\otimes 2}\big[(\cS_{ D_T}/T) \big]^{2\kappa} + \tilde Q^{\otimes 2}\big[(
T/\cS_{0})^{(1-\alpha)/2} \big]^{4\kappa/(1-\alpha)}\Big)^{\frac{1}{2r}} \,.
\]
To conclude we therefore need to show that \(\tilde Q^{\otimes 2}(\cS_{ D_T}/T) \leq C_A\) and that \(\tilde Q^{\otimes 2} ((T/\cS_0)^{(1-\alpha)/2}) \leq C_A\), for some constant $C_A$ that may depend on \(A\) but not on \(\beta\in (\beta_0,2\beta_0)\).
Then, the upper bound can be made arbitrarily close to~$1$ by choosing $r$ large (that is $u$ small).

\smallskip
Let us start with $\cS_{D_T}$ and show that \(\tilde Q^{\otimes 2}(\cS_{ D_T})\leq C_A T\).
Assume that $D_T$ is even (this entails no loss of generality) and consider the index \(j\) such that $\cT_{D_T}=\tau_j$ and let $\vartheta= \cT_{D_T-1}-\tau_{j-1}$ (referring to \Cref{fig:overshoot} may help the reader to get a visual grasp on the notation).
Then, conditionally on $\cT_{D_T}=\tau_j$ and on $\vartheta=x$, the distribution of $\cS_{D_T}$ is that of $\tau_1-\tau_0-x$ conditioned to be larger than $T$. 
We thus have, for \(t \geq T\),
\[
\tilde Q^{\otimes 2} \left( S_{D_T}\ge t \ | \  \cT_{D_T}=\tau_j, \vartheta=x\right)=\frac{\bar K_{\beta}(x+t)}{\bar K_{\beta}(x+T)} =  e^{ -\tf(\beta)(t-T)} \frac{\int_0^{\infty} e^{-\tf(\gb) u } K( x + t +s) \dd s }{\int_0^{\infty} e^{-\tf(\gb) u } K( x + T +s) \dd s} \leq e^{ -\tf(\beta)(t-T)} \,,
\]
where the last inequality follows from the fact that $K(t)$ is non-increasing, see \Cref{lem:unimod}. 
This entails that $\tilde Q^{\otimes 2} (\cS_{ D_T} )\le T+\tf(\beta)^{-1} \leq (1+A^{-1}) T$, as desired. 
As far as \(\cS_0\) is concerned, let us show that $\tilde Q^{\otimes 2} ( (\cS_{0})^{-\eta})\le C_A T^{-\eta}$ for \(\eta := \frac12 (1-\alpha)\).
Recalling that under \(\tilde Q\) the first renewal \(\tau_1\) has density \(\tf'(\gb) \bar K_{\gb}(s)\), the density distribution of \(\cS_0 = |\tau_1-\tau_1'|\) is given by 
\[
f_{\cS_0}(t)=  2 \tf'(\gb)^2 \int^{\infty}_0 \bar K_{\beta}(s)\bar K_{\beta}(s+t)\dd s \le  4 \tf'(\gb) K(t)  \,,
\]
where we have used that \(\bar K_{\gb}(s+t) = \frac{\gb}{\gb_0} K(s+t) e^{- (s+t) \tf(\gb)} \leq 2 K(t)\) by monotonicity of \(K(\cdot)\) (see \Cref{lem:unimod}) and then fact that \(\tf'(\gb) \bar K_{\gb}(s)\) is a density.
Hence integrating up to $T$ (and bounding the rest trivially) we obtain that
\begin{equation*}
 \tilde Q^{\otimes 2} (\cS_{0}^{-\eta} )\le 4 \tf'(\gb) \int^{T}_0  s^{-\eta}\bar K(s) \dd s + T^{-\eta} \,.
\end{equation*}
Now, combining \Cref{lem:Zc} and~\eqref{eq:Doney}, we get that \(\tf'(\gb) \leq c \tf(\gb)^{-(1-\alpha)} = c A^{1-\alpha}  T^{\alpha-1}\).
For \(\eta < 1-\alpha\), we also have that \(\int^{T}_0  s^{-\eta}\bar K(s) \dd s \leq c T^{1-\alpha-\eta}\).
All together, this shows that $\tilde Q^{\otimes 2} (\cS_{0}^{-\eta} )\le C_A T^{-\eta}$, as desired.
\qed

\subsection{Proof of \texorpdfstring{\Cref{lespetitspoissons}}{the lemma}}

We show~\eqref{hartz0} and~\eqref{hartz} separately.
Note that we can assume that $\beta/\beta_0$ is small (thus \(T\) large) as  otherwise there is nothing to prove.

\smallskip\noindent
\textit{Proof of~\eqref{hartz0}. }
Let us denote $f^{(v)} = f_{\gb}^{(v)}$ the conditional density of $\cS_2/\cS_1$ given $\cS_1 =v$, under $\tilde Q^{\otimes 2}$. 
We have the following expression
\[
 f^{(v)}(s):= v\Big(K_{\beta}(v(s+1))+ \int^v_0 u_{\beta}(y)K_{\beta}(v(s+1)- y) \dd y\Big) \,.
\]
Given~$\gep_1$, we can take $\gb$ sufficiently close to $\gb_0$ such that $K_{\gb}(t) = \frac{\gb}{\gb_0}e^{-t\tf(\gb)} K(t) \leq (1+\gep_1)^{1/2} K(t)$ for all \(t\geq 0\) and we can also choose $\theta$ sufficiently small so that for all $t\leq \theta T$ we have  $u_{\gb}(t) \leq (1+\gep_1)^{1/2} u(t)$, see \Cref{lem:Zc}.
Therefore, for $v\in[M,\theta T]$ we have
\begin{equation*}
  f^{(v)}(s)\le (1+\gep_1)v \Big( K(v(s+1))+\int^v_0 u(y)K(v(s+1)- y) \dd y \Big)=: (1+\gep_1) f^{(v)}_0(s) \,,
\end{equation*}
so that 
\[
\tilde Q^{\otimes 2}[ (\cS_2/\cS_1)^{-\kappa}] \leq (1+\gep_1) \int_0^{\infty} s^{-\kappa} f^{(v)}_0(s) \dd s \,.
\]
For $v\in [\theta T,T]$ the above is valid with an additional multiplicative constant that depends on $A$.
We then need to show that the r.h.s.\ is smaller than $e^{-\kappa'}$ for $v\geq M$, provided that \(M\) is sufficiently large.
Since $\gep_1$ is arbitrary, this is equivalent to showing that for $\kappa = \frac14 (1-2\alpha)>0$ we have
\begin{equation}
  \label{Eratio}
 \lim_{v\to \infty} \int_0^{+\infty} s^{-\kappa} f^{(v)}_0(s) \dd s  <1 \,.
\end{equation}
Note that the integral on the left hand side of~\eqref{Eratio} has two parts, corresponding to the two terms in~$f_0^{(v)}$. 
The first part is
\[
\int_0^{+\infty} s^{-\kappa} v K(v(s+1)) \dd s \leq c v^{-\alpha} \int_0^{+\infty} s^{-\kappa} (1+s)^{-(1+\alpha)}\dd s
\leq c' v^{-\alpha} \,,
\]
recalling also that \(K(t) \sim c t^{-(1+\alpha)}\) so that $K(v(s+1)) \leq c (v(1+s))^{-(1+\alpha)}$ for any $x\geq 0$ and $v\geq 1$.
This can therefore be made arbitrarily small by choosing \(v\) large.
The second part of the integral in the left-hand side of~\eqref{Eratio} is
\[
\int_0^{\infty} s^{-\kappa} \Big( v \int_0^v u(y) K(v(s+1) -y) \dd y  \Big) \dd s =   \int^{\infty}_0 \int_0^1 s^{-\kappa}  v^2 u( v x) K\big( v(1+s-x) \big)  \dd x \dd s .
\]
Now, from \eqref{eq:Doney}, we obtain that for any $x\in (0,1)$
\begin{equation*}
 \lim_{v\to \infty} v^2 u( v x) K\big( v(1+s-x) \big) = \frac{\alpha \sin(\pi \alpha)}{\pi}(1+s-x)^{-(1+\alpha)} x^{\alpha-1} \,.
\end{equation*}
Using that $u(t)\le C t^{1-\alpha} s$ and  $K(t)\le C t^{-(1+\alpha)}$, we also get that for any fixed \(v>1\) 
\begin{equation*}
 v^2 u( v x) K\big( v(1+s-x) \big) \le C^2  (1+s-x)^{-(1+\alpha)} x^{\alpha-1} \,.
\end{equation*}
Hence, using dominated convergence, we have 
\begin{equation*}
  \begin{split}
  \lim_{v\to \infty}\int^{\infty}_0 \int_0^1 x^{-\kappa}  v^2 K\big( (1+x-s)v \big) u(sv) \dd s&  = \frac{\alpha \sin(\pi \alpha)}{\pi}\int^{\infty}_0\int^1_0 s^{-\kappa} (1+s-x)^{-(1+\alpha)} x^{\alpha-1} \dd x \dd s  \\
  & =  \frac{ \sin(\pi \alpha)}{\pi}\int^{\infty}_0 \frac{s^{-\alpha-\kappa}}{1+s} \dd s \,.
  \end{split}
\end{equation*}
For the second inequality, we have computed the integral in $x$ by performing a change of variable \(t= \frac{1+s-x}{x}\) (\(x=\frac{1+s}{1+t}\)), so it is equal to \(\frac{1}{1+s}\int_{s}^{+\infty}  t^{-(1+\alpha)} \dd t = \frac{s^{-\alpha}}{\alpha(1+s)} \).
We also have that for \(u\in (0,1)\), by a change of variable \(t= \frac{s}{1+s}\),
\[
\int_0^{\infty} \frac{s^{-u}}{1+s} \dd s = \int_0^1 t^{-u} (1-t)^{u-1}  \dd t = \Gamma(u) \Gamma(1-u) = \frac{\pi}{\sin(\pi u)}\,.
\]
Hence, we get that
\[
\lim_{v\to \infty}\int^{\infty}_0 \int_0^1 x^{-\kappa}  v^2 K\big( (1+x-s)v \big) u(sv) \dd s = \frac{\sin(\pi \alpha)}{\sin(\pi (\alpha+\kappa))}  <1 \,,
\]
the last inequality being valid because \(\alpha+\kappa  <\frac12\) (recall \(\kappa = \frac14(1-2\alpha)\) and \(\alpha<\frac12\)).
This concludes the proof of~\eqref{Eratio} and thus of~\eqref{hartz0}.

\smallskip\noindent
\textit{Proof of~\eqref{hartz}. }
For the first inequality, on the event $\cS_1 \in [\theta T, T]$, we simply write 
\[
  \tilde Q^{\otimes 2}\left[ \cS_2\ge T\ | \ \cS_1 \right] \geq  \int_0^{\cS_1} u_{\gb}(t) \bar K_{\gb} ( T + \cS_1 -t ) \dd t \ge   \bar K_{\gb} ( 2T ) \int_0^{\theta T} u_{\gb}(t) \dd t \,.
\]
Now, we have that $\bar K_{\beta}(2T) \geq \frac{\gb}{\gb_0} \int_{2T}^{3T} e^{- s\tf(\gb)} K(s) \dd s \geq e^{-3A} TK(3T)$, since \(K(\cdot)\) is non-increasing (see \Cref{lem:unimod}): this gives that \(\bar K_{\beta}(2T) \geq c_A T^{-\alpha}\).
On the other hand, \Cref{lem:Zc} together with~\eqref{eq:Doney} shows that $u_{\gb}(t) \geq c_A t^{\alpha-1}$ uniformly for $t \leq T =A\theta \tf(\gb)^{-1}$, so that $\int_0^{\theta T} u_{\gb}(s) \dd s \geq c_A' (\theta T)^{\alpha}$.
All together, we obtain that
\[
\tilde Q^{\otimes 2}\left[ \cS_2\ge T\ | \ \cS_1 \right] \geq p = p(\theta,A) >0 \,.
\]
For the second inequality in~\eqref{hartz}, we also easily have that
\[
\begin{split}
\tilde Q^{\otimes 2} \left[ \cS_2\ge M\ | \ \cS_1 \right]
&  = \bar K_{\gb}(M+\cS_1)+  \int_0^{\cS_1} u_{\gb}(y) \bar K_{\gb}(M+\cS_1-y) \dd y  \\
& \ge p_M \bigg( \bar K_{\gb}(\cS_1) + \int_0^{\cS_1} u_{\gb}(y) \bar K_{\gb}(\cS_1-y) \dd y \bigg) =p_M \,,
\end{split}
\]
where we have set
\[
p_{M} := \inf_{\beta \in (\beta_0,2\beta_0]}\inf_{s\geq 0}  \frac{K_{\gb}(M+s)}{K_{\gb}(s)} = e^{-M \tf(2\beta_0)} \inf_{s\geq 0}  \frac{K(M+s)}{K(s)} \,,
\] 
so that \(\bar K_{\gb}(M+s) \geq p_M \bar K_{\gb}(s)\) for all \(s\geq 0\).
This concludes the proof of~\eqref{hartz}.
\qed

\medskip

\noindent {\bf Acknowledgments:}
H.L.\ acknowledges the support of a productivity grand from CNQq and of a CNE grant from FAPERj.
Q.B. acknowledges the support of Institut Universitaire de France and ANR Local (ANR-22-CE40-0012-02).

\appendix

\section{Existence of the free energy: proof of \texorpdfstring{\Cref{freeenergy}}{the proposition}}
\label{app:freeen}

We only need to prove the existence and coincidence of the \(\bbP\)-a.s.\ and \(L^1(\bbP)\) limits $\lim_{T\to\infty}\frac{1}{T}\log Z^{Y,\mathrm c}_{\beta,T}$ and $\lim_{T\to\infty}\frac{1}{T}\log Z^{Y}_{\beta,T}$ under the assumption \eqref{JPP}.
Indeed, pointers to adequate references have been given for all other items in the statement. 
We only deal with the convergence of $\frac{1}{T}\log Z^{Y,\mathrm c}_{\beta,T}$ to $\tf(\rho,\beta)$, since the equality of the limits follows from the following relation between the free and the constrained partition function
\begin{equation}
\label{freecontraint}
  Z^{Y}_{\beta,T} = 1+ \int^T_0   Z^{Y,\cons}_{\beta,t} \dd t  \,,
\end{equation}
recall for instance~\eqref{goodexpress}.
Note that since \(Z^{Y}_{\beta,T} \geq 1\) we readily have that \(\tf(\rho,\beta) \geq 0\).
To proceed we first use superadditivity to establish the existence of the limit in the constrained case, for integer valued $T$, then check that that variation of $\log Z^{Y,\mathrm c}_{\beta,T}$ between two consecutive integers cannot by very large and then prove that $\tf(\rho,\beta)\ge 0$ by a simple comparison argument. Recalling the definition \eqref{partsegos}
we have
\begin{equation}\label{supermu}
 Z^{Y,\mathrm c}_{\beta,T}\ge  \beta \bE\left[e^{\beta H^Y_T(X)}\ind_{\{X_S=Y_S; X_T=Y_T\}} \right]   = \beta^{-1} Z^{Y,\mathrm c}_{\beta,S} Z^{Y,\mathrm c}_{\beta,[S,T]}.
\end{equation}
By applying the subadditive ergodic Theorem to the sequence $\big(\beta^{-1}Z^{Y,\mathrm c}_{\beta,n}\big)_{n\in \bbN}$, we obtain the existence and coincidence of the following limits
\begin{equation}\label{existos}
\tf(\rho,\beta):=\lim_{n\to \infty} \frac{1}{n}\bbE\left[ \log  Z^{Y,\mathrm c}_{\beta,n}\right]=  \lim_{T\to \infty} \frac{1}{n} \log  Z^{Y,\mathrm c}_{\beta,n}.
\end{equation}
To complete the statement, we only need to check that  $\log  Z^{Y,\mathrm c}_{\beta,T}$ cannot vary a lot within an interval $[n,n+1]$. 
For $T\geq 1$, we let $\mathrm{T}_- :=\lfloor T\rfloor-1$ and $\mathrm{T}_+ := \lceil T \rceil +1$, so from~\eqref{supermu} we have
\begin{equation*}
Z^{Y,\mathrm c}_{\beta,\mathrm{T}_-} Z^{Y,\mathrm c}_{\beta,[\mathrm{T}_-,T]}\le Z^{Y,\mathrm c}_{\beta,T}\le \frac{Z^{Y,\mathrm c}_{\beta,\mathrm{T}_+}}{ Z^{Y,\mathrm c}_{\beta,[T,\mathrm{T}_+]}} \, .
\end{equation*}
We obtain that the convergence \eqref{existos} extends to $T\to \infty$, $T\in \bbR_+$ by showing that 
\begin{equation}
\lim_{T\to \infty}\frac{1}{T}\log  Z^{Y,\mathrm c}_{\beta,[\mathrm{T}_-,T]}=\lim_{T\to \infty}\frac{1}{T}\log  Z^{Y,\mathrm c}_{\beta,[T,\mathrm{T}_+]}=0 \,,
\end{equation}
almost surely and in $L^1$. 
We focus on the first case, the second one being identical.
We have 
\begin{equation*}
\bP \left( X_{T-T_{-}}= Y_{T}- Y_{T_{-}}\right) \le  Z^{Y,\mathrm c}_{\beta,[\mathrm{T}_-,T]}   \le e^{\beta (T-T_{-})}\bP \left( X_{T-T_{-}}= Y_{T}- Y_{T_{-}}\right).
\end{equation*}
Since  $T-\mathrm{T}_-\in [1,2)$ it is sufficient to show that the probability 
$\bP(X_{T-\mathrm{T}_-} = Y_T-Y_{\mathrm{T}_-})$ does not decay exponentially.
Note that the probability is bounded from below by the probability of reaching the endpoint in one jump which is equal to 
$(T-T_-)e^{-(T-T_-)} J(Y_{T}- Y_{T_{-}})$. Since $J$ has polynomial decay it is sufficient to show that $|Y_{T}- Y_{T_{-}}|$ cannot grow exponentially in $T$. Given the tail behavior of~$J$, this is an easy consequence of the Borel-Cantelli Lemma.
\qed

\section{Homogeneous pinning model of a continuous-time renewal process}
\label{app:homogeneous}

The goal of this section is to prove \Cref{lem:Zc}. But let us first give a few results about continuous-time renewal processes.
Let $K(t)$ be some positive density on $\bbR_+$ and let $\tau = (\tau_0=0,\tau_1,\ldots)$ be a renewal process with inter-arrival distribution $K(\cdot)$; we may also interpret $\tau$ as a subset of $\bbR_+$. 
Recall that we define the conditional distribution $\bQ(  \cdot \mid t\in \tau )=\lim_{\gep\to 0 }\bQ(  \cdot \mid \tau\cap[t,t+\gep]\ne \emptyset )$. 
We then have the following lemma, analogous to~\cite[Lem.~A.2]{GLT10} (with an identical proof).
\begin{lemma}
\label{lem:removecond}
Assume that the inter-arrival density is $K(t)= L(t) t^{-(1+\alpha)}$ for some $\alpha>0$ and some slowly varying function $L(\cdot)$.
Then, there exists a constant $C$ such that, for any $t>1$ and any non-negative measurable function~$f$
\[
\frac{\bQ\big[ f\big( \tau\cap [0,t] \big) \mid 2t\in \tau \big]}{\bQ\big[ f\big( \tau\cap [0,t] \big)  \big]}\leq \max_{r\in [0,t]}\left( \frac{ \int^{2t}_t K(s-r)u(2t-s)\dd s + K(2t-r)  }{u(2t)\int^{\infty}_t K(s-r)\dd s}\right) \le C  \,.
\]
\end{lemma}

\begin{proof}[Proof of \Cref{lem:Zc}]
We only consider the case where $\gb/\gb_0$ is close to $1$, in which case $\tf(\gb)$ is small,  \textit{i.e.}\ smaller than some fixed $\delta>0$.
The case where $\tf(\gb)>\delta$ is easier since then $u(1/\tf(\gb))$ is bounded from below by a constant.
Let us start with the case $t\leq \frac{1}{\tf(\gb)}$.
Recalling \eqref{zubet} and~\eqref{asanexpect} we have
\begin{equation*}
 u_{\beta}(t) = z_{\gb,t}^{\cons} e^{-\tf(\beta)t} = u(t)e^{-\tf(\beta)t}\bQ\Big[ (\gb/\gb_0)^{|\tau\cap (0,t]|}  \; \Big| \; t\in \tau \Big] \ge u(t)e^{-\tf(\beta)t}.
\end{equation*}
This yields the desired lower-bound in \eqref{cawre}-\eqref{cawre2}.
For the upper-bound, we apply Cauchy-Schwarz inequality and use the symmetry of $\tau\cap (0,t)$ around $t/2$ under \(\bQ (\cdot \mid t\in \tau)\) to obtain
\begin{equation*}
\bQ\Big[ (\gb/\gb_0)^{|\tau\cap (0,t]|}  \; \Big| \; t\in \tau \Big] \leq  (\beta/\beta_0)\bQ\Big[ (\gb/\gb_0)^{2|\tau\cap (0,\frac12 t]|}  \; \Big| \; t\in \tau \Big].
\end{equation*}
Now, using \Cref{lem:removecond} (assume that \(t\geq 2\) otherwise the proof is simpler), we get that there exists a constant $C$ such that
\begin{equation*}
 \bQ\Big[ (\gb/\gb_0)^{|\tau\cap (0,t]|}  \; \Big| \; t\in \tau \Big]  \le (\beta/\beta_0) \left(1+  C \bQ\Big[ (\gb/\gb_0)^{2|\tau\cap (0,\frac12 t]|} -1  \Big]\right).
\end{equation*}
To conclude the proof of the upper-bound in \eqref{cawre}-\eqref{cawre2} we need to show that 
\begin{equation*}
 \bQ\Big[ (\gb/\gb_0)^{2|\tau\cap (0,\frac12 t]|}\Big]  \le \begin{cases}
  1+C   \quad &\text{ if } t\le \frac{1}{\tf(\beta)},\\
  1+\gep/C  \quad &\text{ if } t\le \frac{\eta}{\tf(\beta)}.
  \end{cases}
\end{equation*}
For more readability we set $t'=t/2$ and $\beta'=\beta^2/\beta_0$.
Using a decomposition of the expectation up to the last contact point, denoting \(\overline{K}(t) :=\int_t^{\infty} K(s) \dd s\) and similarly for \(\overline K_{\beta}\), we have
\begin{equation*}
  \bQ\Big[ (\gb'/\gb_0)^{|\tau\cap (0,t']|}\Big]= \overline K(t') + \int^{t'}_0 u_{\beta'}(s) e^{s \tf(\gb')} \overline K(t-s) \dd s  .
 \end{equation*}
We show just below that we have
\begin{equation}
  \label{supKKb}
 \sup_{r\in [0,t']} \frac{\bar K(r)}{e^{r\tf(\beta')}\bar K_{\beta'}(r)}\le 1+ C(t) \quad \text{ with } C(t) =
 \begin{cases}
  C   \quad &\text{ if } t\le \frac{1}{\tf(\beta)},\\
  \gep/C  \quad &\text{ if } t\le \frac{\eta}{\tf(\beta)} \,.
  \end{cases}
\end{equation}
Therefore, one gets
\[
\bQ\Big[ (\gb'/\gb_0)^{|\tau\cap (0,t']|}\Big] \leq (1+C(t)) e^{t \tf(\beta')} \Big( \overline K_{\beta'}(t')+ \int^{t'}_0 u_{\beta'}(s)\overline K_{\beta'}(t-s) \dd s \Big) \leq 1+C(t) \,,
\]
the last identity coming again from a last exit decomposition.
To prove~\eqref{supKKb} we start with the identity
\begin{equation*}
  % \label{lintegrule}
e^{r\tf(\beta')}\bar K_{\beta'}(r)=\frac{\beta'}{\beta_0}\int^{\infty}_0 e^{-v\tf(\beta')}K(r+v) \dd v \geq e^{-c b} \int_0^{b/\tf(\gb)} K(r+v) \dd v \,,
\end{equation*}
where \(b= 1\) or \(b=\eta^{1/2}\) depending on whether we consider the case \(t\leq \tf(\gb)^{-1}\) or \(t\leq \eta\tf(\gb)^{-1}\), and where we have used that $\tf(\beta')$ and $\tf(\beta)$ are of the same order (cf.\ \Cref{homener}). 
We therefore get that 
\[
\frac{\bar K(r)}{e^{r\tf(\beta')}\bar K_{\beta'}(r)} \leq e^{c b} \frac{\bar K(r)}{\bar K(r) - \bar K(b\tf(\gb)^{-1})} \,.
\]
In the case where \(t\leq \tf(\gb)^{-1}\), taking \(b=1\) we get that \(\bar K(r)- \bar K(\tf(\gb)^{-1}) \geq c \bar K(r)\) uniformly for \(r\leq t'= \frac12 \tf(\gb)^{-1}\). 
This concludes the first part of~\eqref{supKKb}.
In the case where \(t\leq \eta/\tf(\gb)\), taking \(b=\eta^{1/2}\), we have that \(\bar K(r)- \bar K(b\tf(\gb)^{-1}) \geq (1-c \eta^{\alpha/2}) \bar K(r)\) uniformly for \(r\leq t' = \frac12\eta^{1/2} b \tf(\gb)^{-1}\). 
This concludes the second part of~\eqref{supKKb}, with \(\gep = C(e^{c \eta^{1/2}}(1-c\eta^{\alpha/2})^{-1} -1)\)

\smallskip

Now let us treat the case where $t\ge \frac{1}{\tf(\beta)}$. 
We set 
\[
m_{\beta}:=\min_{s \geq \tf(\beta)^{-1}} u_{\beta}(s)
 \quad \text{ and } \quad   
M_{\beta}:=\max_{s \geq \tf(\beta)^{-1}} u_{\beta}(s) \,,
\]
and we need to show that 
\begin{equation}\label{dylls}
m_{\beta}\ge c_1 u(\tf(\beta)^{-1}) \quad \text{ and } \quad  M_{\beta}\le c_2 u(\tf(\beta)^{-1}).
\end{equation}
Using the renewal equation \( u_{\beta}(t)= K_{\beta}(t)+\int^{t}_0 u_{\beta}(s) K_{\beta}(t-s)\dd s\), our result for $s\le 1/\tf(\beta)$ and the definition of $m_{\gb}$, we obtain that for any \(t\geq 1/\tf(\gb)\),
\begin{equation*}
 u_{\beta}(t)\ge   K_{\beta}(t)+c\int^{1/\tf(\beta)}_0 u(s) K_{\beta}(t-s)  \dd s + m_{\beta} \int^{t-1/\tf(\beta)}_{0} K_{\beta}(r)\dd r.
\end{equation*}
Using this at a time when $u_{\beta}(t)=m_{\beta}$ (these are the only times where \eqref{dylls} needs to be proved), we obtain that 
\begin{equation*}
 \Big( \int_{t-1/\tf(\beta)}^{\infty} K_{\beta}(r)\dd r \Big)  m_{\beta}(t) \ge  K_{\beta}(t)+c\int^{1/\tf(\beta)}_0 K_{\beta}(t-s) u(s) \dd s.
\end{equation*}
Reorganizing the terms and writing \(K_{\beta}(r) =  \frac{\gb}{\gb_0} e^{-r \tf(\gb)} K(r)\), we obtain that
\begin{equation*}
 m_{\beta} \ge  \frac{K(t)+c\int^{1/\tf(\beta)}_0 e^{s\tf(\beta)} K(t-s) u(s) \dd s}{\int_{-1/\tf(\beta)}^{+\infty} e^{-s\tf(\beta)} K(t+s)\dd s}.
\end{equation*}
It is now an easy exercise to see that for \(t \geq 1/\tf(\gb)\), 
\[
\int_{-1/\tf(\beta)}^{+\infty} e^{-s\tf(\beta)} K(t+s)\dd s \asymp  \tf(\gb)^{-1} K(t)
\quad \text{and} \quad
\int^{1/\tf(\beta)}_0 e^{s\tf(\beta)} K(t-s) u(s) \dd s \asymp  \tf(\gb)^{-1} K(t) u(\tf(\gb)^{-1})  \,,
\]
using also~\eqref{eq:Doney} for the second one.
We therefore get that 
\[
 m_{\beta} \ge c \big( \tf(\gb)^{-1}+ u(\tf(\gb)^{-1})  \big)  \geq c' u(\tf(\gb)^{-1})\,,
\]
using again \eqref{eq:Doney} for the last inequality.
This concludes the proof of \eqref{dylls} for $m_{\beta}$.
The proof for $M_{\beta}$ is analogous.
\end{proof}

\section{Proof of Proposition~\ref{prop:LLT}.}
\label{app:LLT}

Note that \eqref{LLT} is simply the local limit theorem, see e.g.\ \cite[Chapter~9]{GK68}, so we focus on \eqref{krixx}.
The proof follows from the strategy to prove the local limit theorem. We provide the details in the case of the $\gamma$-stable walk (the SRW case is only simpler).
In fact, for pedagogical purposes, we start by proving~\eqref{implicit} since the proof of~\eqref{krixx} will build upon it.

\begin{proof}[Proof of~\eqref{implicit}]
We start by writing,
\begin{equation}
  \label{FourierK}
K(t):= \beta_0 \P(W_t=0) =  \frac{\beta_0}{2\pi} \int_{[-\pi,\pi]} \E\big[ e^{i\xi W_t}\big] \dd \xi=  \frac{\beta_0}{2\pi} \int_{[-\pi,\pi]} e^{-t q(\xi)} \dd \xi \,, 
\end{equation}
where the last identity follows simply by first conditioning on the (Poissonian) number of jumps, and \(1-q(\xi)\) is the characteristic function of \(J(\cdot)\), \textit{i.e.}\
\[
q(\xi) = 1- \sum_{x\in \bbZ} J(x) e^{i\xi x}  = \sum_{x\in \bbZ}  J(x) (1-\cos(\xi x)) \,,
\]
the last identity coming from the symmetry of \(J(\cdot)\).
Now, using a change of variable, we obtain that
\begin{equation}\label{identity}
 1 =  \frac{\beta_0}{2\pi} \int_{\bbR} e^{-t q(u K(t))} \ind_{\{|u| K(t) \le  \pi\} } \dd u.
\end{equation}
Note that \(q(\xi) = q(-\xi)\geq 0\).
Note also that, in view of~\eqref{JPP}, by a Riemann approximation we have
\begin{equation}
  \label{q:asymp}
  q(\xi)\stackrel{\xi\to 0}{\sim} c_{\gamma} \varphi(1/|\xi|) |\xi|^{\gamma}  \quad \text{ with } \quad  c_{\gamma}:=\int_{\bbR}  |s|^{-(1+\gamma)}(1-\cos s) \dd s. 
\end{equation}
We let \(\ell \in [0,\infty]\) be any subsequential limit of \(tq(K(t))\) and let us work along a subsequence \((t_n)_{n\geq 0}\) such that $\lim_{n\to \infty} t_nq(K(t_n))=\ell$.
Using the fact that $\varphi$ is slowly varying, we get that for any fixed \(u\in \mathbb{R}\), \(q(u K(t)) \sim q(K(t)) |u|^{\gamma}\) as \(t\to\infty\), so that
\begin{equation}
  \label{limittq}
  \lim_{n\to\infty} t_n q(u K(t_n)) = \ell |u|^{\gamma} \,.
\end{equation}
Hence, applying Fatou's lemma, we get from \eqref{identity} that \( \int_{\bbR} e^{- \ell |u|^{\gamma}} \dd y \leq \frac{2\pi}{\beta_0}\), so in particular \(\ell>0\).
Additionally, since  \(q: [0,\pi] \to [0,\infty)\) is regularly varying at \(\xi=0\) with index \(\gamma\), we get by Potter's bound that there is a constant \(c>0\) such that \(q(\xi)  \ge c  q( u^{-1} \xi ) u^{\gamma/2}\) for all \(\xi\in [0,\pi] \) and \(u>1\). 
Applying this to \(\xi=u K(t_n)\) gives that \(t_n q(u K(t_n))\geq c t_n q(K(t_n)) u^{\gamma}\) for all \(u\in [1,\pi K(t_n)^{-1}]\). A similar inequality holds for \(-u\), by symmetry.
Bounding \(t_n q(uK(t_n))\geq 0\) outside this interval and using also that \(\ell =\lim_{n\to\infty} t_n q(K(t_n)) >0\), we conclude that
\begin{equation}
  \label{domtq}
 t_n q(u K(t_n)) \ge  c' |u|^{\gamma/2}   \ind_{\{ |u| \in [1,\pi K(t_n)^{-1}] \}} \,,
\end{equation}
for some \(c'>0\).
Now, combining~\eqref{limittq} and~\eqref{domtq}, we get by dominated convergence (and simple domination) that
\[
\begin{split}
  \lim_{\gep \downarrow 0} \lim_{n\to\infty} &\int_{\bbR} e^{-t_n q(u K(t_n))} \ind_{\{|u|\in [0,\gep^{-1}]\}} \dd u =  \int_{\bbR} e^{-\ell |u|^{\gamma}} \dd y \\
\text{ and } \quad 
  \lim_{\gep \downarrow 0} \limsup_{n\to\infty} & \int_{\mathbb{R}} e^{-t_n q(u K(t_n))} \ind_{\{|u| \in [\gep^{-1},\pi K(t_n)^{-1}]\}} \dd u =0\,,
\end{split}
\]
so in particular, we deduce from~\eqref{identity} that
\begin{equation}
  \label{def:ell}
 1 = \frac{\beta_0}{2\pi}\int_{\bbR} e^{- \ell |u|^{\gamma}} \dd y \,.
\end{equation}
This shows that all subsequential limits are equal, so that \(\lim_{t\to\infty} t q(K(t)) =\ell\), where \(\ell\) is defined by the expression~\eqref{def:ell}. Recalling~\eqref{q:asymp}, this shows~\eqref{implicit}.
\end{proof}

\begin{proof}[Proof of~\eqref{krixx}]
Notice that we just have to prove that $t K'(t)/K(t)$ converges, the value of the limit being then prescribed by the behavior of $K(t)$ at infinity. 
Using the expression~\eqref{FourierK} for $K(t)$, we have 
\begin{equation*}
 K'(t):= \frac{\beta_0}{2\pi} \int_{[-\pi,\pi]} q(\xi) e^{-t q(\xi)} \dd \xi = \frac{\beta_0}{2\pi} K(t) \int_{\bbR} q(u K(t)) e^{-t q(u K(t))} \ind_{\{|u|K(t) \le \pi\} } \dd u\,.
\end{equation*}
We therefore get that 
\[
\frac{t K'(t)}{K(t)} = \frac{\beta_0}{2\pi} \int_{\bbR} t q(u K(t)) e^{-t q(u K(t))} \ind_{\{|u|K(t) \le \pi\} } \dd u \,.
\]
Then, using~\eqref{limittq}-\eqref{domtq}, we get by dominated convergence (similarly as for~\eqref{def:ell}) that 
\[
\lim_{t\to \infty} \frac{t K'(t)}{K(t)}=  \frac{\beta_0}{2\pi}\int_{\mathbb{R}} \ell |u|^{\gamma} e^{- \ell |u|^{\gamma}}  \dd u\,,
\]
which concludes the proof.
\end{proof}

\begin{rem}
  The same proof in fact shows that for any \(j\geq 1\), the \(j\)-th derivative of \(K(t)\) verifies
  \[
  \lim_{t\to \infty} \frac{t^j K^{(j)}(t)}{K(t)}=  \frac{\beta_0}{2\pi}\int_{\mathbb{R}} (\ell |u|^{\gamma})^j e^{- \ell |u|^{\gamma}}  \dd u\,,
  \]
  so that \(K^{(j)}(t) \stackrel{t\to \infty}{\sim} (-1)^j (1+\alpha) \cdots (j+\alpha) \,t^{-j} K(t)\) .
\end{rem}

\bibliographystyle{abbrv}
\bibliography{biblio.bib}

\end{document}